\documentclass[10pt]{amsart}
\usepackage{setspace}
\usepackage{amsmath, amsfonts, amsthm, amstext, amssymb, xspace}
\usepackage{eucal}
\usepackage{amssymb,latexsym}
\usepackage{xcolor}
\usepackage[shortlabels]{enumitem}
\pdfoutput=1 
\usepackage{microtype}
\usepackage{longtable,booktabs}
\usepackage{tikz-cd}
\usepackage{hyperref}
\usepackage[capitalise, nameinlink]{cleveref}
\hypersetup{
    colorlinks,
    linkcolor={red!50!black},
    citecolor={blue!50!black},
    urlcolor={blue!80!black}
}

\theoremstyle{definition}
\newtheorem{thm}{Theorem}[subsection]
\newtheorem{lemma}[thm]{Lemma}
\newtheorem{corollary}[thm]{Corollary}
\newtheorem{proposition}[thm]{Proposition}

\newtheorem{remark}[thm]{Remark}
\newtheorem{definition}[thm]{Definition}
\newtheorem{example}[thm]{Example}
\newtheorem{question}[thm]{Question}

\newtheorem{thmx}{Theorem}

\def\a{\mathbb{A}}
\def\c{\mathbb{C}}
\def\e{\mathbb{E}}
\def\f{\mathbb{F}}
\def\g{\mathbb{G}}
\def\m{\mathbb{M}}
\def\n{\mathbb{N}}
\def\p{\mathbb{P}}
\def\q{\mathbb{Q}}
\def\r{\mathbb{R}}
\def\z{\mathbb{Z}}

\def\ca{\mathcal{A}}

\def\Sp{\mathcal{S}\mathrm{p}}
\def\LMod{\mathrm{L}\mathcal{M}\mathrm{od}}
\newcommand{\CAlg}{\mathcal{C}\mathrm{Alg}}

\newcommand{\Spc}{\mathcal{S}\mathrm{pc}}

\def\Spec{\operatorname{Spec}}
\def\Ext{\operatorname{Ext}}
\def\Hom{\operatorname{Hom}}
\def\colim{\operatorname{colim}}
\def\ker{\operatorname{ker}}
\def\coker{\operatorname{coker}}
\def\cl{{\text{cl}}}
\newcommand{\dcl}{{\text{dcl}}}
\newcommand{\thetatilde}{\widetilde{\theta}}
\def\Sq{\operatorname{Sq}}
\def\gr{\operatorname{gr}}
\def\unr{\operatorname{un}}
\def\im{\operatorname{Im}}
\def\big{\operatorname{big}}
\newcommand{\hyp}{\hbox{{-}}}
\newcommand{\tors}{{\text{tors}}}
\newcommand{\Fbar}{\overline{F}}
\newcommand{\alphabar}{\overline{\alpha}}
\newcommand{\Be}{\operatorname{Be}}
\newcommand{\cell}{\text{cell}}
\newcommand{\Fun}{\text{Fun}}
\newcommand\floor[1]{\lfloor#1\rfloor}
\newcommand\ceil[1]{\lceil#1\rceil}
\newcommand{\cotimes}{\mathbin{\widehat{\otimes}}}

\newcommand\xqed[1]{%
  \leavevmode\unskip\penalty9999 \hbox{}\nobreak\hfill
  \quad\hbox{#1}}
\newcommand\tqed{\xqed{$\triangleleft$}}

\author{William Balderrama}\address{University of Virginia}\email{eqr8nm@virginia.edu}
\author{Dominic Leon Culver}\address{Max-Planck-Institut f\"ur Mathematik}\email{dominic.culver@gmail.com}
\author{J.D. Quigley}\address{Cornell University}\email{jdq27@cornell.edu}

\title[The motivic lambda algebra and motivic Hopf invariant one problem]{The motivic lambda algebra \\ and motivic Hopf invariant one problem}

\begin{document}

\begin{abstract}
We investigate forms of the Hopf invariant one problem in motivic homotopy theory over arbitrary base fields of characteristic not equal to $2$. Maps of Hopf invariant one classically arise from unital products on spheres, and one consequence of our work is a classification of motivic spheres represented by smooth schemes admitting a unital product.

The classical Hopf invariant one problem was resolved by Adams, following his introduction of the Adams spectral sequence. We introduce the motivic lambda algebra as a tool to carry out systematic computations in the motivic Adams spectral sequence. Using this, we compute the $E_2$-page of the $\mathbb{R}$-motivic Adams spectral sequence in filtrations $f \leq 3$. This universal case gives information over arbitrary base fields.

We then study the $1$-line of the motivic Adams spectral sequence. We produce differentials $d_2(h_{a+1}) = (h_0+\rho h_1)h_a^2$ over arbitrary base fields, which are motivic analogues of Adams' classical differentials. Unlike the classical case, the story does not end here, as the motivic $1$-line is significantly richer than the classical $1$-line. We determine all permanent cycles on the $\mathbb{R}$-motivic $1$-line, and explicitly compute differentials in the universal cases of the prime fields $\mathbb{F}_q$ and $\mathbb{Q}$, as well as $\mathbb{Q}_p$ and $\mathbb{R}$.
\end{abstract}

\maketitle

\tableofcontents

\section{Introduction}

Motivic homotopy theory is a homotopy theory for algebraic varieties, developed by Morel and Voevodsky in the 1990s \cite{MV99}. Since its conception and subsequent use by Voevodsky to resolve the Milnor \cite{Voe03} and Bloch--Kato \cite{Voe11} conjectures, an immense amount of work has gone into the theory, with applications to algebraic geometry, algebraic number theory, and algebraic topology.

Motivic stable homotopy theory is the home of $\a^1$-invariants on algebraic varieties, such as algebraic $K$-theory, motivic cohomology, and algebraic cobordism. The universal such invariants are motivic stable homotopy groups, and as such the internal structure of the motivic stable homotopy groups of spheres reflects the broad-scale structure of the motivic stable homotopy category. These motivic stable stems encode deep geometric and number-theoretic information; for example, Morel \cite{Mor04} showed that the Milnor--Witt $K$-theory of a field appears in its stable stems, and R{\"o}ndigs--Spitzweck--{\O}stv{\ae}r \cite{RSO19, RSO21} have identified motivic stable stems in low Milnor-Witt stem in terms of variants of Milnor $K$-theory, Hermitian $K$-theory, and motivic cohomology.

Motivic homotopy theory was originally developed to apply ideas and tools from homotopy theory to problems in algebraic geometry and algebraic $K$-theory. Information now flows the other way as well. After $p$-completion, $\c$-motivic stable stems capture information about classical stable stems that is not seen using classical techniques. This has led to the highly successful program of Gheorghe--Isaksen--Wang--Xu \cite{Isa19, IWX20, GWX21}, yielding groundbreaking advances in computations of classical stable homotopy groups of spheres. A similar program using $\r$-motivic stable stems to capture information about $C_2$-equivariant stable stems has also developed \cite{BHS20, BI20, DI16a, DI17, GI20, BGI21}. More recently, work of Bachmann--Kong--Wang--Xu \cite{BKWX22} has related $F$-motivic stable homotopy theory over a general field $F$ to classical complex cobordism.

All of this has motivated a swath of explicit computations of motivic stable stems over particular base fields $F$. We refer the reader to \cite{IO19} for a general survey, but mention the following $2$-primary computations:

\begin{enumerate}
\item[{$F = \c$}] Dugger--Isaksen \cite{DI10} computed the $\c$-motivic stable stems through the $36$-stem, and Isaksen \cite{Isa19} and Isaksen--Wang--Xu \cite{IWX20} pushed these computations out to the 90 stem.

\item[{$F = \r$}] Dugger--Isaksen \cite{DI16a} computed the first $4$ Milnor--Witt stems over $\r$, and Belmont--Isaksen \cite{BI20} expanded on this to compute the first $11$ Milnor--Witt stems over $\r$.
 
\item[{$F = \f_q$}] Wilson \cite{Wil16} and Wilson--{\O}stv{\ae}r \cite{WO17} computed the motivic stable homotopy groups of finite fields in motivic weight zero through topological dimension $18$.
\end{enumerate}

There are still many mysteries contained in the motivic stable stems. All of the above computations were enabled by the \textit{motivic Adams spectral sequence}, originally introduced by Morel \cite{Mor99} and further developed by Dugger--Isaksen \cite{DI10}. This is a motivic analogue of the classical Adams spectral sequence, which was developed by Adams \cite{Ada58, Ada60} to resolve the \textit{Hopf invariant one problem}. Adams used this spectral sequence to prove that the only elements of Hopf invariant one in the classical stable stems $\pi_\ast^\cl$ are the classical Hopf maps $\eta_\cl\in \pi_1^\cl$, $\nu_\cl\in \pi_3^\cl$, and $\sigma_\cl\in \pi_7^\cl$. This theorem has a number of implications, including a classification of which spheres can be made into $H$-spaces, which spheres are parallelisable, which $2$-dimensional modules over the Steenrod algebra can be realized by cell complexes, which dimensions a finite-dimensional real division algebra can have, and more.

This paper is concerned with topics surrounding motivic analogues of the classical Hopf invariant one problem. There is an element $\eta$ in the motivic stable stems, represented by the canonical map $\eta\colon \a^2\setminus\{0\}\rightarrow\p^1$, which refines the classical complex Hopf map $\eta_\cl$. Hopkins and Morel \cite{Mor04} showed that $\eta$ is one of the generators of the Milnor--Witt $K$-theory of the base field. This motivic $\eta$ behaves quite differently from the classical Hopf map; most famously, $\eta$ is not nilpotent, and is generally not $2$-torsion. Because $\eta$ is not nilpotent, one may consider the $\eta$-inverted stable stems $\pi_{\ast,\ast}[\eta^{-1}]$. These are closely related to Witt $K$-theory \cite{Bac20, BH20}, and have been the subject of thorough investigation \cite{AM17, GI15, GI16, OR20, Wil18}.

Using the theory of Cayley--Dickson algebras, Dugger--Isaksen \cite{DI13} have shown that the classical quaternionic and octionic Hopf maps $\nu_\cl$ and $\sigma_\cl$ also admit geometric refinements to motivic classes $\nu$ and $\sigma$. All of these motivic Hopf maps $\eta$, $\nu$, and $\sigma$ are maps of Hopf invariant one, but unlike classically, they are not the only such maps. For example, the classical stable stems include into the weight $0$ portion of the motivic stable stems, and $\eta_\cl$, $\nu_\cl$, and $\sigma_\cl$ give rise to distinct examples of maps of Hopf invariant one in the motivic setting. If we reformulate the condition of a map $\alpha$ having nontrivial Hopf invariant as asking that the homology of the $2$-cell complex with attaching map $\alpha$ is not split as a module over the motivic Steenrod algebra, then the situation becomes even richer: for example, $\sigma_\cl^2$ admits an $\r$-motivic refinement to a map of nontrivial Hopf invariant in this sense, closely related to the nonexistent Hopf map coming next in the sequence $\eta$, $\nu$, $\sigma$.

All of this motivates the present work, the purpose of which is three-fold:

\begin{enumerate}
	\item To analyze the motivic Hopf invariant one problem and deduce geometric consequences;
	\item To advance our understanding of motivic stable stems over general base fields;
	\item To introduce the \emph{motivic lambda algebra}, a new tool for motivic computations.
	
\end{enumerate}

As mentioned above, Adams resolved the Hopf invariant one problem by introducing and studying the Adams spectral sequence. Morel \cite{Mor99} and Dugger--Isaksen \cite{DI10} have already introduced the \textit{$F$-motivic Adams spectral sequence}, which takes the form
\[
E_2^{\ast,\ast,\ast} = \Ext_{\ca^F}^{\ast,\ast,\ast}(\m^F,\m^F)\Rightarrow \pi_{\ast,\ast}^F.
\]
Here, $\ca^F$ is the $F$-motivic Steenrod algebra \cite{Voe03} \cite{HKO17}, which acts on $\m^F$, the mod $2$ motivic cohomology of $\Spec(F)$. This spectral sequence converges to $\pi_{\ast,\ast}^F$, the homotopy groups of the $(2,\eta)$-completed $F$-motivic sphere \cite{HKO11a, KW19a}. Implicit is the assumption that $2$ is invertible in $F$.

In this paper, we bring the motivic Adams spectral sequence back to its classical roots, using it to study the motivic Hopf invariant one problem. We do not follow Adams' original approach. Instead, at least in broad outline, we follow J.S.P.\ Wang's approach \cite{Wan67}, which proceeded by first gaining a good understanding of the $E_2$-page of the Adams spectral sequence. Importing this approach to motivic homotopy theory requires analyzing the $E_2$-page of the motivic Adams spectral sequence over general base fields in ranges beyond what is known by previous techniques.

To carry out this analysis, we bring another tool from classical stable homotopy theory into the motivic context: the \textit{lambda algebra}. The classical lambda algebra $\Lambda^{\cl}$ is a certain differential graded algebra, originally constructed by Bousfield--Curtis--Kan--Quillen--Rector--Schlesinger \cite{BCK+66}, whose homology recovers the $E_2$-page of the Adams spectral sequence. The classical lambda algebra is now a standard member of the homotopy theorist's toolbox, and we cannot hope to list all of its applications, but the following are a handful:

\begin{enumerate}
\item Wang's computation of $E_2$-page of the Adams spectral sequence through the $3$-line, and subsequent simplified resolution of the Hopf invariant one problem \cite{Wan67};
\item Some of the first automated computations of the $E_2$-page of the Adams spectral sequence, including products and Massey products \cite{Tan85, Tan93, Tan94, CGMM87};
\item The construction of Brown--Gitler spectra \cite{BG73}, which played an important role in analyzing the $bo$-resolution \cite{Mah81, Shi84}, the proof of the Immersion Conjecture \cite{Coh85}, and more \cite{Mah77, Goe99, HK99};
\item The algebraic Atiyah--Hirzebruch spectral sequence for $\r P^\infty$ \cite{WX16}, used as input to Wang--Xu's proof of the nonexistence of exotic smooth structures on the $61$-sphere \cite{WX17};
\item The only complete computations of the $4$- and $5$-lines of the Adams $E_2$-term \cite{Che11, Lin08}.
\end{enumerate}

We expect that the motivic lambda algebra will likewise become a useful member of the motivic homotopy theorist's toolbox. This paper focuses in particular on developing the lambda algebra and applying this to the motivic Hopf invariant one problem. We consider both the unstable problem, with applications to $H$-space structures on motivic spheres, and the stable problem, which is concerned with the $1$-line of the motivic Adams spectral sequence. The motivic situation is substantially richer than the classical situation, and requires us to develop a number of new techniques for motivic computations across general base fields.

Adams' resolution of the classical Hopf invariant one problem asserted the existence of differentials $d_2(h_{a+1}) = h_0 h_a^2$ in the Adams spectral sequence. There are classes $h_a$ in the $F$-motivic Adams spectral sequence for any field $F$, corresponding to the motivic Hopf maps discussed above for $a\leq 3$. Using Betti realization, it is possible to lift Adams' differentials to the $\c$-motivic Adams spectral sequence. It follows that if $F$ admits a complex embedding, then $h_{a+1}$ must support a nontrivial differential for $a\geq 3$. However, this is insufficient to determine the precise target of the differential, as well as to determine what happens over other base fields, particularly fields of positive characteristic. The techniques we develop in this paper are geared towards resolving this sort of issue. We use these to obtain a number of new results; let us give the following here, as it is the most pleasant to state.

\begin{thmx}[\cref{thm:nohopf}]\label{thmx:nohopf}
For an arbitrary base field $F$ of characteristic not equal to $2$, there are differentials of the form
\[
d_2(h_{a+1}) = (h_0+\rho h_1)h_a^2
\]
in the $F$-motivic Adams spectral sequence, which are nonzero for $a\geq 3$.
\tqed
\end{thmx}

It is worth making a couple remarks to distinguish this from the classical result.

\begin{remark}
Classically, there is at most one possible nontrivial target for a $d_2$-differential on $h_{a+1}$. As suggested by the target in \cref{thmx:nohopf}, the motivic situation is more complicated. For example, when $F=\r$, we show that if $a\geq 4$ then the group of potential values of $d_2(h_{a+1})$ is given by $\f_2\{h_0 h_a^2, \rho h_1 h_a^2\}$. The general picture is similar, only where there may be additional interference coming from the mod $2$ Milnor $K$-theory of $F$. This computation requires new techniques for computing the cohomology of the motivic Steenrod algebra, which is much richer than the analogous classical computation.
\tqed
\end{remark}

\begin{remark}
Even once we have carried out the algebraic work of identifying potential values of $d_2(h_{a+1})$, the classical proof does not directly generalize to yield \cref{thmx:nohopf}. In spirit, our proof follows Wang's classical inductive proof \cite{Wan67}. The base case of Wang's induction is the differential $d_2(h_4) = h_0 h_3^2$, which follows easily from graded commutativity of stable stems. By contrast, our base case must include the differential $d_2(h_5) = (h_0 + \rho h_1)h_4^2$. Over $\r$, this differential may be deduced by combining complex and real Betti realization, but a completely different argument is required to obtain the differential for other fields. To obtain this differential over other base fields, we use a certain motivic Hasse principle to reduce to considering fields with simple mod $2$ Milnor $K$-theory, then analyze how the classical Kervaire class $\theta_4$ appears in the motivic stable stems.
\tqed
\end{remark}

\begin{remark}
There is another elegant proof of the classical Adams differential $d_2(h_{a+1}) = h_0^2 h_a$ due to Bruner, which makes use of power operations in the Adams spectral sequence \cite[Chapter VI, Corollary 1.5]{BMMS86}. Tilson has explored analogues of Bruner's results in the $\r$-motivic setting \cite{Til17}, but so far these methods have only succeeded in determining the $\r$-motivic differential $d_2(h_{a+1})$ for $a\leq 3$.
\tqed
\end{remark}

\subsection{Brief overview}

Now let us give a very brief overview of what we do in this paper, before giving a more thorough summary in \cref{ssec:summary} just below. This paper has three main parts. These parts are not independent, but none rely on the hardest aspects of the others.

The first part is purely algebraic, and is the most computationally intensive. In \cref{Sec:Structure}, we introduce the \textit{$F$-motivic lambda algebra} (\cref{thmx:lambda}), and in \cref{Sec:ExtLow} we use the $\r$-motivic lambda algebra to compute $\Ext_\r$ in filtrations $f\leq 3$ (\cref{thmx:extr}). The result is quite complicated, with $8$ infinite families of multiplicative generators and numerous relations between these. As we explain in \cref{ssec:extf}, this gives information about $\Ext_F$ for any base field $F$ once the mod $2$ Milnor $K$-theory of $F$ is known.

The second part is shorter, and does not rely on the above computation. In \cref{sec:hi1}, after some preliminaries in \cref{sec:preliminaries}, we consider the motivic analogue of the Hopf invariant one problem in its classical \textit{unstable} formulation, concerning unstable $2$-cell complexes with specified cup product, as well as concerning geometric applications such as to $H$-space structures on motivic spheres. Our analysis proceeds by a novel reduction to the classical case and other known results, by first formulating a certain motivic Lefschetz principle (\cref{prop:lefschetz}), then using this to build unstable ``Betti realization'' functors over arbitrary algebraically closed fields (\cref{prop:bettiun}). One consequence of this analysis is a complete classification of motivic spheres which are represented by smooth schemes admitting a unital product (\cref{thmx:products}).

The third part is our main homotopical contribution. In \cref{sec:hopf}, we give a detailed study of the $1$-line of the $F$-motivic Adams spectral sequence. This work has a direct geometric interpretation: permanent cycles on the $1$-line of the motivic Adams spectral sequence classify how the motivic Steenrod algebra can act on the cohomology of a motivic $2$-cell complex. This section does not rely on the full strength of our computation of $\Ext_\r$, and should be understandable by the reader familiar with prior work on the $\r$-motivic Adams spectral sequence. The main theorems in this section are \cref{thmx:nohopf} above, together with much more detailed information about the $1$-line of the $F$-motivic Adams spectral sequence for the particular fields $F = \r$, $F = \f_q$ with $q$ an odd prime-power, $F = \q_p$ with $p$ any prime, and $F = \q$ (\cref{thmx:1line}). As this includes all the prime fields, these computations give information that applies to an arbitrary base field. When $F = \r$, we completely determine all permanent cycles on the $1$-line by comparison with a computation in Borel $C_2$-equivariant homotopy theory (\cref{thmx:rp}); both the equivariant computation and the method of comparison are of independent interest.

\subsection{Summary of results}\label{ssec:summary}

We now summarize our work in more detail. We begin with our introduction of the motivic lambda algebra. The nature of the classical lambda algebra $\Lambda^\cl$ \cite{BCK+66} was greatly clarified by Priddy \cite{Pri70}, who introduced the notion of a \textit{Koszul algebra}, and showed that $\Lambda^\cl$ is the Koszul complex of the classical Steenrod algebra. We carry out the motivic analogue of this, producing the following.

\begin{thmx}[\cref{SS:StructureSummary}]\label{thmx:lambda}
There is a differential graded algebra $\Lambda^F$, the \textit{$F$-motivic lambda algebra}, with the following properties.
\begin{enumerate}
\item $\Lambda^F$ may be described explicitly in terms of generators, relations, and monomial basis.
\item There is a surjective and multiplicative quasiisomorphism $C(\ca^F)\rightarrow\Lambda^F$ from the cobar complex of the $F$-motivic Steenrod algebra to $\Lambda^F$. In particular, there is an isomorphism
\[
H_\ast \Lambda^F\cong \Ext_F^\ast
\]
compatible with all products and Massey products. Moreover, the squaring operation $\Sq^0\colon \Ext_F^\ast\rightarrow\Ext_F^\ast$ lifts to a map $\theta\colon\Lambda^F\rightarrow\Lambda^F$ of differential graded algebras.
\item $\Lambda^F$ generalizes the classical lambda algebra, in the sense that if $F$ is algebraically closed then $\Lambda^F[\tau^{-1}] = \Lambda^\cl[\tau^{\pm 1}]$. In particular, it is considerably smaller than $C(\ca^F)$.
\tqed
\end{enumerate}
\end{thmx}

Here, we have abbreviated $\Ext_{\ca^F}^{\ast,\ast,\ast}(\m^F,\m^F)$ to $\Ext_F^\ast$, where the single index refers to filtration, or homological degree, i.e.\ $\Ext_F^f = H^f(\ca^F)$.

\begin{remark}\label{rmk:koszulsubtleties}
Several subtleties arise in the construction and identification of the motivic lambda algebra. We note two interesting points here:

\begin{enumerate}

\item Priddy's notion of Koszul algebra in \cite{Pri70} is not general enough for our situation: $\ca^F$ is generally not augmented as an $\m^F$-algebra, and $\m^F$ is generally not central in $\ca^F$. This forces us to consider a more general notion of a Koszul algebra, as well as to find new arguments to prove that $\ca^F$ is Koszul in this more general sense.

\item As readers familiar with the motivic Adem relations might suspect, the elements $\tau$ and $\rho$ of $\m^F$ appear in the relations defining the motivic lambda algebra, as well as in its differential and the endomorphism $\theta$ lifting $\Sq^0$. Determining these formulas precisely is delicate and requires some careful arguments. 
\tqed

\end{enumerate}
\end{remark}

\begin{remark}
As indicated above, we construct the $F$-motivic lambda algebra as a certain Koszul complex for the $F$-motivic Steenrod algebra. The Koszul story produces other complexes as well: for any $\ca^F$-modules $M$ and $N$ with $M$ projective over $\m^F$, there are complexes $\Lambda^F(M,N)$ serving as small models of the cobar complex computing $\Ext_{\ca^F}(M,N)$. An amusing special case of this produces a lambda algebra $\Lambda^{C_2}$ for the $C_2$-equivariant Steenrod algebra (\cref{Rmk:GeneralCobar}).
\tqed
\end{remark}

We use the motivic lambda algebra to study $\Ext_F$ in low filtration. Before diving into our more extensive computations, we illustrate the structure of $\Lambda^F$ with some simple examples in \cref{ssec:simpleexamples}, showing how it may be used to give easy rederivations of some well-known low-dimensional relations in $\Ext_F$. We then carry out our main algebraic computation in \cref{Sec:ExtLow}, where we prove the following. Note that $\Ext_\r^0 = \f_2[\rho]$.

\begin{thmx}\label{thmx:extr}
The structure of $\Ext_\r$ in filtrations $f\leq 3$ is as described in \cref{Sec:ExtLow}; in particular, the $\f_2[\rho]$-module structure is described in \cref{thm:extadd}, including a description of multiplicative generators and the action of $\Sq^0$, and the majority of the multiplicative structure is described in \cref{thm:extmult}.
\tqed
\end{thmx}

Here, we are justified in focusing on $\Ext_\r$ as it is, in a certain precise sense, the universal case (see \cref{rmk:runiv}). We explain in \cref{ssec:extf} how to pass from information about $\Ext_\r$ to information about $\Ext_F$ for other base fields $F$.

\begin{example}[{Part (1) of \cref{thm:extadd}}]\label{ex:generators}
The computation of $\Ext_\r^{\leq 3}$ is much more involved than the corresponding classical computation, and the result is much richer. We refer the reader to \cref{Sec:ExtLow} for the full statements, but illustrate this here with the following sample. Classically, $\Ext_\cl^{\leq 3}$ is generated as an algebra by the classes $h_a$ and $c_a$ for $a\geq 0$. By contrast, a minimal multiplicative generating set of $\Ext_\r^{\leq 3}$ as an $\f_2[\rho]$-algebra is given by the classes in the following table:
\begin{longtable}{ll}
\toprule
Multiplicative generator & $\rho$-torsion exponent \\
\midrule \endhead
\bottomrule \endfoot
$h_{a+1}$ & $\infty$ \\
$c_{a+1}$ & $\infty$ \\
$\tau^{\floor{2^{a-1}(4n+1)}}h_a$ & $2^a$ \\
$\tau^{2^a(8n+1)}h_{a+2}^2$ & $2^{a+1}\cdot 3$ \\
$\tau^{\floor{2^{a-1}(2(16n+1)+1)}}h_{a+3}^2h_a$ & $2^a\cdot 13$ \\
$\tau^{2^a(4(4n+1)+1)}h_{a+3}^3$ & $2^a\cdot 7$ \\
$\tau^{\floor{2^{a-1}(16n+1)}}c_a$ & $2^a\cdot 7$ \\
$\tau^{2^{a+1}(8n+1)}c_{a+1}$ & $2^{a+2}\cdot 3$ \\
$\tau^{\floor{2^{a-1}(2(4n+1)+1)}}c_a$ & $2^a\cdot 3$.
\end{longtable}
Here, $a,n\geq 0$, and the $\rho$-torsion exponent of a class $\alpha$ is the minimal $r$ for which $\rho^r \alpha = 0$; the classes $h_{a+1}$ and $c_{a+1}$ are $\rho$-torsion-free. Note that all of the classes listed are named for their image in $\Ext_\c$, and are not themselves products.
\tqed
\end{example}

\begin{example}
Observe that the multiplicative generators $h_a$ and $c_a$ of $\Ext_\cl^{\leq 3}$ appear, with a shift, as $\rho$-torsion-free classes in $\Ext_\r$. This is a general phenomenon: in \cite[Theorem 4.1]{DI16a}, Dugger--Isaksen produce an isomorphism $\Ext_\r[\rho^{-1}]\simeq\Ext_{\dcl}[\rho^{\pm 1}]$; here, $\Ext_{\dcl} = \Ext_\cl$, only given a motivic grading so that $\Ext_\cl^{s,f} = \Ext_{\dcl}^{2s+f,f,s+f}$. As we discuss in \cref{ssec:doubling}, this in fact refines to a splitting $\Ext_\r\cong \Ext_\dcl[\rho]\oplus \Ext_\r^{\rho\hyp\tors}$, where $\Ext_\r^{\rho\hyp\tors}\subset\Ext_\r$ is the subgroup of $\rho$-torsion; moreover, this splitting is modeled by a multiplicatively split inclusion $\thetatilde\colon \Lambda^{\dcl}\rightarrow\Lambda^\r$. The general shape of $\Ext_\r$ forced by this may be illustrated by the following description of the $1$-line:
\begin{equation}\label{eq:extr1}
\Ext_\r^1 = \f_2[\rho]\{h_a:a\geq 1\}\oplus\bigoplus_{a\geq 0}\f_2[\rho]/(\rho^{2^a})\{\tau^{\floor{2^{a-1}(4n+1)}}h_a:n\geq 0\}.
\end{equation}
\tqed
\end{example}

As $\Ext_\cl^{\leq 3}$ is entirely understood by Wang's computation \cite{Wan67}, the hard work of \cref{thmx:extr} is in computing the $\rho$-torsion subgroup of $\Ext_\r^{\leq 3}$. This is the most computationally intensive part of the paper, and proceeds by a direct case analysis of monomials in $\Lambda^\r$ in low filtration. In the end, we find that $\Ext_\r^{\leq 3}$ carries the multiplicative generators listed in \cref{ex:generators}, and that there are many exotic relations between these generators. Our computation describes all of this.

With the algebraic computation of \cref{thmx:extr} in place, we turn to more homotopical topics, namely those surrounding the \textit{Hopf invariant one problem}. There are (at least) \textit{two} good motivic analogues of the Hopf invariant one problem: one which is unstable, concerning the construction of unstable $2$-cell complexes with nontrivial cup product structure, and one which is stable, concerning the construction of stable $2$-cell complexes with nontrivial $\ca^F$-module structure. As we recall in \cref{ssec:1lineext}, understanding the latter question is equivalent to understanding the $1$-line of the $F$-motivic Adams spectral sequence; we get to this in \cref{sec:hopf}, which we will discuss further below.

It is the former unstable formulation which has more direct geometric applications. For example, following work of Dugger--Isaksen \cite{DI13} on the Hopf construction in motivic homotopy theory, it is directly tied up with the question of which unstable motivic spheres $S^{a,b}$ admit $H$-space structures (see \cref{prop:hopfconstruction}). Here, $S^{a,b}$ is the motivic sphere which is $\a^1$-homotopy equivalent to $\Sigma^{a-b} \g_m^{\wedge b}$. We discuss this unstable formulation in \cref{sec:hi1}, which is independent of our other calculations. One pleasant consequence of this story is the following.

\begin{thmx}[\cref{thm:hspace}]\label{thmx:products}
The only motivic spheres which are represented by smooth $F$-schemes admitting a unital product are $S^{0,0}$, $S^{1,1}$, $S^{3,2}$, and $S^{7,4}$.
\tqed
\end{thmx}

The statement of \cref{thmx:products} is directly analogous to the classical result that the only sphere admitting unital products are $S^0$, $S^1$, $S^3$, and $S^7$. Classically, the nonexistence of $H$-space structures on any other spheres may be reduced to the Hopf invariant one problem, which was then established by Adams. This reduction makes use of the instability condition that $\Sq^a(x) = x^2$ whenever $x\in H^a(X)$ for some space $X$. There is an analogous instability condition for the motivic cohomology of a motivic space, but it holds only in a smaller range than we would need; as a consequence, some additional input is needed to analyze the unstable motivic Hopf invariant one problem (see \cref{rmk:instability}).

This additional input is interesting in itself. It follows from the formulation of the unstable motivic Hopf invariant one problem that, at least for nonexistence, one may reduce to the case where $F$ is algebraically closed. In \cref{ssec:lefschetz}, we explain how work of Wilson--{\O}stv{\ae}r \cite{WO17} implies a certain \textit{Lefschetz principle} for suitable $2$-primary categories of cellular motivic spectra. When combined with Mandell's $p$-adic homotopy theory \cite{Man01}, this gives a $2$-primary unstable ``Betti realization'' functor for any algebraically closed field $F$, which is well-behaved with respect to the mod $2$ cohomology of motivic cell complexes; see \cref{ssec:betti}. This gives a direct relation between motivic and classical homotopy theory, and we are then able to analyze the unstable motivic Hopf invariant one problem using a combination of classical results, work of Dugger--Isaksen \cite{DI13} on the motivic Hopf construction, and work of Asok--Doran--Fasel on smooth models of motivic spheres \cite{ADF17}.

Finally, in \cref{sec:hopf}, we turn to a study of the $1$-line of the $F$-motivic Adams spectral sequence. After a few preliminaries, we begin by proving \cref{thmx:nohopf}, producing the differentials 
\[
d_2(h_{a+1}) = (h_0+\rho h_1)h_a^2
\]
valid for any $F$ (\cref{thm:nohopf}). As we mentioned above, the main content of this theorem is not the fact that the classes $h_{a+1}$ for $a\geq 3$ support nonzero differentials, but the exact value of the target of these differentials. We mention two interesting aspects of this computation here:

First, in order to get a more explicit handle on possible targets of $d_2(h_{a+1})$, we reduce to considering the case where $F$ is a prime field, i.e.\ $F = \f_p$ with $p$ odd or $F = \q$. The latter case is then handled with the aid of a \textit{Hasse principle}. We explain how work of Ormsby--{\O}stv{\ae}r on the structure of $\m^\q$ \cite{OO13} may be used to give a concrete description of $\Ext_\q$ and of the Hasse map
\begin{equation}\label{eq:hasse}
\Ext_\q\rightarrow\Ext_\r\times\prod_{p\text{ prime}}\Ext_{\q_p},
\end{equation}
in particular proving this map is injective (\cref{prop:hasse}). In this way we reduce to computing the differentials $d_2(h_n)$ over the fields $\f_p$ with $p$ odd, $\q_p$ with $p$ prime, and $\r$.

Second, the classical argument, using the fact that $2 \sigma^2 = 0$, may be used to compute $d_2(h_4)$, but a new argument is required to produce the differential $d_2(h_5) = (h_0+\rho h_1)h_4^2$ (\cref{lem:d2h5f}). Once this differential is resolved, the remaining follow by an inductive argument analogous to Wang's classical argument \cite{Wan67}. After a further reduction when $F = \r$, the differential $d_2(h_5)$ may be resolved uniformly in the above choices of base field. In short, to resolve this differential, we lift the Hurewicz map $\pi_\ast^{\cl}\rightarrow\pi_{\ast,0}^F$ to a map $\Ext_\cl^{\ast,\ast}\rightarrow\Ext_F^{\ast,\ast,0}$ of spectral sequences (\cref{lem:diagonal}), and by considering the effect of this on the Kervaire class $\theta_4$, deduce that $(h_0+\rho h_1)h_4^2$ must be hit by $h_5$.

The story does not stop with the differentials $d_2(h_{a+1})$, as $\Ext_F^1$ contains many more classes than these; recall for instance $\Ext_\r^1$ from \cref{eq:extr1}. Having resolved these differentials, we move on to giving an explicit analysis of the $1$-line of the $F$-motivic Adams spectral sequence for a number of base fields $F$.  Our main results may be summarized in the following.

\begin{thmx}\label{thmx:1line}
The following are carried out in \cref{sec:hopf}.
\begin{enumerate}
\item In \cref{thm:r}, we compute all $d_2$-differentials out of $\Ext_\r^1$, as well as all permanent cycles in $\Ext_\r^1$.
\item In \cref{thm:enf1}, for $q$ a prime-power satisfying $q\equiv 1 \pmod{4}$, we compute all Adams differentials out of $\Ext_{\f_q}^1$, in particular giving all permanent cycles in $\Ext_{\f_q}^1$.
\item In \cref{thm:enf3}, for $q$ a prime-power satisfying $q\equiv 3 \pmod{4}$, we compute all $d_2$-differentials out of $\Ext_{\f_q}^1$, as well as all higher differentials in stems $s\leq 7$, in particular giving all permanent cycles in $\Ext_{\f_q}^1$ in stems $s\leq 7$.
\item In \cref{thm:qpf}, for $p$ an odd prime, we give as much information about $\Ext_{\q_p}^1$ as was given for $\Ext_{\f_p}^1$.
\item In \cref{thm:q2}, we compute all $d_2$-differentials out of $\Ext_{\q_2}^1$, as well as all higher differentials in stems $s\leq 7$, in particular giving all permanent cycles in $\Ext_{\q_2}^1$ in stems $s\leq 7$.
\item In \cref{thm:q}, we give the same information for $\Ext_\q^1$ as was given for $\Ext_{\q_2}^1$.
\tqed
\end{enumerate}
\end{thmx}

Cases (2--6) of \cref{thmx:1line} proceed by a direct analysis, combining the Hopf differentials we produced in \cref{thmx:nohopf} with arithmetic differentials that may be obtained by comparison with the $F$-motivic Adams spectral sequence for integral motivic cohomology. The latter has been computed by Kylling \cite{Kyl15} for $F = \f_q$ with $q$ an odd prime-power, by Ormsby \cite{Orm11} for $F = \q_p$ with $p$ an odd prime, and by Ormsby--{\O}stv{\ae}r \cite{OO13} for $F = \q_2$ and $F = \q$. Case (6), where $F = \q$, may be read off the cases $F = \r$ and $F = \q_p$, using our good understanding of the Hasse map \cref{eq:hasse}. As with Ormsby and {\O}stv{\ae}r's computations over $\q$, the final description of the set of $d_2$-cycles in $\Ext_\q^1$ is quite intricate, but we feel that our techniques show that understanding the $\q$-motivic Adams spectral sequence for $\pi_{\ast,\ast}^\q$ is an accessible problem ripe for future investigation.

The $\r$-motivic computation requires more work. Recall the structure of $\Ext_\r^1$ from \cref{eq:extr1}. \cref{thmx:nohopf} describes what happens on the $\rho$-torsion-free summand of this, but says nothing about the large quantity of $\rho$-torsion classes. It is possible to use similar methods to compute all $d_2$-differentials supported on this $\rho$-torsion summand, and we do so in \cref{prop:tauhopfdiff}. However, this is insufficient to determine which classes in $\Ext_\r^1$ are permanent cycles, as higher differentials may, and indeed must, occur.

We resolve this by comparison with \textit{Borel $C_2$-equivariant homotopy theory}. In \cite{BS20}, Behrens--Shah formulate and prove an equivalence
\[
(\Sp_\r^\cell)_{(2,\rho)}^\wedge[\tau^{-1}]\simeq \mathrm{Fun}(BC_2,\Sp_2^\wedge)
\]
between the $\tau$-periodic $(2,\rho)$-complete cellular $\r$-motivic category and the $2$-complete Borel $C_2$-equivariant category. Define
\[
\Ext_{BC_2}^{s,f,w} = \Ext_{\ca^{\cl}}^{s-w,f}(\f_2,H^\ast P^\infty_w),
\]
where $P^\infty_w$ is a stunted real projective space. These form the $E_2$-pages of the classical Adams spectral sequences for the stable cohomotopy groups of infinite stunted projective spaces. The equivalence of Behrens--Shah gives an effective method of computing these groups by ``inverting $\tau$'' in $\Ext_\r$. The $\tau$-periodic behavior of $\Ext_\r$ is plainly visible in our computation of $\Ext_\r^{\leq 3}$, allowing us to directly read off the structure of $\Ext_{BC_2}^{\leq 3}$ (\cref{lem:extbc2}). In particular,
\[
\Ext_{BC_2}^1 = \f_2[\rho]\{h_a:a\geq 1\}\oplus\bigoplus_{a\geq 0}\f_2[\rho]/(\rho^{2^{a}})\{\tau^{\floor{2^{a-1}(4n+1)}}h_a:n\in\z\}
\]
(compare \cref{eq:extr1}). We warn the reader that this naming of classes is incompatible with viewing $\Ext_{BC_2}$ as a collection of ordinary Adams spectral sequences; for example, $h_0$ does not detect $2$, but instead the transfer $P^\infty_0\rightarrow S^0$. We may use the relatively simple structure of these $1$-lines to verify that $\Ext_\r^1\rightarrow\Ext_{BC_2}^1$ reflects permanent cycles (\cref{lem:pconly}), and this reduces the identification of permanent cycles in $\Ext_\r^1$ to the identification of permanent cycles in $\Ext_{BC_2}^1$. The problem of $\rho$-torsion permanent cycles in $\Ext_{BC_2}^1$ turns out to be equivalent to the vector fields on spheres problem (\cref{lem:split}), which was resolved by Adams \cite{Ada62}. Together with known information regarding the $\rho$-torsion-free classes, this leads to the following classification of maps $\Sigma^c P^\infty_w\rightarrow S^0$ detected in Adams filtration $1$.

\begin{thmx}[\cref{thm:rpcohomotopy}]\label{thmx:rp}
For $a\geq 0$, write $a = c+4d$ with $0\leq c \leq 3$, and define $\psi(a) = 2^c+8d$. Then the subgroup of permanent cycles in $\Ext_{BC_2}^1$ is given by
\[
\f_2[\rho]\{h_1,h_2,h_3,\rho h_4\}\oplus\bigoplus_{a\geq 0}\f_2[\rho]/(\rho^{\psi(a)})\{\rho^{2^a-\psi(a)}\tau^{\floor{2^{a-1}(4n+1)}}h_a:n\in\z\}.
\]
Moreover, one may characterize maps $\Sigma^cP^\infty_w\rightarrow S^0$ detected by each of these classes.
\tqed
\end{thmx}

\subsection{Future directions}

The classical lambda algebra has been applied broadly in stable homotopy theory. This suggests several natural directions for future work, and we list a few here. 

\subsubsection{Homological computations}

The homology of the classical lambda algebra can be computed algorithmically via a method known as the \textit{Curtis algorithm}. This procedure was refined and implemented by Tangora \cite{Tan85} to compute the cohomology of the Steenrod algebra through internal degree $56$, as well as to compute products and Massey products \cite{Tan93, Tan94}; further computations of Curtis--Goerss--Mahowald--Milgram \cite{CGMM87} pushed this out to describe the cohomology of the Steenrod algebra through stem $51$. More recently, the Curtis algorithm was used by Wang--Xu to compute the algebraic Atiyah--Hirzebruch spectral sequence for $\r P^\infty$ \cite{WX16}, providing the data necessary for their proof of the uniqueness of the smooth structure on the $61$-sphere \cite{WX17}.

Our method for computing $\Ext_\r^{\leq 3}$ is closely related to the homology algorithm of \cite{Tan85}, only modified to take into account the $\f_2[\rho]$-module structure of $\Lambda^\r$, as well to incorporate some additional flexibility in choosing representatives for the sake of a more digestible manual computation. By ignoring this additional flexibility, and incorporating the ideas of \cite[Section 3.4]{Tan85}, one obtains a Curtis algorithm for computing the homology of the $\r$-motivic lambda algebra, as well as of other motivic lambda complexes. The effectiveness of these procedures in higher dimensions remains to be seen.

In addition to its use in computer-assisted computations, the classical lambda algebra has also been used by Lin \cite{Lin08} and Chen \cite{Che11} to completely compute the cohomology of the classical Steenrod algebra through filtration $5$. In principle, there should be no obstruction to continuing our computation of $\Ext_\r^{\leq 3}$ to higher filtrations, other than the rather more involved calculations and bookkeeping that this would necessarily take.

\subsubsection{Motivic Brown--Gitler spectra}

Brown--Gitler spectra \cite{BG73} have many applications in classical algebraic topology, including Mahowald's analysis of the $bo$-resolution \cite{Mah81, Shi84}, Cohen's solution of the Immersion Conjecture \cite{Coh85}, and more \cite{Mah77, HK99, Goe99}. The classical lambda algebra was essential for constructing and analyzing Brown--Gitler spectra; see \cite{BG73, Shi84} as above, as well as \cite{GJM86}. In \cite{CQ21}, the last two authors introduced a motivic analog of the $bo$-resolution, the $kq$-resolution, and analyzed it over algebraically closed fields of characteristic zero. The analysis of the $kq$-resolution over more general base fields would be greatly simplified by the existence of motivic Brown--Gitler spectra. 

\subsubsection{Unstable motivic Adams spectral sequences}

The classical lambda algebra $\Lambda^{\cl}$ has certain subcomplexes $\Lambda^{\cl}(n)$ which form the $E_1$-page of an unstable Adams spectral sequence:
\[
	E_1 \cong \Lambda^{\cl}(n)\Rightarrow \pi_*S^n. 
\]
Moreover, James's $2$-local fiber sequence \cite{Jam57}
\[
S^n\rightarrow \Omega S^{n+1}\rightarrow \Omega S^{2n+1},
\]
which gives rise to the EHP sequence, is modeled by short exact sequences \cite[Section 11]{Cur71}
\[
	0\rightarrow \Lambda^{\cl}(n)\rightarrow \Lambda^{\cl}(n+1)\rightarrow \Sigma^n\Lambda^{\cl}(2n+1)\rightarrow 0,
\]
which are useful both for understanding the unstable complexes $\Lambda^{\cl}(n)$ and the stable complex $\Lambda^{\cl}$. It is natural to ask whether there are analogous subcomplexes of $\Lambda^F$ related to a suitable motivic unstable Adams spectral sequence. The motivic situation seems to be much more delicate: it is not obvious how to define such subcomplexes of $\Lambda^F$, and the nature of the cohomology of motivic Eilenberg--MacLane spaces suggests that a motivic unstable Adams spectral sequence may not be as well-behaved. A better understanding of these topics would shed light both on the nature of $\Lambda^F$ and on unstable $F$-motivic homotopy theory.

\subsection{Conventions}\label{ssec:conventions}

We maintain the following conventions throughout the paper.

\begin{enumerate}
\item We work solely at the prime $2$.
\item We write $F$ for a base field of characteristic not equal to $2$.
\item We write $\pi_{\ast,\ast}^F$ for the homotopy groups of the $(2,\eta)$-completed $F$-motivic sphere spectrum. 
\item Our homotopy and cohomology groups are bigraded by $(s,w)$, where $s$ is stem and $w$ is weight. 
\item In particular, we write $S^{a,b}$ for the motivic sphere which is $\a^1$-homotopy equivalent to $\Sigma^{a-b}\g_m^{\wedge b}$.
\item We write $H^{\ast,\ast}$ for reduced mod $2$ $F$-motivic cohomology and $H^\ast$ for reduced ordinary mod $2$ cohomology.
\item We write, for instance, $H^{\ast,\ast}(X_+)$ for the unreduced mod $2$ motivic cohomology of $X$.
\item We will use homological grading even for cohomology classes, in the sense that if $x\in H^{a,b}(X)$ then we say $|x| = (-a,-b)$. This allows us to say, for instance, $|\tau| = (0,-1)$ and $|\rho| = (-1,-1)$, regardless or whether we are working with homology or cohomology.
\item We write $\m^F = H^{\ast,\ast}(\Spec (F)_+)$ for the unreduced mod $2$ motivic cohomology of a point.
\item We write $\m^F_0$ for the portion of $\m^F$ concentrated on the line $s = w$, so that $\m^F = \m^F_0[\tau]$. (The ring $\m^F_0$ may be identified as the mod $2$ Milnor $K$-theory of $F$, by work of Voevodsky; see \cite[Section 2.1]{IO19} for an overview of the structure of $\m^F$).
\item We write $\Ext_F$ for the cohomology of the $F$-motivic Steenrod algebra, employing the grading conventions given in the following two points.
\item We write $\Ext_F^f$ for the filtration $f$ piece of $\Ext_F$.
\item We write $\Ext_F^{s,f,w}\subset \Ext_F^f$ for the subset of elements in filtration $f$ and topological stem $s$ and weight $w$.
\item We use a subscript or superscript of $\cl$ to denote classical objects; in particular, $\pi_\ast^\cl$ are the classical $2$-completed stable stems, $\ca^\cl$ is the classical mod $2$ Steenrod algebra, and $\Ext_\cl$ is its cohomology.
\item Given integers $a,n\geq 0$, we write $2^{a-1}(2n+1)$ for the integer so represented for $a\geq 1$, and to be $n$ for $a = 0$.
\item We take the binomial coefficient $\binom{a}{b}$ to be $\frac{a!}{b!(a-b)!}$ for $0\leq b\leq a$, and to be zero otherwise.
\end{enumerate}

\subsection{Acknowledgements}

The authors thank U\v{g}ur Yi\v{g}it for sharing a copy of his thesis which proposes an alternative construction of a $C_2$-equivariant lambda algebra, Eva Belmont and Dan Isaksen for sharing data produced for their study of $\Ext_\r$, and the anonymous referees for helpful comments and suggestions. The first author thanks Nick Kuhn for some enlightening conversations about real projective spaces. The second author would like to thank the Max Planck Institute for providing a wonderful working environment and financial support while this paper was being written. The third author thanks Eva Belmont, Dan Isaksen, and Inna Zakharevich for helpful discussions.

\section*{Part I: The motivic lambda algebra}

\section{The motivic lambda algebra}\label{Sec:Structure}

In this section, we show that Priddy's construction \cite{Pri70} of the lambda algebra as a certain Koszul complex can be extended to produce a motivic lambda algebra. As noted in \cref{rmk:koszulsubtleties}, a more refined notion of Koszulity is needed to handle the more exotic nature of the $\m^F$-algebra $\ca^F$. The notion of a Koszul algebra has been generalized in various ways; see for instance \cite{PP05} for an account of some developments in this area. We will use the formulation given in \cite[Section 3]{Bal21}, as this gives a sufficiently general definition of Koszul algebra and explicit description of their associated Koszul complex. The reader familiar with Koszul algebras will find no surprises in this material.

In \cref{SS:Steenrod}, we review the structure of the $F$-motivic Steenrod algebra $\ca^F$. We show that $\ca^F$ is in fact a Koszul algebra in \cref{SS:PBWMotivic}, ultimately by reducing to Priddy's classical PBW criterion for Koszulity \cite[Section 5]{Pri70}. The \textit{$F$-motivic lambda algebra} $\Lambda^F$ is then defined to be the Koszul complex of $\ca^F$. We compute the structure of $\Lambda^F$ explicitly, and introduce an endomorphism $\theta$ of $\Lambda^F$ lifting the squaring operation $\Sq^0$ on $\Ext_F$. All of this structure is summarized in one place in \cref{SS:StructureSummary}.

\subsection{Review of Koszul algebras}\label{SS:Koszul}

This section summarizes the definitions and facts from \cite[Section 3]{Bal21} regarding Koszul algebras which we will use to construct the motivic lambda algebra. We review this material in some detail, in order to specialize from the more abstract context considered there. Many of the results we need have appeared in varying levels of generality throughout the literature; in particular, the definition of Koszulity we use can be considered as a direct generalization of the homogeneous case considered by Rezk in \cite[Section 4]{Rez12}.

We fix throughout this subsection an associative algebra $S$ to serve as our base ring, together with an associative algebra $A$ which is an $S$-algebra in the sense of being equipped with an algebra map $S\rightarrow A$. Equivalently, $A$ is a monoid in the category of $S$-bimodules. We abbreviate $\otimes = \otimes_S$.

We are most interested in the case where $S=\m^F$ and $A=\ca^F$, and so to avoid some subtle points regarding signs, we shall assume that $S$ is of characteristic $2$. In addition, we suppose throughout that $A$ is projective as a left $S$-module.

\begin{definition}
Say that $A$ is a \textit{graded $S$-algebra} if we have chosen a decomposition $A = \bigoplus_{n\geq 0}A[n]$ of $S$-bimodules such that
\begin{enumerate}
\item $S\cong A[0]$;
\item The product on $A$ restricts to $A[n]\otimes A[m]\rightarrow A[n+m]$.
\end{enumerate}
Say that $A$ is a \textit{filtered $S$-algebra} if we have chosen a filtration $A\cong\colim_{n\rightarrow\infty}A_{\leq n}$ such that
\begin{enumerate}
\item $S\cong A_{\leq 0}$;
\item The product on $A$ restricts to $A_{\leq n}\otimes A_{\leq m}\rightarrow A_{\leq n+m}$.
\end{enumerate}
Finally, say that the filtration on a filtered $S$-algebra $A$ is \textit{projective} if (both $A$ and) the associated graded algebra
\[
\gr A := \bigoplus_{n\geq 0}A[n],\qquad A[n] := \coker(A_{\leq n-1}\rightarrow A_{\leq n})
\]
are projective as left $S$-modules.
\tqed
\end{definition}

Fix a left $A$-module $M$. Write $B^{\unr}(A,A,M)$ and $B(A,A,M)$ for the unreduced and reduced bar resolutions of $M$ relative to $S$; that is, for the unnormalized and normalized chain complexes associated to the standard monadic resolution of $M$ with respect to the adjunction $\LMod_S\leftrightarrows\LMod_A$. These are projective left $A$-module resolutions provided that $M$ is projective as a left $S$-module. If $A$ is a filtered algebra, then $B^{\unr}(A,A,M)$ is a filtered complex, with filtration defined by
\begin{equation}\label{eq:barfiltration}
B^{\unr}_n(A,A,M)[\leq\! m] := \im\left( \bigoplus_{m_1 + \cdots + m_n = m} A \otimes A_{\leq m_1} \otimes \cdots \otimes A_{\leq m_n}\otimes M \rightarrow B_n^{\unr}(A,A,M) \right),
\end{equation}
and this descends to a filtration of $B(A,A,M)$; compare for instance \cite[Section 10]{Pri70}, \cite[Section 4]{Rez12}, or \cite[Section 3.5]{Bal21}.
If $A$ is augmented, then this augmentation makes $S$ into an $A$-bimodule, allowing us to form the bar complex $B(A) := S\otimes_A B(A,A,S)$ and consider the homology $H_\ast(A) := H_\ast(B(A))$, and the filtration of \cref{eq:barfiltration} descends to a filtration on $B(A)$. If $A$ is graded, then $A$ is naturally filtered by $A_{\leq n} = \bigoplus_{i\leq n}A[i]$; this filtration is split in the sense that $A\cong \gr A$, and likewise the filtration on $B(A)$ is split with $\gr B(A) = \bigoplus_{m\geq 0}B(A)[m]$. This then passes to a splitting $H_\ast(A) \cong \bigoplus_{m\geq 0}H_\ast(A)[m]$.

\begin{definition}[{\cite[Definition 4.4]{Rez12}, \cite[Definition 3.5.3]{Bal21}}]
We say that $A$ is a \textit{homogeneous Koszul $S$-algebra} provided that
\begin{enumerate}
\item $A$ has been given the structure of a graded $S$-algebra;
\item $H_n(A)[m]=0$ for $n\neq m$.
\end{enumerate}
We say that $A$ is a \textit{Koszul $S$-algebra} if
\begin{enumerate}
\item $A$ has been equipped with a projective filtration;
\item $\gr A$ is a homogeneous Koszul $S$-algebra.
\tqed
\end{enumerate}
\end{definition}

Suppose now that $A$ is projectively filtered, and fix a left $A$-module $M$ which is flat as a left $S$-module. The filtration of \cref{eq:barfiltration} on $B(A,A,M)$ induced by that on $A$ satisfies $\gr B(A,A,M) \cong A\otimes B(\gr A)\otimes M$, and so the convergent spectral sequence associated to this filtration is of signature
\begin{equation}\label{Eqn:KSS}
E^1_{p,q} = A\otimes H_q(\gr A)[p]\otimes M\Rightarrow H_q B(A,A,M),\qquad d^r_{p,q}\colon E^r_{p,q}\rightarrow E^r_{p-r,q-1}.
\end{equation}

\begin{definition}\label{Def:Koszul}
Let $M$ be an $A$-module which is flat as a left $S$-module. The \textit{Koszul resolution} of $M$ is the augmented chain complex
\[
M\leftarrow K(A,A,M)
\]
defined by
\[
K_p(A,A,M) = E^1_{p,p} = A\otimes H_p(\gr A)[p]\otimes M,
\]
with differential given by the $d^1$ differential of the spectral sequence \cref{Eqn:KSS}. When $M$ is projective as a left $S$-module, we define the \textit{Koszul complex} $K_A(M,M')$ as the cochain complex 
\[
K_A(M,M') := \Hom_A(K(A,A,M),M')\cong \Hom_S(H_\ast(\gr A)\otimes M,M'),
\]
with differential inherited from that on $K(A,A,M)$.
\tqed
\end{definition}

Observe that, by construction, $K(A,A,M)$ is a subcomplex of $B(A,A,M)$, and dually $K_A(M,M')$ is a quotient complex of the cobar complex $C_A(M,M') := \Hom_A(B(A,A,M),M')$. When $A$ is Koszul, the spectral sequence of \cref{Eqn:KSS} collapses into the Koszul complex $K(A,A,M)$, proving the following.

\begin{thm}[{cf.\ \cite[Theorem 3.8]{Pri70}, \cite[Proposition 4.8]{Rez12}, \cite[Theorem 3.5.5]{Bal21}}]\label{thm:koszulcx}
Suppose that $A$ is a Koszul $S$-algebra, and fix left $A$-modules $M$ and $M'$.
\begin{enumerate}
\item If $M$ is flat over $S$, then there is an injective quasiisomorphism $K(A,A,M)\subset B(A,A,M)$;
\item If $M$ is projective over $S$, then there is a surjective quasiisomorphism $C_A(M,M')\rightarrow K_A(M,M')$.
\end{enumerate}
In particular, if $M$ is projective over $S$, then the homology of $K_A(M,M')$ is isomorphic to $\Ext_A(M,M')$.
\qed
\end{thm}

This allows us to define Koszul complexes in the generality we need. We now recall some facts from \cite[Sections 3.6--3.7]{Bal21} describing the structure of Koszul complexes; these are direct analogues of \cite[Theorem 4.6]{Pri70}. We begin by fixing some conventions.

\begin{definition}
Fix a left $S$-module $M$. Then the dual $M^\vee = \LMod_S(M,S)$ carries the structure of a \textit{right} $S$-module by
\[
(f\cdot s)(m) = f(m)\cdot s.
\]
If $M$ is in fact an $S$-bimodule, then $M^\vee$ also carries an $S$-bimodule structure, with left $S$-module structure
\[
(s\cdot f)(m) = (f(m\cdot s)).
\]
Now, if $M$ is a left $S$-module and $M'$ is an $S$-bimodule, then there is a comparision map
\[
c\colon M^\vee\otimes M'^\vee\rightarrow (M'\otimes M)^\vee,\qquad c(f\otimes f')(m'\otimes m)=f'(m' f(m)).
\]
If $M$ is finitely presented and projective as a left $S$-module, then this map is an isomorphism. In general, if $M''$ is another left $S$-module, then we write
\[
M^\vee\cotimes M'' := \LMod_S(M,M''),
\]
so that in particular
\[
M^\vee\cotimes M'^\vee\cong (M'\otimes M)^\vee;
\]
in good cases, this may be realized as a topological tensor product, as the notation suggests.
\tqed
\end{definition}

The theory of Koszul algebras is closely related to the theory of quadratic algebras; let us fix some notation for these.

\begin{definition}
Fix an $S$-bimodule $B$ and subbimodule $R\subset B\otimes B$. The \textit{quadratic algebra} generated by the pair $(B,R)$ is the algebra
\begin{gather*}
T(B,R) := \bigoplus_{n \geq 0} T_n(B,R),\\
T_n(B,R) := \coker \left( \sum_{i+j=n} B^{\otimes i-1} \otimes R \otimes B^{\otimes j-1} \rightarrow B^{\otimes n} \right),
\end{gather*}
with multiplication inherited from the tensor algebra $T(B)$. Similarly, given a subbimodule $R'\subset B^\vee\cotimes B^\vee$ dual to a quotient of $B\otimes B$, we define the completed quadratic algebra
\begin{gather*}
\widehat{T}(B^\vee,R') := \prod_{n\geq 0}\widehat{T}_n(B^\vee,R'),\\
\widehat{T}_n(B^\vee,R') := \coker\left(\sum_{i+j=n} (B^\vee)^{\cotimes i-1}\cotimes R'\cotimes (B^\vee)^{\cotimes j-1}\rightarrow (B^\vee)^{\cotimes n}\right).
\end{gather*}

Say that $(B,R)$ is a \textit{quadratic datum} if $T(B,R)$ is projective. In this case, the \textit{dual quadratic datum} to $(B,R)$ is the pair $(B^\vee,R^\perp)$, where $R^\perp = (T_2(B,R))^\vee$.
\tqed
\end{definition}

The cohomology of a homogeneous Koszul algebra may be explicitly described as follows.

\begin{thm}[{cf.\ \cite[Theorem 2.5]{Pri70}, \cite[Proposition 4.12]{Rez12}, \cite[Theorem 3.6.4]{Bal21}}]\label{Thm:GensRelsKoszul}
\hphantom{blank}
\begin{enumerate}
\item Let $(B,R)$ be a quadratic datum. Then $H^1(T(B,R))[1] \cong B^\vee$, and the inclusion $B^\vee \subset H^*(T(B,R))$ extends to an isomorphism $\widehat{T}(B^\vee,R^\perp) \cong \prod_{n \geq 0} H^n(T(B,R))[n]$.

\item Let $A = \bigoplus_{n \geq 0} A[n]$ be a homogeneous Koszul algebra, and let $R = \ker(A[1] \otimes A[1] \rightarrow A[2])$. Then $A \cong T(A[1],R)$ is quadratic, and $H^*(A) \cong \widehat{T}(A[1]^\vee, R^\perp)$.
\qed
\end{enumerate}
\end{thm}

Now fix a quadratic algebra $A = T(A[1],R)$ and left $A$-modules $M$ and $M'$, supposing that $M$ is projective as a left $S$-module. We may use \cref{Thm:GensRelsKoszul} to describe the Koszul complex $K_A(M,M')$. Recall that
\[
K_A^n(M,M') = \LMod_A(A\otimes H_n(A)[n]\otimes M,M')\cong \LMod_S(H_n(A)[n]\otimes M,M').
\]
If we suppose that $H_\ast(A)$ is projective as a left $S$-module, as holds if $A$ is Koszul, then there is an isomorphism $(H_n(A))^\vee\cong H^n(A)$ of $S$-bimodules. In this case, we have
\[
K_A^n(M,M') \cong \LMod_S(M,H^n(\gr A)[n]\cotimes M')\cong \LMod_S(M,\widehat{T}_n(A[1]^\vee,R^\perp)\cotimes M');
\]
Thus $K_A^\ast(M,M')$ is completely described as a graded object by \cref{Thm:GensRelsKoszul}.

It remains to describe the differential on $K_A(M,M')$. Observe first that if $M''$ is an additional $A$-module, then there are pairings
\[
\wr\colon K_A^n(M,M')\otimes_\z K_A^{n'}(M',M'')\rightarrow K_A^{n+n'}(M,M'').
\]
This is a pairing of chain complexes compatible with analogous pairings on cobar complexes, and when $A$ is Koszul, it is a chain-level lift of the standard composition product in $\Ext_A$. In addition, it may be described in terms of the product structure on $\widehat{T}(A[1]^\vee,R^\perp)$ as follows (cf.\ \cite[Sections 3.2, 3.7]{Bal21}). Write $\mu$ for the multiplication on $\widehat{T}(A[1]^\vee,R^\perp)$. Then given $f\colon M\rightarrow \widehat{T}_n(A[1]^\vee,R^\perp)\cotimes M'$ and $g\colon M'\rightarrow \widehat{T}_{n'}(A[1]^\vee,R^\perp)\cotimes M''$, we have 
\[
f\wr g = (\mu\otimes 1)\circ (1\otimes g)\circ f.
\]

In the special case where $M = M'$, these pairings give $K_A(M,M)$ the structure of a differential graded algebra, and give $K_A(M,M')$ the structure of a differential graded $K_A(M,M)\hyp K_A(M',M')$-bimodule. Note that $K^1_A(M,M) = \LMod_S(A[1]\otimes M,M)$. The $A$-module structure on $M$ restricts to an element $Q^M\in K^1_A(M,M)$, and we have the following.

\begin{thm}[{\cite[Theorem 3.7.1]{Bal21}}]\label{Thm:DifferentialKoszul}
The differential on $K_A(M,M')$ is given by
\[
\delta: K_A^n(M,M') \rightarrow K_A^{n+1}(M,M'), \quad \delta(f) = Q^M \wr f -  f \wr Q^{M'}.
\]
In particular, if $M = M'$, then $\delta(f)$ is the commutator $[Q^M,f]$.
\qed
\end{thm}

This theorem describes Koszul complexes for a homogeneous Koszul algebra. Suppose now that $A$ is an arbitrary Koszul $S$-algebra, and fix still left $A$-modules $M$ and $M'$ with $M$ projective as a left $S$-module. The additive and multiplicative structure of the Koszul complexes $K_A(M,M')$ depend only on the algebra $\gr A$ and left $S$-modules $M$ and $M'$, and so are still described by \cref{Thm:GensRelsKoszul}. In practice, the differential on $K_A(M,M')$ may be identified using the following.

Let $qR = \ker(A_{\leq 1} \otimes A_{\leq 1} \rightarrow A_{\leq 2})$, and observe that $(A_{\leq 1},qR)$ is a quadratic datum. Let $A^{\big} = \bigoplus_{n \geq 0}A_{\leq n}$. This is a graded algebra, and the inclusion $A_{\leq 1}\subset A^{\big}$ extends to a map $T(A_{\leq 1},qR)\rightarrow A^{\big}$ of graded algebras.

\begin{thm}[{\cite[Theorem 3.7.3]{Bal21}}]\label{Thm:DifferentialKoszul2}
$~$
\begin{enumerate}
\item The map $T(A_{\leq 1},qR) \rightarrow A^{\big}$ is an isomorphism of graded algebras;
\item $A^{\big}$ is a homogeneous Koszul algebra;
\item The surjection $A^{\big}\rightarrow A$ gives rise to short exact sequences
\[
0\rightarrow K_A^n(M,M')\rightarrow K_{A^{\big}}^n(M,M')\rightarrow K_A^{n-1}(M,M')\rightarrow 0,
\]
which are split if $A$ is augmented.
\end{enumerate}

In particular, $K_A(M,M') \subset K_{A^{\big}}(M,M')$ is a subcomplex with differential on the target described by \cref{Thm:DifferentialKoszul}.
\tqed
\end{thm}

\subsection{The motivic Steenrod algebra}\label{SS:Steenrod}

We will construct the motivic lambda algebra by applying the theory recalled in \cref{SS:Koszul} to the mod $2$ motivic Steenrod algebra, whose structure we now recall. The conventions of \cref{ssec:conventions} are in force throughout this section.

We note in particular that, following these conventions, we take the somewhat nonconventional approach of consistently using \textit{homological grading}. Thus for example $\tau\in H^{0,1}(\Spec (F)_+)$, but we shall write $|\tau| = (0,-1)$, as this is how it will appear in the lambda algebra.

We begin by recalling the general structure of the base ring $\m^F = H^{\ast,\ast}(\Spec(F)_+)$.

\begin{example}\label{Exm:M2F}
For any $F$, we have $\m^F = \m^F_0[\tau]$, where
\[
|\tau| = (0,-1),
\]
and $\m^F_0\subset\m^F$ is the subring concentrated on the line $s=w$, isomorphic to the Milnor $K$-theory of $F$ taken mod $2$. The following are some particular examples of the ring $\m^F_0$. We refer the reader to \cite[Section 2.1]{IO19} for further details.

\begin{itemize}

\item For $F = \Fbar$ algebraically closed, such as $F=\c$, we have
\[
\m^{\Fbar}_0 \cong \f_2
\]

\item For $F = \r$ the real numbers, we have
\[
\m^\r_0 \cong \f_2[\rho],
\]
where $|\rho| = (-1,-1)$. 

\item For $F=\f_q$ a finite field of odd prime-power order $q$, we have
\[
\m^{\f_q}_0 \cong \begin{cases}
\f_q[u]/u^2 \quad & \text{ if } q \equiv 1 \pmod 4, \\
\f_q[\rho]/\rho^2 \quad & \text{ if } q \equiv 3 \pmod 4, 
\end{cases}
\]
where $|\rho|=|u| = (-1,-1)$. 

\item For $F=\q_p$ the $p$-adic rationals, for $p$ an arbitrary prime, we have
\[
\m^{\q_p}_0 \cong 	
	\begin{cases}
		\f_2[\pi,u]/(\pi^2,u^2) \quad & \text{ if } q \equiv 1 \pmod 4, \\
		\f_2[\pi, \rho]/(\rho^2, \rho \pi + \pi^2) \quad & \text{ if } q \equiv 3 \pmod 4,\\
		\f_2[\pi, \rho, u]/(\rho^3, u^2, \pi^2, \rho u, \rho \pi, \rho^2 + u \pi) \quad & \text{ if } q=2,
	\end{cases}
\]
where $|\rho| = |u| = |\pi| = (-1,-1).$

\end{itemize}
See also \cref{ssec:extf} for a discussion of $\m^\q$.
\tqed
\end{example}

Voevodsky \cite{Voe03} (with minor corrections by Riou \cite{Rio12}) and Hoyois--Kelly--{\O}stv{\ae}r \cite{HKO17} have constructed Steenrod squares
\[
\Sq^a : H^{m,n}(X) \rightarrow H^{m+a,n+\floor{a/2}}(X)
\]
for $a\geq 0$ and shown that they generate the algebra $\ca^F$ of natural operations in mod $2$ motivic cohomology. It is convenient to take the convention that $\Sq^a = 0$ for $a < 0$. The relations between these squares are generated by $\Sq^0=1$ together with the \textit{Adem relations}:

\begin{thm}[{\cite[Theorem 10.2]{Voe03}, \cite[Th{\'e}or{\`e}me 4.5.1]{Rio12}, \cite[Theorem 5.1]{HKO17}}]\label{Thm:Adem}
Fix positive integers $a$ and $b$ with $a<2b$.

If $a$ is even and $b$ is odd, then
\[
\Sq^a\Sq^b = \sum_{0\leq j \leq \floor{a/2}}\binom{b-1-j}{a-2j}\Sq^{a+b-j}\Sq^j+\sum_{\substack{1\leq j \leq \floor{a/2}\\j \text{ odd}}}\binom{b-1-j}{a-2j}\rho\Sq^{a+b-j-1}\Sq^j.
\]

If $a$ and $b$ are odd, then
\[
\Sq^a\Sq^b=\sum_{\substack{1\leq j \leq \floor{a/2}\\j \text{ odd}}}\binom{b-1-j}{a-2j}\Sq^{a+b-j}\Sq^j.
\]

If $a$ and $b$ are even, then
\[
\Sq^a\Sq^b = \sum_{0\leq j \leq \floor{a/2}}\tau^{j\,\mathrm{mod}\,2}\binom{b-1-j}{a-2j}\Sq^{a+b-j}\Sq^j.
\]

If $a$ is odd and $b$ is even, then
\[
\Sq^a\Sq^b =\sum_{\substack{0\leq j \leq \floor{a/2}\\j \text{ even}}}\binom{b-1-j}{a-2j}\Sq^{a+b-j}\Sq^j 
+\sum_{\substack{1\leq j \leq \floor{a/2}\\j \text{ odd}}}\binom{b-1-j}{a-1-2j}\rho\Sq^{a+b-j-1}\Sq^j.
\]
In all cases, the bounds on summation are not necessary, but give regions where the given binomial coefficients may be nonzero.
\qed
\end{thm}

As with the classical Steenrod algebra, $\ca^F$ admits an admissible basis.

\begin{definition}
Given a sequence $I = (r_1,\ldots,r_k)$ with $r_i >0$ for all $1 \leq i \leq k$, we abbreviate $\Sq^I = \Sq^{r_1}\ldots \Sq^{r_k}$. Say that $\Sq^I$ is \textit{admissible} if $r_i \geq 2r_{i+1}$ for all $1 \leq i \leq k-1$.
\tqed
\end{definition}

\begin{proposition}[{\cite[Section 11]{Voe03}}]
$\ca^F$ is freely generated as a left $\m^F$-module by the admissible squares $\Sq^I$.
\qed
\end{proposition}

The mod $2$ motivic cohomology $H^{\ast,\ast}(X_+)$ of any smooth scheme $X$ carries the structure of a left $\ca$-module. These actions satisfy the following \textit{Cartan formulas}.

\begin{proposition}[{\cite[Proposition 9.6]{Voe03}, \cite[Proposition 4.4.2]{Rio12}}]\label{prop:cartan}
Let $a \geq 0$ and $x,y \in H^{*,*}(X_+)$. Then
\begin{gather*}
\Sq^{2a}(x y) = \sum_{r=0}^a \Sq^{2r}(x) \Sq^{2a-2r}(y) + \tau \sum_{s=0}^{a-1}\Sq^{2s+1}(x) \Sq^{2a-2s-1}(y),\\
\Sq^{2a+1}(x y) = \sum_{r=0}^a \left(\Sq^{2r+1}(x) \Sq^{2ai-2r}(y) + \Sq^{2r}(x) \Sq^{2a-2r+1}(y)\right) + \rho \sum_{s=0}^{a-1} \Sq^{2s+1}(x)  \Sq^{2a-2s-1}(y). 
\end{gather*}
\qed
\end{proposition}

The action of $\ca^F$ on $\m^F$ is determined by these Cartan formulas and the following.

\begin{proposition}[{\cite{Voe03}; see also \cite[Appendix A]{RO16}}]\label{prop:baseact}
The action of $\ca^F$ on $\m^F$ satisfies
\begin{gather*}
\Sq^{\geq 1}(x) = 0,\qquad x\in\m^F_0 \\
\Sq^1(\tau) = \rho,\qquad \Sq^{\geq 2}(\tau) = 0 .
\end{gather*}
\qed
\end{proposition}

As in the classical case, the Cartan formulas of \cref{prop:cartan} may be encoded in a coproduct on the algebra $\ca^F$. The resulting structure is not quite a Hopf algebra, but is dual to a Hopf algebroid structure on the dual Steenrod algebra $(\ca^F)^\vee$. This complication arises in part due to the following. The Steenrod algebra $\ca^F$ is an $\m^F$-algebra, by way of the homomorphism $\m^F\rightarrow\ca^F$ sending an element $x\in \m^F$ to the stable operation given by left multiplication by $x$. However, $\m^F$ does not land in the center of $\ca^F$; equivalently, $\ca^F$ has nontrivial $\m^F$-bimodule structure. We may describe this structure explicitly as follows.

\begin{proposition}\label{Prop:BimoduleStructure}
The $\m^F$-bimodule structure of $\ca^F$ is determined by the following:
\begin{align*}
\Sq^n x &= x\Sq^n, \qquad x\in \m^F_0, \\
\Sq^{2n}\tau &= \tau\Sq^{2n}+\rho\tau\Sq^{2n-1}, \\
\Sq^{2n+1}\tau&=\tau\Sq^{2n+1}+\rho\Sq^{2n}+\rho^2\Sq^{2n-1}.
\end{align*}
\end{proposition}

\begin{proof}
It suffices to show both sides of each equality concide when evaluated on an arbitrary cohomology class. For example, for any $X$ and $x\in H^{\ast, \ast}(X_+)$, we have
\begin{align*}
(\Sq^{2n}\tau)(x) &= \Sq^{2n}(\tau x) \\
&= \sum_{i+j=n}(\Sq^{2i}\tau)(\Sq^{2j}x)+\tau\sum_{i+j=n-1}(\Sq^{2i+1}\tau)(\Sq^{2j+1}x)\\
&=\tau\Sq^{2n}(x)+\rho\tau\Sq^{2n-1}(x)
\end{align*}
by \cref{prop:cartan}. This proves the second equation, and the other cases are similar. 
\end{proof}

\begin{remark}\label{rmk:runiv}
Although we work in this section over an arbitrary base field $F$, there is a sense in which $F = \r$ represents the universal case: the class $\rho$ may be defined over any field $F$, making $\m^F$ into an $\m^\r$-module, and in all cases we have
\[
\ca^F = \m^F\otimes_{\m^\r}\ca^\r.
\]
In fact, the formulas of \cref{prop:baseact} describe an action of $\ca^\r$ on $\m^F$ for which
\[
\Ext_F\cong \Ext_{\ca^\r}(\m^\r,\m^F),
\]
and at least additively this depends only on the $\f_2[\rho]$-module structure of $\m^F_0$.

It is worth putting this observation in a slightly more general context. The Cartan formulas of \cref{prop:cartan} gives the category of left $\ca^\r$-modules a symmetric monoidal structure. If $R$ is a monoid in this category, then the tensor product $R\otimes_{\m^\r}\ca^\r$ may be equipped with a product with the property that
\[
\LMod_{R\otimes_{\m^\r}\ca^\r}\simeq \LMod_R(\LMod_{\ca^\r});
\]
this is the \textit{semi-tensor product} of \cite{MP65}. Moreover, we have
\[
\Ext_{R\otimes_{\m^\r}\ca^\r}(R,R)\cong \Ext_{\ca^\r}(\m^\r,R).
\]
The algebras $\ca^F$ are obtained in the case where $R = \m^F$.
Another simple class of example is given by the algebras $\ca^\r/(\rho^n,\tau^m)$, where $n$ and $m$ are such that $\tau^m$ is central in $\ca^\r/(\rho^n)$. A more interesting example is the following: there is an isomorphism of algebras
\[
\ca^{C_2}\cong\m^{C_2}\otimes_{\m^\r}\ca^\r,
\]
where $\ca^{C_2}$ is the $C_2$-equivariant Steenrod algebra, $\m^{C_2}$ is the $C_2$-equivariant cohomology of a point, and $\ca^\r$ acts on $\m^{C_2}$ as described, for instance, in \cite[Section 2]{GHIR19} (building on work of Hu--Kriz \cite{HK01}). 
\tqed
\end{remark}

\subsection{The motivic lambda algebra}\label{SS:PBWMotivic}

We now produce the motivic lambda algebra. For simplicity of notation, we consider the base field $F$ as fixed, and abbreviate
\[
 \ca = \ca^F,\qquad \m = \m^F
\]
throughout this subsection.

\subsubsection{Koszulity of $\ca$}

We begin by showing that $\ca$ is Koszul. The algebra $\ca$ is a projectively filtered $\m$-algebra under the length filtration: $\ca_{\leq n}\subset\ca$ is the submodule generated by squares $\Sq^I$ where $I$ is a sequence of length at most $n$. In particular,
\[
\ca_{\leq 1} = \m\{\Sq^a:a\geq 0\}
\]
as a left $\m$-module, with the understanding that $\Sq^0=1$ in $\ca$. By \cref{Def:Koszul}, to show that $\ca$ is Koszul we must show that $\gr \ca$ is homogeneous Koszul. To show that the classical Steenrod algebra is Koszul, Priddy developed a \textit{PBW criterion} for Koszulity \cite[Theorem 5.3]{Pri70}. We cannot apply this criterion directly, in part due to the nontrivial $\m$-bimodule structure of $\gr \ca$. Our strategy is to filter this issue away, thereby reducing to Priddy's criterion.

\begin{thm}\label{Thm:AMotKoszul}
$\ca$ is a Koszul $\m$-algebra. 
\end{thm}

\begin{proof}
As $\ca$ is a projectively filtered algebra, we must only show that $\gr\ca$ is a homogeneous Koszul algebra, i.e.\ that $H_n(\gr\ca)[m] = 0$ for $n\neq m$. To that end, we define a new filtration $\overline{F}_\bullet\gr\ca$ on $\gr\ca$ by declaring $\overline{F}_{\leq m} \gr\ca\subset\gr\ca$ to be generated by elements of the form $\Sq^I$, where $I = (r_1,\ldots,r_k)$ is a sequence satisfying $r_1+\cdots+r_k\leq m$. This filtration is multiplicative, and so we may form its associated graded algebra $\overline{\gr}\gr \ca$.

The same construction employed in \cref{SS:Koszul} shows that the filtration $\overline{F}_{\bullet}\gr\ca$ induces a filtration on the bar complex $B(\gr\ca)$ with associated graded $B(\overline{\gr}\gr\ca)$. This filtration is compatible with the decomposition
\[
B(\gr\ca)\cong \bigoplus _{m\geq 0}B(\gr\ca)[m],
\]
and so for each $m$ there is a convergent spectral sequence
\[
E_1^n = H_n B(\overline{\gr}\gr \ca)[m]\Rightarrow H_n(\gr\ca)[m].
\]
It is thus sufficient to verify that $\overline{\gr}\gr\ca$ is a homogeneous Koszul algebra with respect to the grading $\overline{\gr}\gr\ca = \bigoplus_{m\geq 0}\overline{\gr} \gr^m\ca$.
By passing from $\gr\ca$ to $\overline{\gr}\gr\ca$, we have filtered away both the nontrivial $\m$-bimodule structure on $\gr\ca$ described in \cref{Prop:BimoduleStructure} and the parts of the Adem relations involving $\rho$ which appear in \cref{Thm:Adem}, and in the end we may identify
\[
\overline{\gr}\gr\ca\cong\m^F\otimes_{\f_2[\tau]}\gr \ca^{\c}.
\]
From here, it is easily seen that the admissible basis of $\overline{\gr}\gr\ca$ satisfies Priddy's PBW criterion \cite[Section 5.1]{Pri70}. It now follows from \cite[Proposition 5.3]{Pri70} that $\overline{\gr}\gr A$ is Koszul; the assumption in \cite{Pri70} that the base is a field is not needed so long as everything in sight is free over the base.
\end{proof}

\begin{remark}\label{rmk:runiv2}
When $F=\r$, the filtration $\overline{F}_\bullet \gr A$ coincides with the $\rho$-adic filtration of $\gr A$. The use of $\overline{F}$ allows us to apply our argument uniformly to arbitrary base fields, but we could have also proved \cref{Thm:AMotKoszul} in the $\r$-motivic case, and deduced the general case from this. Indeed, everything in \cref{SS:Koszul} is compatible with base change (cf.\ \cite[Lemma 3.5.7]{Bal21}), so Koszulity of $\ca^\r$ implies that any algebra obtained from the construction of \cref{rmk:runiv} is Koszul. As an example not explicitly covered by the statement of \cref{Thm:AMotKoszul}, $\ca^{C_2}$ is Koszul over $\m^{C_2}$.
\tqed
\end{remark}

\begin{definition}\label{Def:Lambda}
The \textit{$F$-motivic lambda algebra} $\Lambda^F$ is the Koszul complex $K_{\ca^F}(\m^F,\m^F)$ associated to the Koszul $\m^F$-algebra $\ca^F$, as defined in \cref{Def:Koszul}, where $\ca^F$ acts on $\m^F$ as described in \cref{prop:baseact}. 
\tqed
\end{definition}

We shall abbreviate $\Lambda=\Lambda^F$ throughout the rest of this subsection. \cref{thm:koszulcx} now implies the following:

\begin{thm}\label{Thm:CohLambda}
Let $C(\ca) = C_\ca(\m,\m)$ denote the cobar complex of $\ca$. Then there is a surjective multiplicative quasiisomorphism
\[
C(\ca)\rightarrow \Lambda.
\]
In particular,
\[
H_\ast\Lambda\cong\Ext^\ast_\ca(\m,\m),
\]
and this isomorphism is compatible with all products and Massey products.
\qed
\end{thm}

\begin{remark}\label{Rmk:GeneralCobar}
More generally, the theory recalled in \cref{SS:Koszul} produces and describes Koszul complexes $K_\ca(M,M')$ modeling the cobar complex $C_\ca(M,M')$ for any left $\ca$-modules $M$ and $M'$ with $M$ projective over $\m$. Classically, the case where $M = H^\ast(\r P^\infty)$ and $M' = \f_2$ is of particular importance. Another amusing example is given over $F = \r$ with the observation that $K_{\ca^{\r}}(\m^\r,\m^{C_2})\cong K_{\ca^{C_2}}(\m^{C_2},\m^{C_2}) = \Lambda^{C_2}$ (cf.\ \cref{rmk:runiv}, \cref{rmk:runiv2}).
\tqed
\end{remark}

\subsubsection{The structure of the motivic lambda algebra}

We will now apply the theory recalled in \cref{SS:Koszul} to describe $\Lambda$ explicitly. First note that $\Lambda = \bigoplus_{m\geq 0}\Lambda[m]$ with $\Lambda[1] = (\ca[1])^\vee$, where $\ca[1] = \coker(\m\rightarrow\ca_{\leq 1})$. As a left $\m$-module, we may identify
\[
\ca[1]= \m\{\Sq^r:r\geq 1\}.
\]
Dualizing, we may identify
\[
\Lambda[1]=\{\lambda_r:r\geq 0\}\m
\]
as a right $\m$-module, where $\lambda_r$ is dual to $\Sq^{r+1}$ in the given basis. Considering internal algebraic degrees yields $|\lambda_r| = (r+1,\floor{(r+1)/2})$; following our conventions (\cref{ssec:conventions}), we subtract off the filtration from the algebraic stem to obtain the topological stem, and so instead write $|\lambda_r| = (r,\ceil{r/2})$.

We now begin by describing the multiplicative structure of $\Lambda$.

\begin{proposition}\label{Prop:BimodLambda}
The left $\m$-module structure on $\Lambda[1]$ is determined by
\begin{align*}
x \lambda_n &= \lambda_n x,\qquad x\in\m_0, \\
\tau \lambda_{2n+1} &= \lambda_{2n+1}\tau + \lambda_{2n+2}\rho, \\
\tau\lambda_{2n} &= \lambda_{2n}\tau + \lambda_{2n+1}\tau \rho + \lambda_{2n+2}\rho^2.
\end{align*}
\end{proposition}

\begin{proof}
This follows by dualizing \cref{Prop:BimoduleStructure}. 
\end{proof}

\begin{proposition}\label{Prop:AdemLambda}
If $a$ is odd or $b$ is even, then
\[
\lambda_a \lambda_{2a+b+1} = \sum_{0 \leq r < b/2} \lambda_{a+b-r}\lambda_{2a+1+r}\binom{b-r-1}{r},
\]
and if $a$ is even and $b$ is odd, then
\begin{align*}
\lambda_a \lambda_{2a+b+1} &= \sum_{0 \leq r < b/2} \lambda_{a+b-r}\lambda_{2a+1+r}\binom{b-r-1}{r}\tau^{(r-1) \text{ mod } 2} \\ 
					    &+ \sum_{0 \leq r \leq (b+1)/2} \lambda_{a+b+1-r}\lambda_{2a+1+r}\binom{\floor{b/2}-\floor{r/2}}{\floor{r/2}} \rho.
\end{align*}
\end{proposition}

\begin{proof}
By \cref{Thm:GensRelsKoszul}, the bimodule of relations defining $\Lambda$ as a quadratic algebra with generating bimodule $\Lambda[1]$ may be identified as $\ca[2]^\vee = \ker(\ca[1]^\vee\otimes\ca[1]^\vee\rightarrow R^\vee)$, where $R\subset \ca[1]\otimes \ca[1]$ is the projection of the subbimodule $qR\subset \ca_{\leq 1}\otimes\ca_{\leq 1}$ of Adem relations recalled in \cref{Thm:Adem}. It follows by direct computation that this kernel is generated by the indicated relations.
\end{proof}

\begin{remark}
Unless both $a$ and $b$ are even, the Adem relation expanding a product of the form $\lambda_a\lambda_b$ is exactly as in the classical lambda algebra.
\tqed
\end{remark}

The additive structure of $\Lambda$ may be understood just as in the classical case.

\begin{definition}\label{def:coadmissible}
Given a sequence $I = (r_1,\ldots,r_n)$, write $\lambda_I = \lambda_{r_1}\cdots\lambda_{r_n}$. Call the sequence $I$ \textit{coadmissible} if $2r_i\geq r_{i+1}$ for all $1\leq i \leq n-1$.
\tqed
\end{definition}

\begin{proposition}\label{prop:coadmissiblebasis}
$\Lambda$ is freely generated as a right $\m$-module by classes of the form $\lambda_I$ where $I$ is a coadmissible sequence.
\end{proposition}
\begin{proof}
The relations of \cref{Prop:AdemLambda} imply that the coadmissible classes $\lambda_I$ generate $\Lambda$ as a right $\m$-module, and we must only verify that they do so freely. Following \cref{rmk:runiv} and \cref{rmk:runiv2}, there is an isomorphism
\[
\Lambda \cong \Lambda^\r\otimes_{\m^\r}\m;
\]
thus we may reduce to the case where $F = \r$. By construction, $\Lambda$ is free as a right $\m$-module. Thus to show that the coadmissible classes $\lambda_I$ freely generate $\Lambda$ over $\m$, it is sufficient to verify the same for $\Lambda/(\rho)[\tau^{-1}]$ over $\m/(\rho)[\tau^{-1}]$. There is an isomorphism $\Lambda/(\rho)[\tau^{-1}]\cong \Lambda^{\cl}\otimes_{\f_2}\f_2[\tau^{\pm 1}]$, so this follows from the classical case.
\end{proof}

Finally, we describe the differential on $\Lambda$ by applying \cref{Thm:DifferentialKoszul}.

\begin{proposition}\label{Prop:DiffLambda}
The differential on $\Lambda$ is determined by the Leibniz rule, together with
\begin{align*}
\delta(x) &= 0,\qquad x\in\m_0 \\
\delta(\tau) &= \lambda_0 \rho, \\
  \delta(\lambda_n) &= \sum_{1 \leq r \leq n/2} \lambda_{n-r}\lambda_{r-1}\binom{n-r}{r}.
  \end{align*}
\end{proposition}

\begin{proof}
Recall the construction $\ca^{\big} = \bigoplus_{m\geq 0}\ca_{\leq m}$ used in the statement of \cref{Thm:DifferentialKoszul2}. By inspection, we find that $\ca^{\big}$ may be identified as the ``big motivic Steenrod algebra'', defined with generators and relations the same as $\ca$ only without the stipulation that $\Sq^0=1$. Let $\Lambda^{\big} = K_{\ca^{\big}}(\m,\m)$, where $\ca^{\big}$ acts on $\m$ through the quotient $\ca^{\big}\rightarrow\ca$, i.e.\ with $\Sq^0$ acting by the identity.

\cref{Thm:DifferentialKoszul2} tells us that $\ca^{\big}$ is a homogeneous Koszul algebra, and that there is an inclusion $\Lambda\subset\Lambda^{\big}$ of differential graded algebras. As $\ca^{\big}$ is homogeneous Koszul, \cref{Thm:GensRelsKoszul} applies to show that $\Lambda^{\big}$ is generated by classes $\lambda_r$ for $r\geq -1$, subject to relations of the same form as described for $\Lambda$ in \cref{Prop:BimodLambda} and \cref{Prop:AdemLambda}. The inclusion $\Lambda\subset\Lambda^{\big}$ is the obvious one, identifying $\Lambda$ as the subalgebra of $\Lambda^{\big}$ generated by the classes $\lambda_r$ for $r\geq 0$.

\cref{Thm:DifferentialKoszul} describes the differential on $\Lambda^{\big}$ as
\[
\delta(f) = [Q,f] =  Q\cdot f - f \cdot Q,
\]
where $Q\in\Lambda^{\big}[1]\cong (\ca^{\big}[1])^\vee$ is the map $\ca^{\big}[1]\cong \ca_{\leq 1}\otimes\m\rightarrow\m$ induced by the action of $\ca^{\big}$ on $\m$. In the basis $\ca^{\big}[1]=\ca_{\leq 1} = \m\{\Sq^r:r\geq 0\}$, this map is the projection onto $\Sq^0$, which by definition is the class $\lambda_{-1}\in\Lambda^{\big}$. So the differential on $\Lambda^{\big}$ is given by
\[
\delta(f) = [\lambda_{-1},f] = \lambda_{-1}f - f \lambda_{-1},
\]
and $\Lambda\subset\Lambda^{\big}$ is closed under this. The proposition follows upon expanding out this commutator using the relations defining the algebra $\Lambda^{\big}$.
\end{proof}

\begin{remark}
The description of the differential on $\Lambda$ as the commutator $\delta(f) = [\lambda_{-1},f]$ has appeared classically as well, see \cite[Page 83]{Bru88}.
\tqed
\end{remark}

\subsubsection{A closed formula for $\delta(\tau^n)$}

\cref{Prop:DiffLambda} gives a recursive process for computing $\delta(\tau^n)$. It is possible to solve this recursion, and we do so here. Recall that the pair $(\m,\ca^\vee)$ carries the structure of a Hopf algebroid. In particular, $\ca^\vee$ is a commutative ring, and $\ca_{\leq 1}^\vee$ is a quotient of this ring. Now, the differential $\delta\colon\Lambda[0]\rightarrow\Lambda[1]$ may be described as the composite 
\[
\eta_R + \eta_L\colon \Lambda[0] = \m\rightarrow \ca^\vee\rightarrow\ca_{\leq 1}^\vee\rightarrow \coker(\m\rightarrow\ca_{\leq 1}^\vee) = \Lambda[1],
\]
where $\eta_L,\eta_R\colon\m\rightarrow\ca^\vee$ are given by $\eta_R(m)(a) = \epsilon(ma)$ and $\eta_L(m)(a) = \epsilon(am)$, where $\epsilon\colon \ca = \ca\otimes_\m\m\rightarrow\m$ encodes the action of $\ca$ on $\m$.

We may use this interpretation to compute $\delta(\tau^n)$. The full structure of the Hopf algebroid $(\m,\ca^\vee)$ was determined by Voevodsky \cite{Voe03}; however, we only need a small piece of this, which is easily computed by hand from the structure of $\ca$ recalled in \cref{SS:Steenrod}. We record this piece in the following.

\begin{lemma}\label{lem:dual}
There is an isomorphism of rings
\[
\ca_{\leq 1}^\vee = \m[\tau_0,\xi_1]/(\tau_0^2+\xi_1\tau_0\rho+\xi_1\tau),
\]
where the quotient map
\[
\ca_{\leq 1}^\vee\rightarrow\Lambda[1]
\]
acts by
\[
\tau_0^\epsilon\xi_1^n\mapsto \lambda_{2n-1+\epsilon}
\]
for $\epsilon\in\{0,1\}$ and $n\geq 0$, with the interpretation that $\lambda_{-1}=0$. Moreover, the maps $\eta_L,\eta_R\colon \m\rightarrow\ca_{\leq 1}^\vee$ act by
\begin{align*}
\eta_R(x) &= x,\qquad x\in\m,\\
 \eta_L(x)&=x,\qquad x\in\m_0,\\
 \eta_L(\tau) &= \tau+\tau_0\rho.
\end{align*}
\end{lemma}
\begin{proof}
The structure of the ring $\ca_{\leq 1}^\vee$ may be read off the coproduct of $\ca$, as given in \cref{prop:cartan}, and its relation with our basis of $\Lambda[1]$ then follows by construction. The behavior of the left and right units may be read off the $\m$-bimodule structure of $\ca_{\leq 1}$ as given in \cref{Prop:DiffLambda}, together with knowledge of the counit map $\epsilon\colon\ca_{\leq 1}\rightarrow\m$ given in \cref{prop:baseact}.
\end{proof}

The main input to our computation of $\delta(\tau^n)$ is the following elementary computation.

\begin{lemma}\label{lem:taupow}
In the ring $\ca_{\leq 1}^\vee$, we have
\[
\tau_0^n = \sum_{\substack{\epsilon\in\{0,1\}\\ (n-\epsilon)/2\leq i \leq n-1}}\tau_0^\epsilon\xi_1^i\binom{i+\epsilon-1}{n-i-1}\tau^{n-i-\epsilon}\rho^{2i-n+\epsilon}.
\]
These bounds on $i$ are not necessary, but give a region where the binomial coefficients may be nonzero.
\end{lemma}
\begin{proof}
We first compute $\tau^n$ in the quotient ring
\[
\f_2[\tau_0,\xi_1]/(\tau_0^2 + \xi_1\tau_0+\xi_1).
\]
of $\ca_{\leq 1}^\vee$ in which both $\tau$ and $\rho$ are set to $1$. Clearly
\[
\tau_0^n = \sum_{0\leq i \leq n}\left(\xi_1^ic_{n,i}+\tau_0\xi_1^i d_{n,i}\right)
\]
for some $c_{n,i},d_{n,i}\in\f_2$. The relation
\[
\tau_0^n = \xi_1(\tau_0^{n-1}+\tau_0^{n-2})
\]
gives rise to recurrence relations
\begin{align*}
c_{n,i} &= c_{n-1,i-1}+c_{n-2,i-1} \\
d_{n,i} &= d_{n-1,i-1}+d_{n-2,i-1}.
\end{align*}
Set $c'_{i,n} = c_{n+i,i}$ and $d'_{i,n} = d_{n+i,i}$. Then these relations become
\begin{align*}
c'_{i,n} &= c'_{i-1,n-1} + c'_{i-1,n} \\
d'_{i,n} &= d'_{i-1,n-1} + d'_{i-1,n},
\end{align*}
exactly as seen in Pascal's triangle. Paired with the initial conditions
\[
c'_{i,0} = c'_{0,1} = d'_{1,0} = 0,\qquad c'_{1,1} = 1 = d'_{0,1},
\]
we find that
\[
c'_{i,n} = \binom{i-1}{n-1},\qquad d'_{n,i} = \binom{i}{n-1},
\]
and thus
\[
c_{n,i} = \binom{i-1}{n-i-1},\qquad d_{n,i} = \binom{i}{n-i-1}.
\]
Plugging this back in, we find
\[
\tau_0^n = \sum_{0\leq i \leq n}\left(\xi_1^i\binom{i-1}{n-i-1} + \tau_0\xi_1^i\binom{i}{n-i-1}\right) = \sum_{\substack{\epsilon\in\{0,1\} \\ 0\leq i \leq n}}\tau_0^\epsilon \xi_1^i \binom{i+\epsilon-1}{n-i-1}.
\]
To compute $\tau_0^n$ in $\ca_{\leq 1}^\vee$ itself, recall that $|\tau|=(0,-1)$, $|\rho| = (-1,-1)$, $|\tau_0| = (1,0)$, $|\xi_1| = (2,1)$. Solving 
\[
|\tau_0^n| = |\tau_0^\epsilon \xi_1^i\tau^a \rho^b|
\]
yields
\[
a = n-i-\epsilon,\qquad b = 2i - n +\epsilon.
\]
It follows that
\[
\tau_0^n = \sum_{\substack{\epsilon\in\{0,1\} \\ 0\leq i \leq n}}\tau_0^\epsilon\xi_1^i\binom{i+\epsilon-1}{n-i-1}\tau^{n-i-\epsilon}\rho^{2i-n+\epsilon}
\]
in $\ca_{\leq 1}^\vee$. For this binomial coefficient to be nonzero we require
\[
0 \leq i + \epsilon - 1,\qquad 0 \leq n - i - 1,\qquad n - i - 1 \leq i + \epsilon - 1,
\]
giving the stated bounds on summation.
\end{proof}

\begin{proposition}\label{prop:difftau}
The differential $\delta$ satisfies
\[
\delta(\tau^n) = \sum_{r\geq 0}\lambda_r\binom{n+\floor{r/2}}{r+1}\tau^{n-\floor{r/2}-1}\rho^{r+1}.
\]
\end{proposition}
\begin{proof}
Following \cref{lem:dual}, to compute $\delta(\tau^n)$ one may compute
\[
\tau^n+(\tau+\tau_0\rho)^n
\]
in terms of the standard basis of $\ca_{\leq 1}^\vee = \m[\tau_0,\xi_1]/(\tau_0^2+\xi_1\tau_0\rho+\xi_1\tau)$. Moreover, it is sufficient to work in the quotient of $\ca_{\leq 1}^\vee$ wherein $\tau$ and $\rho$ are set to $1$, as the necessary quantity of $\tau$'s and $\rho$'s may be recovered by comparing degrees, just as in the proof of \cref{lem:taupow}. Using \cref{lem:taupow}, we find
\[
1+(1+\tau_0)^n = \sum_{1\leq k \leq n}\binom{n}{k}\tau_0^k = \sum_{1\leq k \leq n}\binom{n}{k}\sum_{\substack{\epsilon\in\{0,1\},\\ i\geq 0}}\binom{i+\epsilon-1}{k-i-1}\tau_0^\epsilon\xi_1^i;
\]
here we are free to omit the bounds of summation on $i$, as they merely recorded when certain binomial coefficients were zero. The coefficient of $\tau_0^\epsilon\xi_1^i$ in this sum is
\[
\sum_{1\leq k \leq n} \binom{n}{k}\binom{i+\epsilon-1}{k-i-1} = \sum_{1\leq k \leq n}\binom{n}{k}\binom{i+\epsilon-1}{2i+\epsilon-k} = \binom{n+i+\epsilon-1}{2i+\epsilon};
\]
here, the first equality uses the standard identity $\binom{a}{b} = \binom{a}{a-b}$, and the second uses Vandermonde's identity. Adding in a sufficient number of $\rho$'s and $\tau$'s, and converting to $\Lambda[1]$, we learn
\[
\delta(\tau^n) = \sum_{\substack{\epsilon\in\{0,1\},\, i\geq 0 \\ (i,\epsilon) \neq (0,0)}} \lambda_{2i+\epsilon-1}\binom{n+i+\epsilon-1}{2i+\epsilon}\tau^{n-i-\epsilon}\rho^{2i+\epsilon}.
\]
Set $r = 2i + \epsilon - 1$. Then $\floor{r/2} = i + \epsilon - 1$, leading to the given description.
\end{proof}

\subsubsection{Lift of $\Sq^0$}

The dual motivic Steenrod algebra $\ca^\vee$ is a commutative Hopf algebroid, and thus its cohomology, which agrees by definition with $\Ext_{\ca}(\m,\m)$, is equipped with algebraic Steenrod operations \cite[Chapter 4]{BMMS86}. The purpose of this section is to lift the operation $\Sq^0$ to an endomorphism of $\Lambda$. Our approach essentially follows the proof of \cite[Proposition 1.4]{Pal07}.

Let $C(\ca) = C_\ca(\m,\m)$ denote the cobar complex of the algebra $\ca$; this is by definition the same as the cobar complex of the Hopf algebroid $\ca^\vee$. As $\ca^\vee$ is a commutative ring, $C(\ca)$ is the Moore complex of a cosimplicial commutative ring, and the levelwise Frobenius on this cosimplicial commutative ring induces a map
\[
\sigma\colon C(\ca)\rightarrow C(\ca).
\]
This is a map of differential graded algebras, and $\Sq^0$ is the map induced by $\sigma$ in homology.

\begin{thm}\label{Thm:Sq0}
The map $\sigma\colon C(\ca)\rightarrow C(\ca)$ induced by the levelwise Frobenius descends to a map
\[
\theta\colon \Lambda\rightarrow\Lambda
\]
of differential graded algebras. This map is given on generators by
\begin{align*}
\theta(x) &= x^2,\qquad x\in\m, \\
\theta(\lambda_{2n-1}) &= \lambda_{4n-1}, \\
\theta(\lambda_{2n}) &=  \lambda_{4n+1}\tau+ \lambda_{4n+2}\rho.
\end{align*}
\end{thm}
\begin{proof}
Recall $\ca^{\big}$ and $\Lambda^{\big}$ from the proof of \cref{Prop:DiffLambda}. Let $C(\ca^{\big})$ be the cobar complex for $\ca^{\big}$ with respect to augmentation of $\ca^{\big}$, so that $H_\ast C(\ca^{\big}) = \Lambda^{\big}$ as algebras. The levelwise Frobenius gives a map
\[
\sigma\colon C(\ca^{\big})\rightarrow C(\ca^{\big})
\]
of differential graded algebras, and by taking homology this induces a map
\[
\theta'\colon \Lambda^{\big}\rightarrow\Lambda^{\big}
\]
of algebras. We claim that to produce $\theta$ it suffices to show that $\theta'$ restricts to an endomorphism of $\Lambda\subset\Lambda^{\big}$ satisfying the given formulas. Indeed, there is an inclusion $C(\ca)\subset C(\ca^{\big})$ of algebras, which does not respect differentials but does commute with the levelwise Frobenius $\sigma$. It would thus follow that the restriction $\theta$ of $\theta'$ to $\Lambda$ is induced by the levelwise Frobenius on $C(\ca)$. In particular, this would show that $\sigma\colon C(\ca)\rightarrow C(\ca)$ indeed descends to an algebra map $\theta\colon \Lambda\rightarrow\Lambda$. That $\theta$ moreover respects the differential is inherited from $\sigma$.

To understand $\theta'$, it suffices to understand its effect on the generators of $\Lambda^{\big}$, i.e.\ to understand the map
\[
\theta'\colon \Lambda^{\big}[1]\rightarrow\Lambda^{\big}[1].
\]
Recall that $\Lambda^{\big}[1] = (\ca^{\big}[1])^\vee = \ca_{\leq 1}^\vee$. This ring was described in \cref{lem:dual}, and $\theta'$ acts on it by the Frobenius. We find that $\theta'$ satisfies the same formulas as described for $\theta$, only with the addition that $\theta'(\lambda_{-1}) = \lambda_{-1}$. In particular $\theta'$ does restrict to $\Lambda$, and this restriction satisfies the stated formulas, proving the theorem.
\end{proof}

\subsection{Summary}\label{SS:StructureSummary}

For ease of reference, let us summarize what we have learned in one place. As always, $F$ is a base field of characteristic not equal to $2$.

\subsubsection{Generators}

There is a differential graded algebra $\Lambda^F$, the \textit{$F$-motivic lambda algebra}, together with a multiplicative quasiisomorphism $C(\ca^F)\rightarrow\Lambda^F$, where $C(\ca^F)$ is the reduced cobar complex of $\ca^F$. We write $\Lambda^F = \bigoplus_{m\geq 0}\Lambda^F[m]$, where the differential on $\Lambda^F$ is of the form $\delta\colon\Lambda^F[m]\rightarrow\Lambda^F[m+1]$.

The $F$-motivic lambda algebra $\Lambda^F$ is an $\m^F$-bimodule algebra, generated by classes $\lambda_r\in\Lambda^F[1]$ for $r\geq 0$. In the trigrading $(\text{stem},\text{filtration},\text{weight})$, we have
\[
|\tau| = (0,0,-1),\qquad |\rho| = (-1,0,-1),\qquad |\lambda_a| = (a,1,\ceil{a/2}).
\]
A right $\m^F$-module basis of $\Lambda^F$ is given by those $\lambda_{r_1}\cdots \lambda_{r_n}$ with $2r_i\geq r_{i+1}$ for $1\leq i \leq n-1$.

\subsubsection{Relations}

The $F$-motivic lambda algebra is a quadratic $\m^F$-bimodule algebra. Recall that $\m^F = \m^F_0[\tau]$. The $\m^F$-bimodule structure of $\Lambda^F$ is determined by
\begin{align*}
x \lambda_n &= \lambda_n x,\qquad x\in\m_0^F, \\
\tau \lambda_{2n+1} &= \lambda_{2n+1}\tau + \lambda_{2n+2}\rho, \\
\tau\lambda_{2n} &= \lambda_{2n}\tau + \lambda_{2n+1}\tau \rho + \lambda_{2n+2}\rho^2,
\end{align*}
and the quadratic relations are given as follows. Fix $a,b\geq 0$. If $a$ is odd or $b$ is even, then
\[
\lambda_a \lambda_{2a+b+1} = \sum_{0 \leq r < b/2} \lambda_{a+b-r}\lambda_{2a+1+r}\binom{b-r-1}{r};
\]
and if $a$ is even and $b$ is odd, then
\begin{align*}
\lambda_a \lambda_{2a+b+1} &= \sum_{0 \leq r < b/2} \lambda_{a+b-r}\lambda_{2a+1+r}\binom{b-r-1}{r}\tau^{(r-1) \text{ mod } 2}  \\
					     &+ \sum_{0 \leq r \leq \ceil{b/2}} \lambda_{a+b+1-r}\lambda_{2a+1+r}\binom{\floor{b/2} - \floor{r/2}}{\floor{r/2}} \rho .
\end{align*}

\subsubsection{Differentials}

The differential on $\Lambda$ is determined by the Leibniz rule, together with
\begin{align*}
\delta(x) &= 0,\qquad x\in\m^F_0, \\
\delta(\tau) &= \lambda_0\rho,\\
\delta(\lambda_n) &= \sum_{1\leq r \leq n/2}\lambda_{n-r}\lambda_{r-1}\binom{n-r}{r}.
\end{align*}
Moreover, we have
\[
\delta(\tau^n) = \sum_{r\geq 0}\lambda_r\binom{n+\floor{r/2}}{r+1}\tau^{n-\floor{r/2}-1}\rho^{r+1}.
\]

\subsubsection{The endomorphism $\theta$}

The squaring operation $\Sq^0\colon \Ext_F^{s,f,w}\rightarrow\Ext_F^{2s+f,f,w+f}$ lifts to an endomorphism $\theta\colon\Lambda^F\rightarrow\Lambda^F$ of differential graded algebras, determined by
\begin{align*}
\theta(x)& = x^2,\qquad x\in\m^F, \\
\theta(\lambda_{2n-1})& = \lambda_{4n-1},\\
\theta(\lambda_{2n}) &=  \lambda_{4n+1}\tau+ \lambda_{4n+2}\rho.
\end{align*}

\section{Some first examples, and the doubling map}\label{Sec:FirstExamples}

\subsection{First examples}\label{ssec:simpleexamples}

Before continuing on, we give some basic examples illustrating the form of the motivic lambda algebra. In particular, we use $\Lambda^F$ to define some classes in $\Ext_F$, and reprove some well-known low-dimensional relations. This material is meant only to familiarize the reader with $\Lambda^F$; we give a more thorough and entirely self-contained investigation in \cref{Sec:ExtLow}.

Given a cycle $z\in\Lambda^F$, in this section we write $[z]\in\Ext_F$ for the corresponding cohomology class.

\begin{lemma}
We have $\delta(\lambda_{2^a-1}) = 0$ for all $a\geq 0$.
\end{lemma}

\begin{proof}
The proof is identical to the proof of \cite[Proposition 2.2]{Wan67}. 
\end{proof}

This allows us to define the following Hopf elements.

\begin{definition}
Let $h_a := [\lambda_{2^a-1}]$.
\tqed
\end{definition}

\begin{lemma}
If $\rho = 0$ in $\m^F$, such as if $F$ is algebraically closed, then $\delta(\tau^n) = 0$ for all $n \geq 0$. 
\end{lemma}
\begin{proof}
This is immediate from the differential $\delta(\tau) = \lambda_0\rho$.
\end{proof}

In general, if $\rho$ is nilpotent in $\m^F$, then various powers of $\tau$ will be cycles in $\Lambda^F$. We shall write $\tau^n$ in place of $[\tau^n]$ in this case. We begin by considering some examples in the case where $F$ is algebraically closed.

\begin{proposition}\label{prop:h13}
For $F$ algebraically closed, there is a relation
\[
\tau \cdot h_1^3 = h_2 h_0^2.
\]
\end{proposition}
\begin{proof}
By definition, $\tau \cdot h_1^3 = [\lambda_1^3\tau]$ and $h_2 h_0^2 = h_0^2 h_2 = [\lambda_0^2\lambda_3]$. We have
\[
\lambda_0^2 \lambda_3 = \lambda_1^3\tau,
\]
so these classes coincide in $\Ext_F$.
\end{proof}

\begin{proposition}
For $F$ algebraically closed, there is a relation
\[
\tau\cdot h_1^4 = 0.
\]
However, $h_1^n \neq 0$ for any $n$.
\end{proposition}
\begin{proof}
Observe that $\lambda_0\lambda_1 = 0$, and thus
$
h_1 h_0 = 0.
$
Combined with \cref{prop:h13}, we find
\[
\tau\cdot h_1^4 = \tau h_1^3\cdot h_1 = h_2 h_0^2\cdot h_1 = 0.
\]
Alternately, $\tau h_1^4 = [\lambda_1^4\tau]$, and there is a differential
\[
\delta(\lambda_2^2\lambda_1) = \lambda_1^4\tau.
\]
On the other hand, for $h_1^n$ to vanish, the class $\lambda_1^n$ must be nullhomotopic, i.e.\ $\delta(x)=\lambda_1^n$ for some $x\in\Lambda$. The class $x$ must live in stem $n+1$, weight $n$, and filtration $n-1$, and in this degree $\Lambda$ is generated by the cycle $\lambda_3\lambda_1^{n-2}$. So no such $x$ exists.
\end{proof}

Next we consider some examples relevant to base fields $F$ over which $\rho$ does not vanish.
We begin by defining some classes. Note that the differential
\[
\delta(\tau) = \lambda_0 \rho
\]
implies that $\delta(\tau^{2^n})\equiv 0 \pmod{\rho^{2^n}}$. This allows for the following definition.

\begin{definition}
If $F = \r$, then
\[
\tau^{2^{a-1}}h_a := \left[\frac{1}{\rho^{2^a}}\delta(\tau^{2^a})\right]
\]
for $a\geq 1$. In general, $\tau^{2^{a-1}}h_a\in\Ext_F$ is defined by pushing these classes forward along the map $\Lambda^\r\rightarrow\Lambda^F$ induced by $\m^\r\rightarrow\m^F$ (cf.\ \cref{rmk:runiv}).
\tqed
\end{definition}

\begin{remark}
Following our convention that $\Lambda^F$ is considered primarily as a right $\m^F$-module, it would be more natural to write $h_a \tau^{2^{a-1}}$ for the classes introduced above. We have chosen instead to work with naming conventions compatible with those in \cite{BI20}, as no confusion should arise.
\tqed
\end{remark}

\begin{remark}
If $\tau^{2^{a-1}}$ is a cycle in $\Ext_F$, then $\tau^{2^{a-1}}h_a = \tau^{2^{a-1}}\cdot h_a$.
\tqed
\end{remark}

\begin{example}
We have
\begin{align*}
\tau h_1 &= [\lambda_1 \tau + \lambda_2 \rho], \\
\tau^2 h_2 &= [\lambda_3 \tau^2 + \lambda_5 \tau \rho^2 + \lambda_6 \rho^3], \\
\tau^4 h_3 &= [\lambda_7 \tau^4 + \lambda_{11} \tau^2 \rho^4 + \lambda_{13} \tau \rho^6 + \lambda_{14} \rho^7].
\end{align*}
In fact, we may identify $\tau^{\floor{2^{a-1}}}h_a = [\tau^{2^a} \lambda_{2^a-1}]$ for all $a\geq 1$.
\tqed
\end{example}

The following relation was proved over $\r$ by Dugger--Isaksen \cite[Proof of Lemma 6.2]{DI16a} using Massey products and May's Convergence Theorem. We may use the lambda algebra to provide an explicit direct proof.

\begin{proposition}\label{Prop:HiddenExample}
There is a relation
\[
(h_0+\rho h_1)\cdot\tau h_1 = 0.
\]
\end{proposition}

\begin{proof}
By definition, 
\begin{align*}
h_0 \cdot \tau h_1 &= [\lambda_0(\lambda_1 \tau + \lambda_2 \rho)], \\
\rho h_1 \cdot \tau h_1 &= [\rho \lambda_1(\lambda_1 \tau + \lambda_2 \rho)].
\end{align*}
Expanding, we have
\begin{align*}
\lambda_0(\lambda_1 \tau + \lambda_2 \rho) &= \lambda_1^2 \tau \rho + \lambda_1 \lambda_2 \rho^2 + \lambda_{2}\lambda_1 \rho^2, \\
\rho h_1(\lambda_1 \tau + \lambda_2 \rho) &= \lambda_1^2 \tau \rho + \lambda_1 \lambda_2 \rho^2.
\end{align*}
But
\[
\delta(\lambda_3 \tau \rho + \lambda_4 \rho^2) =  \lambda_2 \lambda_1 \rho^2,
\]
so $h_0\cdot\tau h_1  = \rho h_1 \cdot \tau h_1$. The proposition follows.
\end{proof}

The fact that $\delta(\tau^n) \equiv 0\pmod{\rho}$ allows for the following definition.

\begin{definition}
If $F = \r$, then
\[
\tau^{2n}h_0 := \left[\frac{1}{\rho}\delta(\tau^{2n+1})\right].
\]
In general, $\tau^{2n}h_0\in\Ext_F$ is defined by pushing these classes forward along the map $\Lambda^\r\rightarrow\Lambda^F$ induced by $\m^\r\rightarrow\m^F$ (see \cref{rmk:runiv}).
\tqed
\end{definition}

\begin{example}
We have
\begin{align*}
h_0 &= [\lambda_0], \\
\tau^2 h_0 &= [\lambda_0 \tau^2 + \lambda_1 \tau^2 \rho + \lambda_3 \tau \rho^3 + \lambda_4 \rho^4], \\
\tau^4 h_0 &= [\lambda_0\tau^4 + \lambda_3\tau^3\rho^3 + \lambda_4\tau^2\rho^4+\lambda_5\tau^2\rho^5+\lambda_7\tau\rho^7+\lambda_8\rho^8].
\end{align*}
\tqed
\end{example}

The following proposition was originally proved over $\r$ by Dugger--Isaksen \cite[Proof of Lemma 6.2]{DI16a} using Massey products, May's Convergence Theorem, and analysis of the $\rho$-Bockstein spectral sequence. Using the lambda algebra, the proof amounts to checking that the products of cycle representatives are equal. 

\begin{proposition}
There is a relation
\[
\tau^2 h_0 \cdot h_1 = \rho (\tau h_1)^2.
\]
\end{proposition}

\begin{proof}
We may directly compute
\begin{align*}
\tau^2 h_0 \cdot h_1 &= [(\lambda_0 \tau^2 + \lambda_1 \tau^2 \rho + \lambda_3 \tau \rho^3 + \lambda_4 \rho^4) \lambda_1] \\
				&= [\lambda_1^2 \tau^2 \rho + \lambda_2\lambda_1 \tau \rho^2 + \lambda_2^2 \rho^3 + \lambda_2 \lambda_3 \rho^4] \\
				 &= [\rho (\lambda_1 \tau + \lambda_2 \rho)^2] = \rho(\tau h_1)^2.
\end{align*}
\end{proof}

\subsection{The doubling map}\label{ssec:doubling}

In \cite[Theorem 4.1]{DI16a}, Dugger--Isaksen produce an isomorphism
\[
\Ext_{\cl}[\rho^{\pm 1}]\cong\Ext_\r[\rho^{-1}],
\]
which doubles internal degrees. We can lift this isomorphism to a quasiisomorphism of lambda algebras.

\begin{proposition}\label{prop:retract}
Let $\Lambda^{\dcl}$ denote the classic lambda algebra, only given a motivic grading where $|\lambda_n|$ has stem $2n+1$ and weight $n+1$. For any $F$, there is a retraction
\begin{center}\begin{tikzcd}
\Lambda^{\dcl}\ar[r,"\thetatilde"]&\Lambda^F\ar[r]&\Lambda^{\Fbar}\ar[r,"q"]&\Lambda^{\dcl}
\end{tikzcd}\end{center}
with the following properties:
\begin{enumerate}
\item All maps shown are maps of differential graded algebras respecting $\theta$;
\item $\thetatilde$ is given on generators by $\thetatilde(\lambda_n)=\lambda_{2n+1}$;
\item $q$ is given on generators by $q(\tau) = 0$, $q(\lambda_{2n})=0$, and $q(\lambda_{2n+1}) = \lambda_n$.
\end{enumerate}
Now say $F = \r$, and write $\Ext_\r^{\rho\hyp\tors}\subset\Ext_\r$ for the $\rho$-torsion subgroup of $\Ext_\r$.
\begin{enumerate}[resume]
\item The map $\Ext_{\dcl}[\rho]\oplus\Ext_\r^{\rho\hyp\tors}\rightarrow\Ext_\r$ induced by $\thetatilde$ and the inclusion of $\rho$-torsion is an isomorphism;
\item In particular, $\thetatilde$ extends to a quasiisomorphism $\Lambda^{\dcl}\otimes_{\f_2}\f_2[\rho^{\pm 1}]\rightarrow\Lambda^\r[\rho^{-1}]$.
\end{enumerate}
\end{proposition}
\begin{proof}
The assignments given in (2) and (3) are easily seen to extend to maps of differential graded algebras, proving (1), and that the resulting sequence is a retraction is clear. Evidently (4) implies (5), so we are left with proving (4).

It is equivalent to verify that the composite $\Ext_{\dcl}[\rho]\rightarrow\Ext_\r\rightarrow\Ext_\r/\Ext_\r^{\rho\hyp\tors}$ is an isomorphism. This is a split inclusion of free $\f_2[\rho]$-modules, so for it to be an isomorphism it is sufficient to verify that it is an isomorphism after inverting $\rho$, and for this it is sufficient for the injection $\Ext_{\dcl}[\rho^{\pm 1}]\rightarrow\Ext_\r[\rho^{-1}]$ to be an isomorphism. By Dugger--Isaksen's isomorphism $\Ext_\r[\rho^{-1}]\cong\Ext_{\dcl}[\rho^{\pm 1}]$ \cite[Theorem 4.1]{DI16a}, we find that our map $\Ext_{\dcl}[\rho^{\pm 1}]\rightarrow\Ext_\r[\rho^{-1}]$ is an injection between vector spaces of equal finite rank in each degree, and is thus an isomorphism.
\end{proof}

\begin{remark}
We point out the following amusing corollary of \cref{prop:retract}: there is a multiplicative injection
\[
Q\colon \ker(\Sq^0\colon \Ext_{\cl}\rightarrow\Ext_{\cl})\rightarrow \Ext_\c^{\tau\hyp\tors},
\]
acting in degrees as $\Sq^0$ would. For example, as $\thetatilde{\lambda_0^n} = \lambda_1^n$, we find that $Q(h_0^n)=h_1^n$. This provides another explanation of the fact that $h_1$ is not nilpotent in $\Ext_\c$. It is natural to ask whether $Q$ accounts for all indecomposable $\tau$-torsion classes in $\Ext_\c$, but a counterexample is given by the class $B_6$ in stem $55$ and filtration $7$, as $\Ext_\cl^{24,7} = 0$.
\tqed
\end{remark}

\section{\texorpdfstring{$\Ext_\r$}{Ext\_R} in filtrations \texorpdfstring{$f\leq 3$}{f <= 3}}\label{Sec:ExtLow}

In this section, we use the $\r$-motivic lambda algebra to compute $\Ext_\r^f$ for $f\leq 3$. Throughout this section, we shall abbreviate 
\[
\Lambda = \Lambda^\r.
\]

\subsection{Preliminaries}\label{ssec:extpreliminaries}

We begin by describing our strategy for computing $\Ext_\r$. We rely on the following device, which uses ideas from Tangora's work on the classic lambda algebra \cite{Tan85} to produce something like a chain-level lift of the $\rho$-Bockstein spectral sequence \cite{Hil11}. While the algorithm is essentially standard, we include a detailed description since we were unable to find a reference with the algorithm in precisely the form we need in the sequel. We begin with some preliminary definitions.

\begin{definition}
Let $V = \f_2\{x_s:s\in S\}$ be a (locally) finite $\f_2$-vector space with ordered basis.
\begin{enumerate}
\item The \textit{leading term} of a class $x\in V$ is the largest term appearing when $x$ is written as a sum of basis elements.
\item We write $x < x'$ when the leading term of $x$ is less than that of $x'$.
\item Given another vector space $U = \f_2\{x_s:s\in T\}$ with ordered basis, map $\phi\colon V\rightarrow U$, and $s\in S$, $t\in T$, we write
\[
\phi(x_s+{<}) = y_t+{<}
\]
for the \textit{relation} that there exist some classes $u < x_s$ and $v < y_t$ for which $\phi(x_s+u) = y_t+v$.
\tqed
\end{enumerate}
\end{definition}

The main technical lemma we need is the following. The reader is invited to skip this lemma on first reading; the details are not necessary to understand our computation, and we rephrase what we need in the context of $\Lambda$ in \cref{thm:tags}.

\begin{lemma}\label{lem:tags}
Let $(C,d)$ be a chain complex of locally finite and free $\f_2[\rho]$-modules, and suppose (for simplicity) that $H_\ast C[\rho^{-1}]=0$. Choose an ordered basis $\f_2\{x_s:s\in T\}$ for $C/(\rho)$, and extend this to a basis $\f_2\{\rho^n x_s:(s,n)\in T\times\n\}$ for $C$, itself ordered by $\rho^n x_s < \rho^m x_t$ whenever $n > m$, or else $n = m$ and $s < t$. Let $\{\alpha_s:s\in B\}$ be a basis for $H_\ast(C/(\rho))$, indexed by a subset $B\subset T$ with the property that for each $\alpha_s$ there is some $z_s\in C$ with leading term $x_s$ which projects to a cycle representative of $\alpha_s$. Let $B_1\subset B$ be the subset of those $s$ for which $x_s$ is the leading term of some cycle in $C$, and let $B_0 = B\setminus B_1$.

There is then a unique injection $t\colon B_0\rightarrow B$ such that
\[
d(x_s+{<}) = \rho^{r(s)}x_{t(s)}+{<}
\]
for all $s\in B_0$. Here, $r(s)\geq 1$ is an integer uniquely determined by comparing the degrees of $x_s$ and $x_{t(s)}$. Moreover, $t$ restricts to a bijection $t\colon B_0\cong B_1$, and there is an isomorphism
\[
H_\ast C = \bigoplus_{s\in B_0}\f_2[\rho]/(\rho^{r(s)}),
\]
where we may take the summand indexed by $s$ to be generated by any class of the form $\rho^{-r(s)}\cdot d(x_s+{<})$ with leading term $x_{t(s)}$.
\end{lemma}
\begin{proof}
We begin by defining a function $t^{-1}\colon B_1\rightarrow B$. Fix $b\in B_1$; we claim that there exists some $s\in B$ such that $d(x_s+{<}) = x_b+{<}$. The function $t^{-1}$ will then be defined by declaring $t^{-1}(b)$ to be the minimal $s$ for which $d(x_s+{<}) = x_b + {<}$.

Indeed, let $z_b$ be a cycle with leading term $x_b$ which projects to a cycle representative for $\alpha_b$. As $H_\ast C[\rho^{-1}] = 0$, necessarily $\rho^r z_b$ is nullhomologous for some minimal $r\geq 1$. That is, there is some $y\in C$ not divisible by $\rho$ such that $d(y) = \rho^r z_b$. If $y = x_s + {<}$ with $s\in B$, then we are done. Otherwise, as $y$ is a cycle in $C/(\rho)$, necessarily $y$ is homologous to some $x_s+u$ with $u < x_s$ and $s\in B$, in which case there exists some $v$ with $d(v) = x_s+u+y$. We find that
\[
d(x_s+{<}) = d(x_s+u) = d(x_s+u+d(v)) = d(y) = \rho^r z_b = \rho^r x_b+{<}
\]
as claimed. Thus we have produced the function $t^{-1}$.

Next we claim that $t^{-1}$ restricts to a function $t^{-1}\colon B_1\rightarrow B_0$. Indeed, suppose towards contradiction that there were some $b\in B_1$ such that $x_{t^{-1}(b)}$ is the leading term of some cycle. That is to say, there are some $u,v < x_{t^{-1}(b)}$ such that
\[
d(x_{t^{-1}(b)}+u) = x_b+{<},\qquad d(x_{t^{-1}(b)}+v) = 0.
\]
Adding these together, we find
\[
d(u+v) = x_b+{<}.
\]
As $u+v < x_{t^{-1}(b)}$, this contradicts minimality of $t^{-1}(b)$. Thus we have a function $t^{-1}\colon B_1\rightarrow B_0$.

Next we claim that $t^{-1}$ is a bijection. It is a function between locally finite sets, and the assumption that $H_\ast C[\rho^{-1}] = 0$ implies that these sets have the same cardinality in each degree. So it is sufficient to verify that $t^{-1}$ is an injection. Indeed, suppose towards contradiction that there were some $b < c$ in $B_1$ for which $t^{-1}(b) = s = t^{-1}(c)$. Thus there are $u,v < x_s$ such that
\[
d(x_s+u) = x_b+{<},\qquad d(x_s+v) = x_c+{<}.
\]
Adding these together, we find
\[
d(u+v) = x_c+{<}.
\]
As $u+v<x_s$, this contradicts minimality of $t^{-1}(c)$.

By taking the inverse of $t^{-1}\colon B_1\rightarrow B_0$, we have thus proved the existence of a bijection $t\colon B_0\rightarrow B_1$ with the property that $d(x_s+{<}) = x_{t(s)}+{<}$ for all $s\in B_0$. With this $t$, the given description of $H_\ast C$ is clear; in effect, we have described how to choose a basis for $C$ for which $d$ is upper triangular, where if a diagonal entry is divisible by $\rho^r$ so too are all entries above it. Compare the notion of a tag from \cite{Tan85}.

It remains to verify uniqueness. Suppose towards contradiction that we have found some other injection $t'\colon B_0\rightarrow B$ such that $d(x_s+{<}) = x_{t'(s)}+{<}$ for all $s\in B_0$. The condition that $t' \neq t$ means that there exists some $s\in B_0$ for which $d(x_s+{<}) = x_{t'(s)}+{<}$, but $s$ is not minimal among possible $a\in B_0$ with $d(x_a+{<}) = x_{t'(s)}+{<}$. Choose such $s$ with $t'(s)$ maximal, and let $a = t^{-1}(t'(s))$ be the minimal $a\in B_0$ with $d(x_a+{<}) = x_{t'(s)}+{<}$. So there are $u,v < x_a$ for which
\[
d(x_a+u) = x_{t'(s)}+{<},\qquad d(x_a+v) = x_{t'(a)}+{<}.
\]
Adding these together, we find that
\[
d(u+v) = x_{t'(s)}+x_{t'(a)}+{<},
\]
where $u+v < x_a$. If $t'(a) < t'(s)$, then this reduces to
\[
d(u+v) = x_{t'(s)}+{<},
\]
contradicting minimality of $a$. If $t'(s) < t'(a)$, then this reduces to
\[
d(u+v) = x_{t'(a)} + {<},
\]
contradicting maximality of $t'(s)$. So there is no such $t'$, proving that $t$ is the unique injection satisfying the required property.
\end{proof}

We now specialize to the computation of $\Ext_\r$. Observe that by \cref{prop:retract}, we may reduce to considering only the $\rho$-torsion subgroup of $\Ext_\r$. In terms of $\Lambda$, this amounts to ignoring monomials of the form $\lambda_I$ where $I$ is a sequence of odd numbers. We will apply \cref{lem:tags} to compute this $\rho$-torsion subgroup as follows.

We take as basis of $\Lambda/(\rho)$ the standard basis $\lambda_I\tau^n$ where $I$ is coadmissible (\cref{def:coadmissible}) and $n\geq 0$. We also need to order this basis. In the region where we will compute, our choice of order makes no difference, in the sense that all ``error terms'' appearing in ``${}+{<}$'' will be divisible by $\rho$. But for concreteness let us say that $\lambda_I\tau^n < \lambda_J \tau^m$ if $n > m$, or else $n = m$ and $I < J$ lexicographically, i.e.\ if $I = (i_1,\ldots,i_f)$ and $J = (j_1,\ldots,j_f)$, then $i_1 < j_1$, or else $i_1 = j_1$ and $i_2 < j_2$, and so forth.

We must fix some further notation. Let $\{\alpha_s':s\in S_0\}$ be a basis for $\Ext_{\cl}$, and write $\alpha_s\in\Ext_\c$ for the image of $\alpha_s'$ under the map induced by $\thetatilde\colon \Lambda^{\dcl}\rightarrow\Lambda^\c$ (see \cref{prop:retract}). Extend this to a minimal generating set $\{\alpha_s : s\in S\}$ for $\Ext_\c$ as an $\f_2[\tau]$-module. For $s\in S$, let $n_s$ denote the $\tau$-torsion exponent of $\alpha_s$, so that $\{\alpha_s\tau^n:s\in S,\, n< n_s\}$ is an $\f_2$-basis for $\Ext_\c$. For each $s\in S$, choose a distinct coadmissible monomial $\lambda_{I(s)}$ which is the leading term of a cycle representative for $\alpha_s$ in $\Lambda_\c$, making this choice so that if $s\in S_0$ then $\lambda_{I(s)}$ is in the image of $\thetatilde$. See the discussion following \cref{prop:extc} for the particular choices we will take in our computation.

Let $B' = \{(s,n) : s\in S,\, n < n_s\}$. Given $b = (s,n)\in B'$, write $x_b = \lambda_{I(s)}\tau^n\in\Lambda^\r$. Let $B\subset B'$ be the subset of pairs not of the form $(s,0)$ with $s\in S_0$. Let $B_1\subset B$ be the subset of those $b$ such that $x_b$ is the leading term of some cycle, and let $B_0 = B\setminus B_1$. Let $B[f]\subset B$ be the subset of those $b$ for which $x_b$ is in filtration $f$, and extend this notation to all the indexing sets under consideration. 

For our computation, we will produce for every $b\in B_0[f]$ with $f\leq 2$, some $t(b)\in B$ such that
\[
\delta(x_b+{<}) = \rho^{r(b)}x_{t(b)}+{<},
\]
making this choice so that $t\colon B_0\rightarrow B$ is injective. Here, $r(b)\geq 1$ is some integer which may be determined by comparing the stems of $x_b$ and $x_{t(b)}$.

\begin{definition}\label{def:tag}
In the above situation, we shall write $x_b\rightarrow x_{t(b)}\rho^{r(b)}$.
\tqed
\end{definition}

\begin{thm}\label{thm:tags}
Fix notation as above. Then
\begin{enumerate}
\item $t$ is uniquely determined (given our choice of ordered basis);
\item $t$ restricts to bijections $t\colon B_0[f]\cong B_1[f+1]$;
\item The $\rho$-torsion subgroup of $\Ext_\r^{f+1}$ is isomorphic to
\[
\bigoplus_{b\in B_0[f]}\f_2[\rho]/(\rho^{r(b)}),
\]
where the summand corresponding to $b\in B_0[f]$ is generated by any class of the form
\[
\frac{\delta(x_b+{<})}{\rho^{r(b)}}
\]
with leading term $x_{t(b)}$.
\end{enumerate}
\end{thm}
\begin{proof}
This follows by specializing \cref{lem:tags} to the complementary summand of $\thetatilde\colon \Lambda^{\cl}\subset \Lambda$.
\end{proof}

Most notably, the $\rho$-torsion in $\Ext_\r^{f+1}$ is obtained by understanding differentials out of $\Lambda[f]$; this is significantly easier than finding cycles in $\Lambda[f+1]$ directly.

We end with two remarks, which could have been made in the more general context of \cref{lem:tags}.

\begin{remark}\label{rmk:modrho}
More, generally, $H^\ast(\ca^\r/(\rho^m)) = H_\ast(\Lambda/(\rho^m))$ (denoted $\Ext_{(m)}$ in \cref{sec:hopf}) may be read off our computation as follows. For each $b\in B_0$, choose $u_b\in \Lambda$ such that $u_b < x_b$ and $\delta(x_b+u_b) = \rho^{r(b)}x_{t(b)}+{<}$, and let $z_b = \rho^{-r(b)}\cdot \delta(x_b+u_b)$. Then $H_\ast(\Lambda/(\rho^m))$ is given as follows.
\begin{enumerate}
\item For each $s\in S_0$, there is a summand of the form $\f_2[\rho]/(\rho^m)$, generated by the image of $\alpha_s$.
\item For each $x_b\rightarrow\rho^{r(b)}x_{t(b)}$, there is a summand of the form $\f_2[\rho]/(\rho^{\min(m,r(b))})$, generated by the class with cycle representative $z_s$.
\item For each $x_b\rightarrow\rho^{r(b)}x_{t(b)}$, there is a summand of the form $\f_2[\rho]/(\rho^{m-\max(0,m-r(b))})$, generated by the class with cycle representative $\rho^{\max(0,m-r(b))}(x_b+u_b)$.
\tqed
\end{enumerate}
\end{remark}

\begin{remark}\label{rmk:bockstein}
Our approach to computing $\Ext_\r$ via $\Lambda$ is closely related to the computation of $\Ext_\r$ via the $\rho$-Bockstein spectral sequence $\Ext_\c[\rho]\Rightarrow\Ext_\r$ \cite{Hil11}. The precise relation is as follows. For $b = (s,n)\in B$, let $\alpha_b = \alpha_s\tau^n$, so that $\{\alpha_b : b\in B\}$ is a basis of $\Ext_\c$. Our ordering on $\Lambda$ and choice of classes $x_b$ gives $B$ an order, thus making this into an ordered basis of $\Ext_\c$. Now, $x_b\rightarrow \rho^{r(b)}x_{t(b)}$ if and only if $d_{r(b)}(\alpha_b+{<}) = \rho^{r(b)}\alpha_{t(b)}+{<}$ in the $\rho$-Bockstein spectral sequence.
\tqed
\end{remark}

The above discussion describes how we will compute $\Ext_\r^{\leq 3}$ as an $\f_2[\rho]$-module. The computation gives more, as it produces explicit cocycle representatives for our generators of $\Ext_\r^{\leq 3}$. We will use this in \cref{ssec:products} to compute products in $\Ext_\r^{\leq 3}$.

\subsection{\texorpdfstring{$\Ext_\r^f$}{Ext\_R\^f} for \texorpdfstring{$f\leq 3$}{f <= 3}}

We now proceed to the computation. We begin by understanding $\Lambda^\r/(\rho) \cong \Lambda^\c$.

\begin{proposition}\label{prop:extc}
$\Ext_\c^{\leq 3}$ is generated as a commutative $\f_2[\tau]$-algebra by classes $h_a$ for $a\geq 0$, represented in $\Lambda^\c$ by $\lambda_{2^a-1}$, and $c_a$ for $a\geq 0$, represented in $\Lambda^\c$ by $\lambda_{2^a3-1}\lambda_{2^{a+2}-1}^2$. A full set of relations is given by
\[
h_{a+1}h_a=0,\qquad h_{a+2}^2h_a=0,\qquad h_2h_0^2 = \tau h_1^3,\qquad h_{a+3}h_{a+1}^2 = h_{a+2}^3,
\]
for all $a\geq 0$. This is free over $\f_2[\tau]$, with basis given by the classes in the following table.
\begin{longtable}{ll}
\toprule
Class & Constraints \\
\midrule \endhead
\bottomrule \endfoot
$1$ \\
$h_a$ & $a\geq 0$ \\
$h_a\cdot h_b$ & $a\geq b \geq 0$ and $a\neq b+1$ \\
$h_a\cdot h_b\cdot h_c$ & $a\geq b\geq c\geq 0$, where $a \neq b+1$ and $b\neq c+1$,\\
& and where if $b = c$ or $a = b$ then $a\neq c+2$ \\
$c_a$ & $a\geq 0$
\end{longtable}
The only such classes not in the image of $\thetatilde\colon\Ext_{\dcl}\rightarrow\Ext_\c$ are those in which either $h_0$ or $c_0$ appears.
\end{proposition}
\begin{proof}
This is essentially well-known, owing to work of Isaksen on the cohomology of the $\c$-motivic Steenrod algebra \cite{Isa19}. Alternately, one may compute $H_{\leq 3}(\Lambda^\c/(\tau))$ following Wang's approach \cite{Wan67}, and run the $\tau$-Bockstein spectral sequence to recover $\Ext_\c^{\leq 3}$. One finds that $H_{\leq 3}(\Lambda^\c/(\tau))$ agrees with $\Ext_{\cl}^{\leq 3}$, with two exceptions:
\begin{enumerate}
\item Instead of $h_0^2\cdot h_2 = h_1^3$, one has $h_0^2\cdot h_2 = 0$;
\item There is a new cycle $\alpha$ represented by $\lambda_2^2\lambda_1$.
\end{enumerate}
There is a $\tau$-Bockstein differential $d_1(\alpha) = \tau h_1^4$, after which we recover the claimed $\f_2[\tau]$-module basis of $\Ext_\c^{\leq 3}$. The hidden extension $h_0^2\cdot h_2 = \tau h_1^3$ was shown in \cref{prop:h13}; alternately, it is the only relation compatible with $\Sq^0(h_0^2 \cdot h_2) = \tau^2 h_1^2 h_3 = \tau^2h_2^3 = \Sq^0(\tau h_1^3)$.
\end{proof}

\cref{prop:extc} describes a basis for $\Ext_\c^{\leq 3}$, thus giving our set $S[\leq\! 3]$. We must also choose lambda algebra representatives of these classes. We shall choose $c_n$ to be represented by $\lambda_{2^n3-1}\lambda_{2^{n+2}-1}^2$ and a product $h_{n_1}\cdots h_{n_k}$ with $n_1\geq\cdots\geq n_k$ to be represented by $\lambda_{2^{n_1}-1}\cdots\lambda_{2^{n_k}-1}$. We warn that these representatives are not minimal; for example, we have chosen $\lambda_3\lambda_0$ as our representative for $h_2h_0$, rather than the minimal representative which is $\lambda_2\lambda_1$. However, they are easily defined, and convenient enough for our computation.

The following identity will be used frequently in consolidating various cases in our computation.  It is an immediate consequence of the description of $\theta$ given in \cref{Thm:Sq0}.

\begin{lemma}\label{lem:theta0}
We have
\[
\theta^a(\lambda_0\tau^{n}) = \lambda_{2^a-1}\tau^{\floor{2^{a-1}(2n+1)}}+O(\rho^{2^{a-1}}).
\]
for all $n\geq 0$, the error term being omitted when $a=0$.
\qed
\end{lemma}

\begin{remark}
Explicitly, 
\[
\floor{2^{a-1}(2n+1)}  = \begin{cases}2^{a-1}(2n+1)&a\geq 1\\ n&a=0
\end{cases}
\]
This sort of pattern appears frequently throughout our computation, as a consequence of \cref{lem:theta0}.
\tqed
\end{remark}

We now produce the relation ``$\rightarrow$'' described in \cref{def:tag}, proceeding filtration by filtration. To start, observe that $B_0[0] = \{\tau^n:n\geq 1\}$.


\begin{proposition}\label{prop:b0}
We have
\[
\delta(\tau^{{2^a}(2m+1)}) = \lambda_{2^a-1}\tau^{\floor{2^{a-1}(4m+1)}}\rho^{2^a}+O(\rho^{\lceil 2^a+2^{a-1}\rceil})
\]
for all $a,m\geq 0$. In particular,
\[
\tau^{2^a(2m+1)}\rightarrow\lambda_{2^a-1}\tau^{\floor{2^{a-1}(4m+1)}}\rho^{2^a}.
\]
\end{proposition}
\begin{proof}
When $a=0$, as $\tau^2$ is a cycle mod $\rho^2$, we may compute
\[
\delta(\tau^{2m+1}) = \delta(\tau)\tau^{2m}+O(\rho^2) = \lambda_0\tau^{2m}\rho + O(\rho^2),
\]
as claimed. By \cref{lem:theta0}, applying $\theta^a$ for $a\geq 1$ to this yields
\[
\delta(\tau^{2^a(2m+1)}) = (\lambda_{2^a-1}\tau^{2^{a-1}(4m+1)}+O(\rho^{2^{a-1}}))\rho^{2^a}+O(\rho^{2^{a+1}}) =  \lambda_{2^a-1}\tau^{\floor{2^{a-1}(4m+1)}}\rho^{2^a}+O(\rho^{2^a+2^{a-1}}).
\]
Combining the cases $a=0$ and $a\geq 1$ yields the proposition.
\end{proof}

\begin{corollary}
The set $B_0[1]$ consists of those $\lambda_{2^a-1}\tau^n$ such that $n$ is not of the form $2^{a-1}(4m+1)$ for any $m$.
\qed
\end{corollary}

We have located the following indecomposable classes.

\begin{definition}\label{def:b1ind}
For $a,n\geq 0$, we declare 
\[
\tau^{\floor{2^{a-1}(4n+1)}}h_a
\]
to be the class represented by
\[
\rho^{-2^a}\cdot\delta(\tau^{2^a(2n+1)}).
\]
\tqed
\end{definition}

We now compute out of $B_0[1]$.

\begin{proposition}\label{prop:b1}
For the combinations of $a$ and $b$ below, we have $\lambda_{2^b-1}\tau^{2^a(2m+1)}\rightarrow$ the following monomial:

\begin{longtable}{lll}
\toprule
Row & Case & Target \\
\midrule \endhead
\bottomrule \endfoot
(1) & $a<b-1$ or $a=b$ & $\lambda_{2^b-1}\lambda_{2^a-1}\tau^{\floor{2^{a-1}(4m+1)}}\rho^{2^a}$ \\
(2) & $a > b+1$ and $b\neq 0$ & $\lambda_{2^a-1}\lambda_{2^b-1}\tau^{\floor{2^{a-1}(4m+1)}}\rho^{2^a}$ \\
(3) & $a=b-1$ and $m=2n+1$ & $\lambda_{2^b-1}^2\tau^{2^b(4n+1)}\rho^{2^b}$ \\
(4) & $a = b+1$ and $b \neq 0$ & $\lambda_{2^{b+1}-1}^2\tau^{\floor{2^{b-1}(8m+1)}}\rho^{2^b3}$
\end{longtable}
Moreover, these cases are mutually exclusive, and altogether exhaust $B_0[1]$.
\end{proposition}
\begin{proof}
That these cases are mutually exclusive and altogether exhaust $B_0[1]$ is seen by direct inspection. As the monomials arising as targets are $\rho$-multiples of distinct elements of $B[2]$, it suffices to only verify that for each claim of $x\rightarrow y$ we have $\delta(x+{<})=y+{<}$.

(1)~~We have
\[
\delta(\lambda_{2^b-1}\tau^{2^a(2m+1)}) = \lambda_{2^b-1}\lambda_{2^a-1}\tau^{\floor{2^{a-1}(4m+1)}}\rho^{2^a}+O(\rho^{2^a+2^{a-1}}).
\]

(2)~~Note that
\[
\tau^{2^a(2m+1)}\lambda_{2^b-1} = \lambda_{2^b-1}\tau^{2^a(2m+1)}+{<},
\]
as $\tau$ is central mod $\rho$. Now we have
\begin{align*}
\delta(\tau^{2^a(2m+1)}\lambda_{2^b-1}) &= (\lambda_{2^a-1}\tau^{\floor{2^{a-1}(4m+1)}}\rho^{2^a}+O(\rho^{2^a+2^{a-1}}))\lambda_{2^b-1} \\
&= \lambda_{2^a-1}\lambda_{2^b-1}\tau^{\floor{2^{a-1}(4m+1)}}\rho^{2^a}+O(\rho^{2^a+1}).
\end{align*}

(3)~~Note that
\[
\theta^b(\lambda_0\tau^{2n+1})=\lambda_{2^b-1}\tau^{\floor{2^{b-1}(4n+3)}} +O(\rho).
\]
Now we have
\[
\delta(\theta^b(\lambda_0\tau^{2n+1})) = \theta^b(\delta(\lambda_0\tau^{2n+1}))=\theta^b(\lambda_0^2\tau^{2n}\rho + O(\rho^2)) = \lambda_{2^b-1}^2\tau^{2^b(4n+1)}\rho^{2^b}+O(\rho^{2^b+1}).
\]

(4)~~We have
\begin{align*}
\delta(\lambda_{2^b-1}\tau^{2^{b+1}(2m+1)}) &= \delta(\theta^{b-1}(\lambda_1\tau^{8m+4})) = \theta^{b-1}(\lambda_1\delta(\tau^4)\tau^{8m}+O(\rho^8)) \\
&= \theta^{b-1}(\lambda_3^2\tau^{8m+1}\rho^6+O(\rho^7)) = \lambda_{2^{b+1}-1}^2\tau^{2^{b-1}(8m+1)}\rho^{2^b3}+O(\rho^{2^b3+2^{b-1}}).
\end{align*}
Here, the third equality uses the Adem relations $\lambda_1\lambda_3=0$ and $\lambda_1\lambda_5=\lambda_3\lambda_3$ to determine the leading term of $\lambda_1\delta(\tau^4)$.
\end{proof}

\begin{corollary}
The set $B_0[2]$ consists of those $\lambda_{2^b-1}\lambda_{2^c-1}\tau^n$ where $b=c$ or $b\geq c+2$, and where moreover
\begin{enumerate}
\item $n \neq \floor{2^{b-1}(4m+1)}$ and $n\neq \floor{2^{c-1}(4m+1)}$ for any $m$;
\item If $b = c = 0$, then $n$ is odd;
\item If $b = c \geq 1$, then $n \neq 2^b(4m+1)$ for any $m$;
\item If $b = c \geq 2$, then $n \neq 2^{b-2}(8m+1)$ for any $m$.
\qed
\end{enumerate}
\end{corollary}

We have located the following indecomposable classes.

\begin{definition}\label{def:b2ind}
For $a,n\geq 0$, we declare
\[
\tau^{2^a(8n+1)}h_{a+2}^2
\]
to be the class represented by
\[
\rho^{-2^{a+1}3}\cdot\delta(\lambda_{2^{a+1}-1}\tau^{2^{a+2}(2n+1)})
\]
\tqed
\end{definition}

We now compute out of $B_0[2]$.

\begin{proposition}\label{prop:b2}
For $b = c$ or $b\geq c+2$, we have $\lambda_{2^b-1}\lambda_{2^c-1}\tau^{2^a(2m+1)}\rightarrow$ the following monomial.
\begin{longtable}{lll}
\toprule
\# & Case & Target \\
\midrule \endhead
\bottomrule \endfoot
(1) & $b=c=0$, $a=-1$, $m=2n+1$ & $\lambda_0^3\tau^{2n}\rho$ \\
(2) & $b=c\geq 1$, $a=b-1$, $m=2n+1$ & $\lambda_{2^b-1}^3\tau^{2^b(2n+1)}\rho^{2^b}$ \\
(3) & $b = c \geq 0$, $a = c$, $m = 2n+1$ & $\lambda_{2^b-1}^3\tau^{\floor{2^{b-1}(4(2n+1)+1)}}$ \\
(4) & $b=c\geq 1$, $a=b+1$ & $\lambda_{2^{b-1}3-1}\lambda_{2^{b+1}-1}^2\tau^{\floor{2^{b-2}(16m+1)}}\rho^{2^{b-1}7}$ \\
(5) & $b=c\geq 1$, $a=b+2$ & $\lambda_{2^{b+2}-1}^2\lambda_{2^{b-1}-1}\tau^{\floor{2^{b-2}(2(16m+1)+1)}}\rho^{2^{b-1}13}$ \\
(6) & $b=c\geq 1$, $a\geq b+3$ & $\lambda_{2^a-1}\lambda_{2^b-1}^2\tau^{\floor{2^{a-1}(4m+1)}}\rho^{2^a}$ \\
(7) & $b=c\geq 2$, $a=b-2$, $m=4n+2$ & $\lambda_{2^b-1}^3\tau^{\floor{2^{b-2}(2(4n+1)+1)}}\rho^{2^b}$ \\
(8) & $b=c\geq 2$, $a=b-2$, $m=2n+1$ & $\lambda_{2^{b-2}3-1}\lambda_{2^b-1}^2\tau^{\floor{2^{b-3}(2(4n+1)+1)}}\rho^{2^{b-2}3}$ \\
(9) & $b=c\geq 3$, $a\leq b-3$ & $\lambda_{2^b-1}^2\lambda_{2^a-1}\tau^{\floor{2^{a-1}(4m+1)}}\rho^{2^a}$ \\
(10) & $b-2\geq c = 0$, $a=0$, $m=2n+1$ & $\lambda_{2^b-1}\lambda_0^2\tau^{2n}\rho$ \\
(11) & $b-2=c\geq 1$, $a = b$ & $\tau^{2^c(8n+1)}\lambda_{2^c3-1}\lambda_{2^{c+2}-1}^2\rho^{2^{c+1}3}$ \\
(12) & $b-2=c\geq 1$, $a\geq b+2$ & $\lambda_{2^a-1}\lambda_{2^b-1}\lambda_{2^c-1}\tau^{\floor{2^{a-1}(4m+1)}}\rho^{2^a}$ \\
(13) & $b-3\geq c\geq 1$, $a\geq b$, $a\neq b+1$ & $\lambda_{2^a-1}\lambda_{2^b-1}\lambda_{2^c-1}\tau^{\floor{2^{a-1}(4m+1)}}\rho^{2^a}$ \\
(14) & $b-2\geq c \geq 1$, $c \leq a < b$, $a\notin \{c+1,b-1\}$ & $\lambda_{2^b-1}\lambda_{2^a-1}\lambda_{2^c-1}\tau^{\floor{2^{a-1}(4m+1)}}\rho^{2^a}$ \\
(15) & $b-2=c\geq 1$, $a=c-1$, $m=2n+1$ & $\lambda_{2^{c+1}-1}^3\tau^{2^c(2n+1)}\rho^{2^c}$ \\
(16) & $b-3\geq c\geq 1$, $a=c-1$, $m=2n+1$ & $\lambda_{2^b-1}\lambda_{2^c-1}^2\tau^{2^c(2n+1)}\rho^{2^c}$ \\
(17) & $b-2 = c \geq 1$, $a=c+1$, $m = 2n+1$ & $\lambda_{2^b-1}^3\tau^{\floor{2^{b-3}(4(4n+1)+1)}}\rho^{2^{b-2}7}$ \\
(18) & $b-3=c\geq 1$, $a=c+1$ & $\lambda_{2^{c+2}-1}^3\tau^{\floor{2^{c-1}(8m+1)}}\rho^{2^c3}$ \\
(19) & $b-4\geq c \geq 1$, $a=c+1$ & $\lambda_{2^b-1}\lambda_{2^{c+1}-1}^2\tau^{\floor{2^{c-1}(8m+1)}}\rho^{2^c3}$ \\
(20) & $b-3\geq c \geq 1$, $a=b-1$, $m=2n+1$ & $\lambda_{2^b-1}^2\lambda_{2^c-1}\tau^{2^b(2n+1)}\rho^{2^b}$ \\
(21) & $b-2\geq c\geq 1$, $a=b+1$ & $\lambda_{2^{b+1}-1}^2\lambda_{2^c-1}\tau^{\floor{2^{b-1}(8m+1)}}\rho^{2^b3}$ \\
(22) & $b-2\geq c \geq 2$, $a\leq c-2$ & $\lambda_{2^b-1}\lambda_{2^c-1}\lambda_{2^a-1}\tau^{\floor{2^{a-1}(4m+1)}}\rho^{2^a}$
\end{longtable}
Moreover, these cases are mutually exclusive, and altogether exhaust $B_0[2]$.
\end{proposition}

\begin{proof}
That these cases are mutually exclusive and altogether exhaust $B_0[2]$ is seen by direct inspection. As the monomials arising as targets are $\rho$-multiples of distinct elements of $B[3]$, it suffices to only verify that for each claim of $x\rightarrow y$ we have $\delta(x+{<})=y+{<}$. 

Each case represents a collection of families of monomials whose leading terms are connected by $\theta$. Thus we may always reduce to the smallest possible $c$, with the exception of (9) and (22), where doing so would place extra constraints on $a$. In addition, by working modulo the smallest power of $\rho$ in which the proposed target does not vanish, we may always reduce to the smallest possible $m$.

We may further divide the list of cases provided into three types: those which require no calculations beyond those carried out in \cref{prop:b1}, cases (15) and (18), and the more interesting cases which do require additional calculation, producing new indecomposable classes in $\Ext_\r^3$. Here, cases (15) and (18) are not really exceptional; they could be consolidated into cases (16) and (19), only this would require slightly modifying the setup of \cref{ssec:extpreliminaries}, and it is easier to just separate them out. The more interesting cases are (4), (5), (8), (11), and (17). The remaining less interesting cases may all be handled exactly the same way as the first two cases of \cref{prop:b1} were handled. Thus we shall not handle them individually, and instead only illustrate this point with a verification of (21).
With these reductions in place, the proposition is proved by the following calculations.

(4)~~Here, we are claiming $\delta(\lambda_1^2\tau^4+{<})=\lambda_2\lambda_3^2\rho^7+{<}$. In fact $\delta(\lambda_1^2\tau^4) = \lambda_2\lambda_3^2\rho^7$ on the nose.

(5)~~Here, we are claiming $\delta(\lambda_1^2\tau^8+{<})= \lambda_7^2\lambda_0\tau\rho^{13}+{<}$. Observe that $\delta(\lambda_1^2\tau^8) = \lambda_3^3\tau^4\rho^8+O(\rho^{12})$, but $\lambda_3\tau^4\rho^4$ is already seen as a target in case (1). Thus some additional correction term must be added to $\lambda_1^2\tau^8$ to get down to $\lambda_7^2\lambda_0\tau\rho^{13}$. Such a correction term is given by
\begin{align*}
u = \lambda_3^2\tau^6\rho^4&+\lambda_3\lambda_5\tau^5\rho^6+\lambda_3\lambda_6\tau^4\rho^7+\lambda_5\lambda_7\tau^3\rho^{10}\\
&+(\lambda_5\lambda_8+\lambda_6\lambda_7)\tau^2\rho^{11}+(\lambda_{11}\lambda_3+\lambda_5\lambda_9)\tau^2\rho^{12}+(\lambda_8\lambda_7+\lambda_7\lambda_8+\lambda_6\lambda_9)\tau\rho^{13};
\end{align*}
with this choice of $u$, we have $\delta(\lambda_1^2\tau^8+u) =  \lambda_7^2\lambda_0\tau\rho^{13} + O(\rho^{14})$.

(8)~~Here, we are claiming $\delta(\lambda_3^2\tau^3+{<})=\lambda_2\lambda_3^2\tau\rho^3+{<}$. Indeed, let
\[
u = (\lambda_3\lambda_4+\lambda_4\lambda_3)\tau^2\rho+\lambda_3\lambda_5\tau^2\rho^2+\lambda_4\lambda_5\tau\rho^3;
\]
then we have $\delta(\lambda_3^2\tau^3+u) = \lambda_2\lambda_3^2\tau\rho^3+O(\rho^4)$.

(11)~~Here, we are claiming $\delta(\lambda_7\lambda_1\tau^8+{<}) = \lambda_5\lambda_7^2\tau^2\rho^{12}$. Indeed, let
\[
u = \lambda_9\lambda_7\tau^4\rho^8+\lambda_9\lambda_{11}\tau^2\rho^{12};
\]
then we have $\delta(\lambda_7\lambda_1\tau^8+u)=\lambda_5\lambda_7^2\tau^2\rho^{12}+O(\rho^{14})$.

(15)~~Here, we are claiming $\delta(\lambda_7\lambda_1\tau^3+{<})=\lambda_3^3\tau^2\rho^2+{<}$. Indeed, let
\[
u = \lambda_7\lambda_2\tau^2\rho+(\lambda_9\lambda_1+\lambda_5\lambda_5)\tau^2\rho^2;
\]
then we have $\delta(\lambda_7\lambda_1\tau^3+{<})=\lambda_3^3\tau^2\rho^2+O(\rho^3)$.

(17)~~Here, we are claiming $\delta(\lambda_7\lambda_1\tau^{12}+{<})=\lambda_7^3\tau^5\rho^{14}+{<}$. Indeed, let
\[
u=\lambda_{11}\lambda_3\tau^9\rho^6+(\lambda_{11}\lambda_4+\lambda_{12}\lambda_3+\lambda_7\lambda_8+\lambda_8\lambda_7)\tau^8\rho^7+\lambda_7^2\tau^9\rho^6+\lambda_9\lambda_7\tau^8\rho^8;
\]
then we have $\delta(\lambda_7\lambda_1\tau^{12}+u) = \lambda_7^3\tau^5\rho^{14}+O(\rho^{15})$.

(18)~~Here, we are claiming $\delta(\lambda_{15}\lambda_1\tau^4+{<})=\lambda_7^3\tau\rho^6+{<}$. Indeed, let
\[
u = (\lambda_{19}\lambda_3+\lambda_{11}\lambda_{11})\tau\rho^6;
\]
then we have $\delta(\lambda_{15}\lambda_1\tau^4+u)=\lambda_7^3\tau\rho^6+O(\rho^7)$.

(21)~~Here, we are claiming $\delta(\lambda_{2^b-1}\lambda_{2^c-1}\tau^{2^{b+1}(2m+1)}+{<}) = \lambda_{2^{b+1}-1}^2\lambda_{2^c-1}\tau^{2^{b-1}(8m+1)}\rho^{2^b3}+{<}$, at least provided $b-2\geq c \geq 1$. This case is intended to illustrate all the remaining cases, and is identical in form to case (2) of \cref{prop:b1}. Recall from \cref{prop:b1} that
\[
\delta(\lambda_{2^b-1}\tau^{2^{b+1}(2m+1)}+O(\rho)) = \lambda_{2^{b+1}-1}^2\tau^{2^{b-1}(8m+1)}\rho^{2^b3}+O(\rho^{2^b3+1}).
\]
As
\[
\lambda_{2^b-1}\lambda_{2^c-1}\tau^{2^{b+1}(2m+1)} \equiv \lambda_{2^b-1}\tau^{2^{b+1}(2m+1)}\lambda_{2^c-1}\pmod{\rho},
\]
it follows that
\begin{align*}
\delta(\lambda_{2^b-1}\lambda_{2^c-1}\tau^{2^{b+1}(2m+1)}+O(\rho)) &= \delta(\lambda_{2^b-1}\tau^{2^{b+1}(2m+1)}\lambda_{2^c-1}+O(\rho))\\
&=(\lambda_{2^{b+1}-1}^2\tau^{2^{b-1}(8m+1)}\rho^{2^b3}+O(\rho^{2^b3+1}))\lambda_{2^c-1}\\
&=\lambda_{2^{b+1}-1}^2\lambda_{2^c-1}\tau^{2^{b-1}(8m+1)}\rho^{2^b3}+O(\rho^{2^b3+1}),
\end{align*}
which gives the desired relation. The remaining cases are either identical in form to this, or simpler in that they do not require one to first move $\tau$ around to reduce to a case already considered in \cref{prop:b1}.
\end{proof}

This produces the indecomposable classes
\begin{gather*}
\tau^{2^{a-1}(2(16n+1)+1)}h_{a+3}^2h_a,\qquad\tau^{2^a(4(4n+1)+1)}h_{a+3}^3,\\
\tau^{2^{a-1}(16n+1)}c_a,\qquad\tau^{2^{a+1}(8n+1)}c_{a+1},\qquad\tau^{2^{a-1}(2(4n+1)+1)}c_a,
\end{gather*}
for $a,n\geq 0$, following the same recipe as employed in \cref{def:b1ind} and \cref{def:b2ind}, only where one must employ $\theta$-iterates of $\tau$-multiples of the correction terms $u$ given in \cref{prop:b2}.

\cref{prop:b2} concludes the work necessary for our computation of the $\f_2[\rho]$-module structure of $\Ext_\r^{\leq 3}$. Let us now summarize in one theorem what we have learned. We wish to give a minimal generating set of $\Ext_\r^{\leq 3}$ whose elements are products of the indecomposable classes we have found. Before doing so, let us treat the following subtlety.

By way of example, let $x = \tfrac{1}{\rho^{2^b}}\delta(\lambda_{2^b-1}\tau^{2^{b-1}(4m+3)})$ with $b\geq 1$, and let $\alpha\in\Ext_\r$ be the class represented by $x$. Our computation in \cref{prop:b1} combined with the recipe of \cref{thm:tags} would yield $\alpha$ as an element of a minimal generating set for $\Ext_\r$. Observe that $x$ has leading term $\lambda_{2^b-1}^2\tau^{2^b(4m+1)}$. It follows quickly from this that $x$ has the same leading term as the cocycle representative of $(\tau^{2^{b-1}(4m+1)}h_b)^2$ given by the product of those cocycle representatives for $\tau^{2^{b-1}(4m+1)}h_b$ given in \cref{def:b1ind}. However, this does not prove that $\alpha = (\tau^{2^{b-1}(4m+1)}h_b)^2$: we have not ruled out the possibility that $\alpha + \beta= (\tau^{2^{b-1}(4m+1)}h_b)^2$ for some nonzero $\beta$ represented by a cycle $y < \lambda_{2^b-1}^2\tau^{2^b(4m+1)}$. This is still sufficient to deduce that we may, if necessary, replace $\alpha$ with $\alpha + \beta$ in our minimal generating set in order to obtain a minimal generating set built as products of indecomposables. It turns out that no such correction is necessary.

\begin{lemma}\label{lem:unambiguousproduct}
Write $\phi\colon\Ext_\r\rightarrow\Ext_\c$ for the quotient. Fix classes $\alpha,\beta$ in $\Ext_\r^1$ or $\Ext_\r^2$, at least one of which is $\rho$-torsion, and not both in $\Ext_\r^2$. Let $r$ be minimal for which $\rho^r\alpha=0$ or $\rho^r\beta=0$. Fix $\gamma\in\Ext_\r^{\leq 3}$ not divisible by $\rho$ and such that $\rho^r\gamma=0$, and suppose $\phi(\alpha)\cdot\phi(\beta)=\phi(\gamma)$. Then $\alpha\cdot\beta = \gamma$.
\end{lemma}
\begin{proof}
Under the given conditions, there is in fact a unique class in the degree of $\alpha\cdot\beta$ which is not divisible by $\rho$ and is killed by $\rho^r$. This may be seen by direct inspection of the propositions preceding this. 
\end{proof}

We may now state the main theorem of this section.

\begin{thm}\label{thm:extadd}
(1)~~A minimal multiplicative generating set for $\Ext_\r^{\leq 3}$ as an $\f_2[\rho]$-algebra is given by the given by the classes in the following table:
\begin{longtable}{ll}
\toprule
Multiplicative generator & $\rho$-torsion exponent \\
\midrule \endhead
\bottomrule \endfoot
$h_{a+1}$ & $\infty$ \\
$c_{a+1}$ & $\infty$ \\
$\tau^{\floor{2^{a-1}(4n+1)}}h_a$ & $2^a$ \\
$\tau^{2^a(8n+1)}h_{a+2}^2$ & $2^{a+1}\cdot 3$ \\
$\tau^{\floor{2^{a-1}(2(16n+1)+1)}}h_{a+3}^2h_a$ & $2^a\cdot 13$ \\
$\tau^{2^a(4(4n+1)+1)}h_{a+3}^3$ & $2^a\cdot 7$ \\
$\tau^{\floor{2^{a-1}(16n+1)}}c_a$ & $2^a\cdot 7$ \\
$\tau^{2^{a+1}(8n+1)}c_{a+1}$ & $2^{a+2}\cdot 3$ \\
$\tau^{\floor{2^{a-1}(2(4n+1)+1)}}c_a$ & $2^a\cdot 3$
\end{longtable}
Here, $a,n\geq 0$, and the $\rho$-torsion exponent of a class $\alpha$ is the minimal $r$ for which $\rho^r \alpha = 0$; the classes $h_{a+1}$ and $c_{a+1}$ are $\rho$-torsion-free.

(2)~~The operation $\Sq^0$ acts on these classes by incrementing $a$ in each row.

(3)~~The image of these classes under $\Ext_\r\rightarrow\Ext_\c$ is as their name suggests.

(4)~~A minimal $\f_2[\rho]$-module generating set for $\Ext_\r^{\leq 3}$ is given in the following table. In all cases, the $\rho$-torsion exponent of a given class is the minimal $\rho$-torsion exponent of the multiplicative generators it is written as a product of.
\begin{longtable}{ll}
\toprule
$\f_2[\rho]$-module generator & Constraints\\
\midrule \endhead
\bottomrule \endfoot
$1$ \\
\midrule
$h_a$ & $a\geq 1$ \\
$\tau^{\floor{2^{a-1}(4n+1)}}h_a$ & $a,n\geq 0$ \\
\midrule
$h_a\cdot h_b$ & $a\geq b\geq 1$ and $a \neq b+1$ \\
$h_a\cdot\tau^{\floor{2^{b-1}(4n+1)}}h_b$ & $a\geq 1$ and $b,n\geq 0$, and $a\neq b\pm 1 $ \\
$\tau^{\floor{2^{a-1}}}h_a\cdot\tau^{\floor{2^{a-1}(4n+1)}}h_a$ & $a,n\geq 0$ \\
$\tau^{2^a(8n+1)}h_{a+2}^2$ & $a,n\geq 0$ \\
\midrule
$h_a\cdot h_b\cdot h_c$ & $a\geq b\geq c\geq 1$, where $a \neq b+1$ and $b\neq c+1$,\\
& and where if $b = c$ or $a = b$ then $a\neq c+2$ \\
$h_a\cdot h_b\cdot\tau^{\floor{2^{c-1}(4n+1)}}h_c$ & $a\geq b \geq 1$ and $c,n\geq 0$, with $a\neq b+1$ and $c\notin\{a\pm 1,b\pm 1\}$, and\\
& if $a=b$ then $c\notin\{a-2,a,a+2\}$, and if $a=b+2$ then $c\neq a$ \\
$h_a\cdot \tau^{\floor{2^{b-1}}}h_b\cdot\tau^{\floor{2^{b-1}(2n+1)}}h_b$ & $a\geq 1$ and $b,n\geq 0$, and $a\notin\{b-2,b-1,b+1\}$\\
$h_0\cdot h_0\cdot \tau^{2n}h_0$ & $n\geq 0$ \\
$h_a\cdot\tau^{2^b(8n+1)}h_{b+2}^2$ & $a\geq 1$ and $b,n\geq 0$, and either $a\leq  b-1$ or $a\geq b+4$\\

$\tau^{\floor{2^{a-1}(2(16n+1)+1)}}h_{a+3}^2h_a$ & $a,n\geq 0$ \\
$\tau^{2^a(4(4n+1)+1)}h_{a+3}^3$ & $a,n\geq 0$ \\
$c_a$ & $a\geq 1$ \\
$\tau^{\floor{2^{a-1}(16n+1)}}c_a$ & $a,n\geq 0$ \\
$\tau^{2^{a+1}(8n+1)}c_{a+1}$ & $a,n\geq 0$ \\
$\tau^{\floor{2^{a-1}(2(4n+1)+1)}}c_a$ & $a,n\geq 0$
\end{longtable}
\end{thm}

\begin{proof}
All of this may be read off the preceding computations, using \cref{lem:unambiguousproduct} with \cref{prop:extc} if necessary to write a given class as a product of classes in the given generating set.
\end{proof}

We point out the following corollary.

\begin{corollary}
The operation $\rho\cdot\Sq^0$ is injective on $\Ext_\r^{\leq 3}$.
\qed
\end{corollary}

\begin{remark}
As indicated in \cref{rmk:bockstein}, one may also read off our computation a description of all differentials in the $\rho$-Bockstein spectral sequence $\Ext_\c[\rho]\Rightarrow\Ext_\r$ emanating out of filtration at most $2$. We leave this to the interested reader.
\tqed
\end{remark}

\subsection{Multiplicative structure}\label{ssec:products}

We now compute the multiplicative structure of $\Ext_\r^{\leq 3}$. This material is mostly not needed for our study of the $1$-line of the motivic Adams spectral sequence in \cref{sec:hopf}; the exception is that we will use the relation \cref{prop:3linerelations}(4) in the proof of \cref{thm:r}.

Already \cref{lem:unambiguousproduct} produces a large number of relations. For example, it implies that we may always shift powers of $\tau$ around in products that do not vanish in $\Ext_\c$, provided it makes sense to do so, yielding relations such as
\[
\tau^{\floor{2^{a-1}(4n+1)}}h_a\cdot \tau^{\floor{2^{b-1}(4m+1)}}h_b = h_a\cdot \tau^{\floor{2^{b-1}(4(m+2^{a-b-2}(4n+1))+1)}}h_b
\]
for $a\geq b+2$. These were implicitly used in the proof of \cref{thm:extadd}. The condition that the product does not vanish in $\Ext_\c$ is necessary; see \cref{ex:shifttau} below.

We are left only with relations that would be realized as hidden extensions in the $\rho$-Bockstein spectral sequence. These arise from the possible failure of the relations $h_{a+1}h_a=0$ and $h_{a+2}^2h_a=0$ to lift through $\Ext_\r\rightarrow\Ext_\c$.

\begin{remark}
The following computations will involve some explicit calculations with cocycle representatives. For ease of reference, we collect some important cocycle representatives here.
\begin{longtable}{ll}
\toprule
Class & Cocycle representative \\
\midrule \endhead
\bottomrule \endfoot
$h_0$ & $\lambda_0$ \\
$h_1$ & $\lambda_1$ \\
$h_2$ & $\lambda_3$ \\
$h_3$ & $\lambda_7$ \\
$c_0$ & $\lambda_2\lambda_3^2$ \\
$c_1$ & $\lambda_5\lambda_7^2$ \\
$\tau^{2^{a}}h_{a+1}$ & $\rho^{-2^{a+1}}\cdot\delta(\tau^{2^{a+1}}) = \theta^{a+1}(\lambda_0)=\tau^{2^{a}}\lambda_{2^{a+1}-1} = \lambda_{2^{a+1}-1}\tau^{2^{a}}+ O(\rho^{2^{a}})$ \\
$\tau^2 h_0$ & $\lambda_0\tau^2+\lambda_1\tau^2\rho+\lambda_3\tau\rho^3+\lambda_4\rho^4$ \\
$\tau^4 h_0$ & $\lambda_0\tau^4+\lambda_3\tau^3\rho^3+\lambda_5\tau^2\rho^5+\lambda_7\tau\rho^7+\lambda_8\rho^8$ \\
$\tau h_2^2$ & $\lambda_3^2\tau+(\lambda_3\lambda_4+\lambda_4\lambda_3)\rho = \tau \lambda_3^2$ \\
$\tau^9 h_2^2$ & $\lambda_3^2\tau^9+(\lambda_3\lambda_4+\lambda_4\lambda_3)\tau^8\rho+\lambda_5\lambda_3\tau^8\rho^2+O(\rho^{10})$
\end{longtable}
We will use these without further comment.
\tqed
\end{remark}

We begin with some products in $\Ext_\r^{\leq 3}$ which lift the relation $h_{a+1}h_a=0$.

\begin{proposition}\label{prop:2linerelations}
The following hold.
\begin{enumerate}
\item $h_{a+1}\cdot\tau^{\floor{2^{a-1}(4(2n+1)+1)}}h_a = \rho^{2^a}\cdot\tau^{2^{a}}h_{a+1}\cdot\tau^{2^{a}(4n+1)}h_{a+1}$;
\item $h_{a+1}\cdot\tau^{\floor{2^{a-1}(8n+1)}}h_a = 0$;
\item $\tau^{2^{a+1}(4n+1)}h_{a+2}\cdot h_{a+1} = \rho^{2^{a+1}}\cdot\tau^{2^a(8n+1)}h_{a+2}^2$;
\item $\tau^{2^a(8n+1)}h_{a+2}^2\cdot h_{a+1} = \rho^{2^a}\cdot \tau^{\floor{2^{a-1}(16n+1)}}c_a$;
\item $\tau^{2^a(8n+1)}h_{a+2}^2\cdot\tau^{2^a(4m+1)}h_{a+1} = \rho^{2^a}\cdot \tau^{\floor{2^{a-1}(2(4(m+2n))+1)}}c_a$;
\item $h_{a+3}\cdot\tau^{2^a(16n+1)}h_{a+2}^2 = 0$;
\item $h_{a+3}\cdot\tau^{2^a(8(2n+1)+1)}h_{a+2}^2 = \rho^{2^{a+3}}\cdot\tau^{2^a(4(4n+1)+1)}h_{a+3}^2$.
\end{enumerate}
\end{proposition}
\begin{proof}
In each of these, we may use $\Sq^0$ to reduce to the case $a=0$. In all cases where the product does not vanish, the claimed value of the product is the unique nonzero class in its degree which is both $\rho$-torsion and divisible by $\rho$, so it suffices to verify the product working modulo the smallest power of $\rho$ in which the claimed value does not vanish. In doing so, we may in each case reduce to $n=m=0$. With these reductions in place, the proposition is proved by the following computations.

(1)~~Here, we are claiming $h_1\cdot \tau^2 h_0 = \rho \cdot \tau h_1\cdot \tau h_1$. Indeed, we may compute
\begin{align*}
(\lambda_0\tau^2 + \lambda_1\tau^2\rho+\lambda_3\tau\rho^3+\lambda_4\rho^4)\cdot \lambda_1 &= \lambda_1^2\tau^2\rho+\lambda_2\lambda_1\tau\rho^2+\lambda_2\lambda_2\rho^3+\lambda_2\lambda_3\rho^4\\
&= \rho(\lambda_1\tau+\lambda_2\rho)^2 = \rho \theta(\lambda_0)^2,
\end{align*}
which represents $\rho\cdot\tau h_1\cdot\tau h_1$.

(2)~~There are no nonzero $\rho$-torsion classes in this degree, so the product must vanish.

(3)~~Here, we are claiming $h_1\cdot\tau^2 h_2 = \rho^2 \cdot \tau h_2^2$. Indeed, we may compute
\[
\lambda_1\cdot \tau^2\lambda_3 = \rho^2\cdot\tau \lambda_3^2
\]
on the nose.

(4)~~Here, we are claiming $h_1\cdot\tau h_2^2 = \rho\cdot c_0$. Indeed, we may compute
\[
\lambda_1\cdot \tau\lambda_3^2 = \lambda_2\lambda_3^2\rho
\]
on the nose.

(5)~~Here, we are claiming $\tau h_1\cdot \tau h_2^2 = \rho\cdot \tau c_0$. For this, it suffices to work mod $\rho^2$. Here we may compute
\[
\tau \lambda_1\cdot \tau \lambda_2^2 = \rho\cdot\lambda_2\lambda_3^2\tau + O(\rho^2),
\]
and the claim follows.

(6)~~Here, we have reduced to $a=0$ but not yet to $n=0$. The only nonzero $\rho$-torsion class in this degree is $\rho^6\tau^{16n+1}c_1$, so it suffices to work mod $\rho^7$. In doing so, we may now reduce to $n=0$. Indeed, we have
\[
\tau\lambda_3^2\cdot\lambda_7 = 0,
\]
and the claim follows.

(7)~~Here, we are claiming $h_3\cdot \tau^9 h_2^2 = \rho^8\cdot\tau^5 h_3^3$. For this, it suffices to work mod $\rho^9$. Here, we may compute
\[
(\lambda_3^2\tau^9+(\lambda_4\lambda_3+\lambda_3\lambda_4)\tau^8\rho+\lambda_5\lambda_3\tau^8\rho^2)\cdot \lambda_7 = \lambda_7^3\tau^5\rho^8+O(\rho^9),
\]
yielding the claim.
\end{proof}

\begin{example}\label{ex:shifttau}
We have
\[
\tau^2 h_2\cdot h_1^2 = \rho^3 c_0,\qquad h_2\cdot (\tau h_1)^2 = 0.
\]
This serves as a warning that one cannot in general freely shift around powers of $\tau$ in products.
\tqed
\end{example}

We now give some products that lift the relation $h_{a+2}^2h_a=0$.

\begin{proposition}\label{prop:3linerelations}
The following hold.
\begin{enumerate}
\item $h_{a+2}^2\cdot\tau^{\floor{2^{a-1}(16n+1)}}h_a = 0$;
\item $h_{a+2}^2\cdot\tau^{\floor{2^{a-1}(4(2n+1)+1)}}h_a = \rho^{2^{a+1}}\cdot\tau^{\floor{2^{a-1}(2(4n+1)+1)}}c_a$;
\item $h_{a+2}^2\cdot\tau^{\floor{2^{a-1}(8(2n+1)+1)}}h_a = \rho^{2^a3}\cdot\tau^{2^{a+1}}h_{a+2}\cdot\tau^{2^a(8n+1)}h_{a+2}^2$;
\item $\tau^{2^{a+2}}h_{a+3}\cdot\tau^{2^{a+2}(4n+1)}h_{a+3}\cdot h_{a+1} = \rho^{2^{a+1}3}\cdot\tau^{2^{a}(4(4n+1)+1)}h_{a+3}^3$;
\item $h_{a+1}\cdot h_{a+3}\cdot\tau^{2^{a+2}(4n+1)}h_{a+3} = \rho^{2^{a+2}}\cdot\tau^{2^{a+1}(8n+1)}c_{a+1}$;
\item $\tau^{2^{a+1}(8n+1)}h_{a+3}^2\cdot h_{a+1} = 0$.
\end{enumerate}
\end{proposition}
\begin{proof}
As in the proof of \cref{prop:2linerelations}, we may use $\Sq^0$ to reduce to the case $a=0$, and in all cases where the product does not vanish may reduce to $n=0$.
With these reductions in place, the proposition is proved by the following computations.

(1)~~There are no nonzero $\rho$-torsion classes in this degree, so the product must vanish.

(2)~~Here, we are claiming $h_2^2\cdot\tau^2 h_0 = \rho^2\cdot\tau c_0$. For this, it suffices to work mod $\rho^3$. Recall that $\tau^2 h_0$ is represented by $\lambda_0\tau^2+\lambda_1\tau^2\rho+O(\rho^3)$. We may compute
\[
(\lambda_0\tau^2+\lambda_1\tau^2\rho)\cdot\lambda_3^2 = \rho^2\cdot\lambda_2\lambda_3^2\tau + O(\rho^3),
\]
and the claim follows.

(3)~~Here, we are claiming $h_2^2\cdot\tau^4h_0 = \rho^3\cdot \tau^2h_2\cdot\tau h_2^2$. For this, it suffices to work mod $\rho^4$. Observe that
\[
h_2\cdot h_2\cdot\tau^4 h_0 = h_2\cdot \tau^2 h_2\cdot\tau^2 h_0 = \tau^2h_2\cdot h_2\cdot\tau^2 h_0 = \tau^2h_2\cdot\tau^2h_2\cdot h_0
\]
by \cref{lem:unambiguousproduct}. We may now compute
\[
\lambda_0\cdot\tau^2\lambda_3\cdot\tau^2\lambda_3 =\rho^3\cdot\lambda_3^3\tau^3+O(\rho^4),
\]
yielding the claim.

(4)~~Here, we are claiming $\tau^4h_3\cdot\tau^4h_3\cdot h_1 = \rho^6\cdot\tau^5 h_3^3$. For this, it suffices to work mod $\rho^7$. Here, we may compute
\[
\lambda_1\cdot\tau^4\lambda_7\cdot\tau^4\lambda_7 = \rho^6\cdot\lambda_7^3\tau^5+O(\rho^7),
\]
yielding the claim.

(5)~~Here, we are claiming $h_1\cdot h_3\cdot \tau^4 h_3 = \rho^4\cdot \tau^2 c_1$. For this, it suffices to work mod $\rho^5$. Here, we may compute
\[
\lambda_1\cdot\tau^4\lambda_7\cdot\lambda_7 = \rho^4\cdot\lambda_5\lambda_7^2\tau^2+O(\rho^6),
\]
yielding the claim.

(6)~~There are no nonzero $\rho$-torsion classes in this degree, so the product must vanish.
\end{proof}

The preceding propositions leave open three families of products. A complete resolution of these requires the following, which appeared as a conjecture in an earlier version of this work. We thank Dugger, Hill, and Isaksen for supplying a proof.

\begin{lemma}[Dugger--Hill--Isaksen]\label{lem:movetau}
There are relations
\begin{enumerate}
\item $\tau^{4m+1}h_{1}\cdot \tau^{2l}h_0 = \tau h_1\cdot\tau^{2(2m+l)}h_0$;
\item $\tau^{4(4m+1)}h_3 \cdot\tau^{8l+1}h_2^2 = \tau^4h_3\cdot\tau^{8(2m+l)+1}h_2^2$;
\item $\tau^{8m+1}h_{2}^2\cdot\tau^{2l}h_0 = \tau h_2^2\cdot\tau^{2(2m+l)}h_0$.
\end{enumerate}
\end{lemma}
\begin{proof}
These will be proved using Massey product shuffling techniques. The Massey products we require are most easily computed using the $\rho$-Bockstein spectral sequence, see especially \cite[Section 7.4]{BI20} for a discussion of Massey products in $\Ext_\r$.

(1)~~By induction on $m$, it suffices to show
\[
\tau^{2l}h_0 \cdot \tau^{4m+5}h_1=  \tau^{2l+4}h_0 \cdot \tau^{4m}h_1
\]
for $m\geq 0$. Observe that
\[
\tau^{4m+5}h_1 = \langle \rho^2,\rho^2\tau^2h_2,\tau^{4m+1}h_1\rangle,\qquad \tau^{2l+4}h_0 = \langle \tau^{2l},\rho^2,\rho^2\tau^2 h_1\rangle
\]
with no indeterminacy. We may therefore shuffle
\begin{align*}
\tau^{2l}h_0\cdot \tau^{4m+5}h_1 &= \tau^{2l}h_0\langle \rho^2,\rho^2\tau^2h_2,\tau^{4m+1}h_1\rangle\\
&= \langle \tau^{2l}h_0,\rho^2,\rho^2\tau^2h_2\rangle \tau^{4m+1}h_1 = \tau^{2l+4}h_0\cdot \tau^{4m+1}h_1.
\end{align*}

(2)~~By induction on $m$, it suffices to show
\[
\tau^{8l+1}h_2^2\cdot\tau^{4(4m+1)+16}h_3 = \tau^{8l+17}h_2^2\cdot\tau^{4(4m+1)}h_3
\]
for $m\geq 0$. Observe that
\[
\tau^{4(4m+1)+16}h_3 = \langle \rho^8,\rho^8\tau^8h_4,\tau^{4(4m+1)}h_3\rangle,\qquad \tau^{8l+17}h_2^2 = \langle \tau^{8l+1}h_2^2,\rho^8,\rho^8\tau^8h_4\rangle
\]
with no indeterminacy. We may therefore shuffle
\begin{align*}
\tau^{8l+1}h_2^2\cdot\tau^{4(4m+1)+16}h_3 &= \tau^{8l+1}h_2^2 \langle \rho^8,\rho^8\tau^8h_4,\tau^{4(4m+1)}h_3\rangle\\
&=  \langle \tau^{8l+1}h_2^2,\rho^8,\rho^8\tau^8h_4\rangle \tau^{4(4m+1)}h_3 = \tau^{8l+17}h_2^2\cdot\tau^{4(4m+1)}h_3.
\end{align*}

(3)~~By induction on $m$, it suffices to show
\[
\tau^{2l}h_0\cdot\tau^{8m+9}h_2^2 = \tau^{2l+8}h_0\cdot\tau^{8m+1}h_2^2
\]
for $m\geq 0$. Observe that
\[
\tau^{8m+9}h_2^2 = \langle \rho \tau^4 h_3,\rho^7,\tau^{8m+1}h_2^2\rangle,\qquad \tau^{2l+8}h_0 = \langle \tau^{2l}h_0,\rho \tau^4 h_0,\rho^7\rangle
\]
with no indeterminacy. We may therefore shuffle
\begin{align*}
\tau^{2l}h_0\cdot\tau^{8m+9}h_2^2 &= \tau^{2l}h_0\langle \rho \tau^4 h_3,\rho^7,\tau^{8m+1}h_2^2\rangle\\
&=\langle \tau^{2l}h_0,\rho \tau^4 h_0,\rho^7\rangle \tau^{8m+1}h_2^2 = \tau^{2l+8}h_0\cdot\tau^{8m+1}h_2^2.
\end{align*}
This concludes the proof.
\end{proof}

From here, we have the following.

\begin{proposition}\label{prop:specialproducts}
Write $2m+l+1=2^k(2n+1)$. Then the following hold.
\begin{enumerate}
\item $\tau^{2^a(4m+1)}h_{a+1}\cdot\tau^{\floor{2^{a-1}(4l+1)}}h_a = \rho^{2^a(2^{k+1}-1)}\cdot h_{a+1}\cdot \tau^{2^{a+k}(4n+1)}h_{a+k+1}$;
\item $\tau^{2^{a+2}(4m+1)}h_{a+3}\cdot\tau^{2^a(8l+1)}h_{a+2}^2 = \rho^{2^{a+1}(2^{k+2}-3)}\cdot h_{a+1}\cdot h_{a+3}\cdot \tau^{2^{a+k+2}(4n+1)}h_{a+k+3}$;
\item $\tau^{2^a(8m+1)}h_{a+2}^2\cdot\tau^{\floor{2^{a-1}(4l+1)}}h_a = \rho^{2^a(2^{k+1}-1)}\cdot h_{a+2}^2\cdot\tau^{2^{a+k}(4n+1)}h_{a+k+1}$.
\end{enumerate}
\end{proposition}
\begin{proof}
In each of these, we may use $\Sq^0$ to reduce to the case $a=0$. By working modulo the smallest power of $\rho$ in which the claimed product does not vanish, we may reduce to the case $n=0$. By \cref{lem:movetau}, we may moreover reduce to the case $m=0$. The proposition is now proved by the following computations.

(1)~~Here, we are claiming $\tau h_1\cdot \tau^{2(2^k-1)}h_0 = \rho^{2^{k+1}-1}\cdot h_1\cdot \tau^{2^k}h_{k+1}$. Recall that $\tau^{2(2^k-1)}h_0$ is represented by $\rho^{-1}\delta(\tau^{2(2^k-1)+1})$. Now, the Leibniz rule implies
\begin{align*}
\rho^{-1}\delta(\tau^{2(2^k-1)+1})\cdot \tau\lambda_1 &= \rho^{-1}\delta(\tau^{2^{k+1}})\cdot \lambda_1 + \rho^{-1}\tau^{2(2^k-1)+1}\cdot\delta(\tau)\cdot\lambda_1.
\end{align*}
The second summand vanishes, as $\delta(\tau)\cdot\lambda_1 = \rho\lambda_0\cdot\lambda_1 = 0$. The first summand represents $\rho^{2^{k+1}-1}\cdot\tau^{2^k}h_{k+1}\cdot h_1$, yielding the claimed relation.

(2)~~Here, we are claiming $\tau^4 h_3\cdot\tau^{8(2^k-1)+1}h_2^2 = \rho^{2(2^{k+2}-3)}\cdot h_1\cdot h_3\cdot \tau^{2^{k+2}}h_{k+3}$. Recall that $\tau^{8(2^k-1)+1}h_2^2$ is represented by $\rho^{-6} \delta(\lambda_1\tau^{8(2^k-1)+4})$. Now, the Leibniz rule implies
\[
\rho^{-6}\delta(\lambda_1\tau^{8(2^k-1)+4})\cdot\tau^4\lambda_7 = \rho^{-6}\lambda_1\cdot\delta(\tau^{2^{k+3}})\cdot\lambda_7 + \rho^{-6}\lambda_1\cdot\tau^{8(2^k-1)+4}\cdot\delta(\tau^4)\cdot\lambda_7.
\]
The second term vanishes, as $\delta(\tau^4)\cdot\lambda_3 = \tau^2\lambda_3\cdot\lambda_7 = 0$. The first term represents $\rho^{2(2^{k+2}-3)}\cdot h_1\cdot h_3\cdot \tau^{2^{k+2}}h_{k+3}$, yielding the claimed relation.

(3)~~Here, we are claiming $\tau h_2^2\cdot\tau^{2(2^k-1)}h_0 = \rho^{2^{k+1}-1}\cdot h_2^2\cdot\tau^{2^k}h_{k+1}$. Recall that $\tau^{2(2^k-1)}h_0$ is represented by $\rho^{-1}\delta(\tau^{2(2^k-1)+1})$. Now, the Leibniz rule implies
\[
\rho^{-1}\delta(\tau^{2(2^k-1)+1})\cdot\tau\lambda_3^2 = \rho^{-1}\delta(\tau^{2^{k+1}})\cdot\lambda_3^2+\rho^{-1}\tau^{2(2^k-1)+1}\cdot\delta(\tau)\cdot\lambda_3^2.
\]
The second term vanishes, as $\delta(\tau)\cdot\lambda_3^2 = \rho\lambda_0\cdot\lambda_3^2=0$. The first summand represents $\rho^{2^{k+1}-1}h_{k+1}\cdot h_2^2$, yielding the claimed relation.
\end{proof}

The relations above suffice to write any product in $\Ext_\r^{\leq 3}$ in terms of the basis given in \cref{thm:extadd}. Thus we have the following.

\begin{thm}\label{thm:extmult}
A full set of relations for $\Ext_\r^{\leq 3}$ is given by those visible relations which may be deduced from \cref{lem:unambiguousproduct}, together with the products listed in \cref{prop:2linerelations}, \cref{prop:3linerelations}, and \cref{prop:specialproducts}.
\qed
\end{thm}

\section*{Part II: The motivic Hopf invariant one problem}

\section{Some homotopical preliminaries}\label{sec:preliminaries}

With the algebraic computation of \cref{Sec:ExtLow} out of the way, we now proceed to more homotopical considerations. This brief section collects a couple constructions that will be used in the following sections. Explicitly, \cref{ssec:hurewicz} will be used in our computation of $d_2(h_5)$ in \cref{sec:hopf}, and \cref{ssec:betti} will be used in our discussion of the unstable Hopf invariant one problem in \cref{sec:hi1}.

\subsection{The Hurewicz map}\label{ssec:hurewicz}

The constant functor $c\colon\Sp^{\cl}\rightarrow\Sp^F$ has a lax symmetric monoidal right adjoint $c^\ast$, described by
\[
c^\ast(X) = \Sp^F(S^{0,0},X).
\]
In particular, the unit of $c^\ast(S^{0,0})$ gives a ring map
\[
S^0\rightarrow c^\ast(S^{0,0}),
\]
and on homotopy groups this yields a Hurewicz map
\[
c\colon \pi_\ast^\cl\rightarrow\pi_{\ast,0}^F.
\]

\begin{proposition}\label{lem:diagonal}
For any $F$, there is map
\[
c\colon \Ext_{\cl}^{s,f}\rightarrow\Ext_F^{s,f,0}
\]
of multiplicative spectral sequences, converging to the Hurewicz map
\[
c\colon \pi_\ast^{\cl}\rightarrow\pi_{\ast,0}^F.
\]
Moreover, $c$ commutes with $\Sq^0$ and satisfies $c(h_0) = h_0 + \rho h_1$.
\end{proposition}
\begin{proof}
Write $H\f_2$ for the ordinary mod $2$ Eilenberg-MacLane spectrum and $H\f_2^F$ for the motivic spectrum representing mod $2$ motivic cohomology. Then $c^\ast(H\f_2^F) = H\f_2$, thereby giving maps
\[
H\f_2^{\otimes n}\simeq c^\ast(H\f_2^F)^{\otimes n}\rightarrow c^\ast((H\f_2^F)^{\otimes n}).
\]
Thus there is a map from the canonical Adams resolution of the sphere to the restriction along $c^\ast$ of the canonical Adams resolution of the $F$-motivic sphere. On homotopy groups, this gives a map from the cobar complex of $\ca^{\cl}$ to the weight $0$ portion of the cobar complex of $\ca^F$, and passing to homology we obtain a map
\[
\Ext_{\cl}^{s,f}\rightarrow\Ext_F^{s,f,0}
\]
which is multiplicative and commutes with $\Sq^0$, and by construction this is a map of spectral sequences converging to the Hurewicz map. That $c(h_0) = h_0 + \rho h_1$ follows as these are the classes detecting $2$ (see for instance \cite[Remark 6.3]{IO19}).
\end{proof}

\subsection{The Lefschetz principle}\label{ssec:lefschetz}

The \textit{Lefschetz principle} asserts, informally, that ``everything'' which is true over $\c$ is true over any algebraically closed field. In this subsection, we note how one may read off a certain motivic Lefschetz principle from work of Wilson--{\O}stv{\ae}r \cite{WO17}.

In this paper, we have primarily been concerned with $F$-motivic homotopy theory for $F$ a field of characteristic not equal to $2$. For this subsection, we extend our notation to apply also when $F$ is some ring in which $2$ is invertible. We shall write $S^{0,0}$ for the $H\f_2^F$-nilpotent completion of the $F$-motivic sphere spectrum. When $F$ is a field, this is the $(2,\eta)$-completion of the $F$-motivic sphere spectrum, and when $F$ is an algebraically closed field, this reduces to a $2$-completion \cite{HKO11a, KW19a}. Let $\Sp^F_2$ denote the category of modules over this completed $F$-motivic sphere spectrum. In addition, let $\Sp^{F,\cell}_2\subset\Sp^F_2$ denote the cellular subcategory, i.e.\ the category generated by the spheres $S^{a,b}$ under colimits.

\begin{proposition}\label{prop:lefschetz}
Let $F$ be an algebraically closed field. Then there is an equivalence
\[
\Sp^{F,\cell}_2\simeq \Sp^{\c,\cell}_2.
\]
Moreover, this is compatible on Adams spectral sequences with the isomorphism $\Ext_F\cong\Ext_\c$.
\end{proposition}
\begin{proof}
First suppose that $F$ is of odd characteristic $p$. We follow the methods of \cite[Section 6]{WO17}. Let $W(F)$ be the ring of Witt vectors on $F$, and choose an algebraically closed field $L$ of characteristic $0$ together with embeddings
\[
\c\rightarrow L\leftarrow W(F)\rightarrow F.
\]
This gives rise to base change functors
\[
\Sp^\c\rightarrow \Sp^L\leftarrow\Sp^{W(F)}\rightarrow\Sp^F,
\]
and in particular maps
\begin{equation}\label{eq:zigzagpi}
\pi_{\ast,\ast}^\c\rightarrow\pi_{\ast,\ast}^L\leftarrow \pi_{\ast,\ast}^{W(F)}\rightarrow \pi_{\ast,\ast}^F.
\end{equation}
Although $W(F)$ is not a field, Wilson--{\O}stv{\ae}r show that its Steenrod algebra and Adams spectral sequence are still well-behaved, and \cite[Corollary 6.3]{WO17} shows that the above maps are modeled on motivic Adams spectral sequences by a zigzag of isomorphisms
\[
\Ext_\c\rightarrow\Ext_L\leftarrow\Ext_{W(F)}\rightarrow\Ext_F.
\]
It follows that \cref{eq:zigzagpi} is a zigzag of isomorphisms. In particular, consider the zigzag
\[
\Sp^{\c,\cell}_2\rightarrow\Sp^{L,\cell}_2\leftarrow\Sp^{W(F),\cell}_2\rightarrow \Sp^{F,\cell}_2.
\]
This is a zigzag of colimit-preserving functors of compactly generated stable categories which are equivalences on subcategories of compact generators, and is thus a zigzag of equivalences. This yields the canonical equivalence $\Sp^{\c,\cell}_2\simeq\Sp^{F,\cell}_2$.

If $F$ is of characteristic zero, then we may apply the same argument instead to a zigzag of the form
\[
\c\rightarrow L \leftarrow F
\]
with $L$ algebraically closed.
\end{proof}

\subsection{Betti realization}\label{ssec:betti}
If $X$ is a smooth scheme over $\c$, then the space of complex points of $X$ is a complex manifold. This refines to give \textit{Betti realization} functors \cite{MV99} from $\c$-motivic spaces to ordinary spaces, and from $\c$-motivic spectra to ordinary spectra, with a number of nice properties. We may use the Lefschetz principle of \cref{prop:lefschetz} to obtain an analogue for an arbitrary algebraically closed field $F$.

Let $S^0$ denote the $2$-completed sphere spectrum, and $\Sp^\cl_2$ the category of modules thereover.

\begin{proposition}\label{cor:betti}
Let $F$ be an algebraically closed field. Then there is a symmetric monoidal ``Betti realization'' functor
\[
\Be\colon\Sp^{F,\cell}_2\rightarrow\Sp^\cl_2,
\]
factoring through an equivalence from the category of modules over $S^{0,0}[\tau^{-1}]$ in $\Sp^{F,\cell}_2$ to $\Sp^\cl_2$, with the following properties.
\begin{enumerate}
\item $\Be(\tau) = 1$. In particular, $\Be(S^{a,b}) = S^a$, so that $\Be$ induces a map $\pi_{s,w}^F\rightarrow\pi_s^\cl$, and these patch together to an isomorphism $\pi_{\ast,\ast}^F[\tau^{-1}]\cong\pi_\ast^\cl[\tau^{\pm 1}]$.
\item The above isomorphism is modeled on Adams spectral sequences by the map
\[
\Ext_F\rightarrow\Ext_F[\tau^{-1}]\cong\Ext_\cl[\tau^{\pm 1}].
\]
\item The composite $\Be\circ c\colon \Sp^\cl_2\rightarrow\Sp^{F,\cell}_2\rightarrow\Sp^\cl_2$ is an equivalence. In particular, the map $c\colon\Ext_\cl\rightarrow\Ext_F$ of \cref{lem:diagonal} extends to an equivalence $\Ext_\cl[\tau^{\pm 1}]\rightarrow\Ext_F[\tau^{-1}]$.
\end{enumerate}
\end{proposition}
\begin{proof}
These facts are known of the Betti realization functor for $F = \c$ \cite[Section 2]{DI10}, and the general case immediately follows from \cref{prop:lefschetz}.
\end{proof}

Using Mandell's $p$-adic homotopy theory \cite{Man01}, we may also produce an unstable analogue. Let $F$ be an algebraically closed field. Note from \cite[Proposition 15]{HKO11} that the spectrum $H\f_2^F$ is cellular; moreover, $\Be(H\f_2^F) = H\f_2$, as can be seen by inspection of homotopy groups. Let $\Spc(F)$ be the category of $F$-motivic spaces and $\Spc_2$ be the category of $2$-complete spaces.

\begin{proposition}\label{prop:bettiun}
Let $F$ be an algebraically closed field, and define
\[
\Be\colon\Spc(F)\rightarrow\Spc_2,\qquad \Be(X) = \CAlg_{H\f_2}(\Be((H\f_2^F)^{X_+}),\overline{\f}_2).
\]
Then $\Be(S^{a,b}) = (S^a)_2^\wedge$, and, at least when restricted to the full subcategory of $\Spc(F)$ consisting of simply connected finite motivic cell complexes, the functor $\Be$ preserves finite colimits and satisfies
\[
H\f_2^{\Be(X)_+}\simeq\Be((H\f_2^F)^{X_+}).
\]
\end{proposition}
\begin{proof}
We begin by recalling two facts from Mandell's work on $p$-adic homotopy theory \cite{Man01}. Strictly speaking, Mandell states his main theorem at the level of homotopy categories; a reference explicitly treating the full homotopical version we use is \cite[Section 3]{lurie2011rational}. First, the functor
\[
\Spc\rightarrow\CAlg_{H\overline{\f}_2},\qquad Y \mapsto H\overline{\f}_2^{Y_+}
\]
is fully faithful when restricted to the full subcategory of connected $2$-complete nilpotent spaces with locally finite mod $2$ cohomology. In particular, if $Y$ is a connected nilpotent space with locally finite mod $2$ cohomology, then the unit map
\[
Y\simeq\Spc(\ast,Y)\rightarrow\CAlg_{H\overline{\f}_2}(H\overline{\f}_2^{Y_+},H\overline{\f}_2^{\ast_+})\simeq \CAlg_{H\f_2}(H\f_2^{Y_+},H\overline{\f}_2)
\]
realizes the target as the $2$-completion of $Y$. Second, the functor
\[
\CAlg_{H\f_2}^{\text{op}}\rightarrow\Spc,\qquad R\mapsto \CAlg_{H\f_2}(R,H\overline{\f}_2)
\]
lands in $\Spc_2$ and preserves finite colimits when restricted to the full subcategory of $\e_\infty$ algebras $R$ over $\f_2$ such that $R_\ast$ is locally finite-dimensional, $R_0 = \f_2$, $R_1 = 0$, and the Dyer--Lashof operation $Q^0$ acts by the identity on $R_\ast$.

We now apply this to our situation. The stable Betti realization functor is symmetric monoidal, and thus $\Be((H\f_2^F)^{X_+})$ is indeed an $\e_\infty$ ring over $\f_2$. Moreover, as $\Sq^0$ acts by the identity on $H^{\ast,\ast}(X)$, the Dyer--Lashof operation $Q^0$ acts by the identity on $\pi_\ast \Be((H\f_2^F)^{X_+})$. In particular, $\Be((H\f_2^F)^{S^{a,b}_+})\simeq H\f_2^{S^a_+}$, and so the proposition follows by applying Mandell's theory.
\end{proof}

\begin{remark}
We have focused in this section on $2$-primary motivic homotopy theory over a field $F$ of characteristic not $2$. However, we note that our discussion applies in general to $p$-primary motivic homotopy theory over a field $F$ of characteristic not $p$.
\tqed
\end{remark}

\section{The motivic Hopf invariant one problem}\label{sec:hi1}

In this section, we formulate and discuss motivic analogues of the Hopf invariant one problem. The material in this section is not needed for \cref{sec:hopf}.

\subsection{The unstable Hopf invariant one problem} 

Classically, Adams' determination of the permanent cycles in $\Ext^1_\cl$ resolved the Hopf invariant one problem. The Hopf invariant one problem may be formulated motivically using the following.

\begin{definition}\label{def:hi1}
Let $f\colon S^{2a-1,2b}\rightarrow S^{a,b}$ be an unstable map between motivic spheres; in particular, $a\geq b\geq 0$ and $a\geq 1$.
Write $C(f)$ for the cofiber of $f$. The map $f$ vanishes in mod $2$ motivic cohomology for degree reasons, and thus there exists an isomorphism
\[
H^{\ast,\ast}(C(f)_+) \cong \m^F\{1,x,y\}
\]
of $\m^F$-modules, where $|x| = (-a,-b)$ and $|y| = (-2a,-2b)$. Say that $f$ has \textit{Hopf invariant one} if one may choose such generators $x$ and $y$ to satisfy
\[
x^2 = y,
\]
i.e.\ if $H^{\ast,\ast}(C(f)_+)\cong\m^F[x]/(x^3)$; otherwise $x^2 = 0$ and $f$ has Hopf invariant zero.
\tqed
\end{definition}

The unstable motivic Hopf invariant one problem is now the following question.

\begin{question}\label{q:hi1}
For which $(a,b)$ does there exist a map $f\colon S^{2a-1,2b}\rightarrow S^{a,b}$ of Hopf invariant $1$?
\tqed
\end{question}

This turns out to mostly reduce to the classical case, by way of the following.

\begin{lemma}\label{lem:hi1basechange}
Let $f\colon S^{2a-1,2b}\rightarrow S^{a,b}$ be an unstable $F$-motivic map. Then $f$ has Hopf invariant one if and only if its base change to an algebraic closure of $F$ is of Hopf invariant one.
\end{lemma}
\begin{proof}
This is immediate from the definitions.
\end{proof}

\begin{proposition}\label{prop:nohi1}
Suppose that $F$ is algebraically closed, and fix an unstable $F$-motivic map $f\colon S^{2a-1,2b}\rightarrow S^{a,b}$ of Hopf invariant one. Then the Betti realization (see \cref{prop:bettiun}) of $f$ is an odd multiple of $2$, $\eta$, $\nu$, or $\sigma$. In particular, $a\in\{1,2,4,8\}$.
\end{proposition}
\begin{proof}
By \cref{lem:hi1basechange}, we may as well suppose that $F$ is algebraically closed. Let $C(f)$ denote the cofiber of $f$ and $C(\Be(f))$ the cofiber of $\Be(f)$. Then $\Be(C(f)) = C(\Be(f))$ by \cref{prop:bettiun}, and thus $H^\ast (C(\Be(f))_+) = H^\ast(\Be(C(f))_+) = \f_2[x]/(x^3)$ with $|x| = -a$. In other words, the map between $2$-completed spheres $\Be(f)\colon S^{2a-1}\rightarrow S^a$ has Hopf invariant one. The proposition now follows from Adams' resolution of the Hopf invariant one problem \cite{Ada60}.
\end{proof}

\cref{prop:nohi1} is not a complete answer to \cref{q:hi1}, as we have not given any bounds on $b$, nor have we discussed the existence of maps of Hopf invariant one. Although we will not end up with a complete answer in general, there is more we can say. Before this, we recall what information is encoded in the $1$-line of the $F$-motivic Adams spectral sequence.

\subsection{The stable Hopf invariant one problem}\label{ssec:1lineext}

\cref{q:hi1} can be rephrased as asking when there exists an unstable $2$-cell complex, with cells in dimension $(a,b)$ and $(2a,2b)$, such that in cohomology the bottom cell squares to the top cell. In the stable category, one no longer has cup-squares; instead, one has Steenrod operations. Thus we may consider the stable motivic Hopf invariant one problem to be the following question.

\begin{question}\label{q:hi1stable}
What $\ca^F$-modules arise as the cohomology of $2$-cell complexes? In particular, for which $(a,b)$ does there exist a $2$-cell complex, with cells in dimensions $(0,0)$ and $(a,b)$ and attaching map vanishing in mod $2$ motivic cohomology, such that $H^{\ast,\ast} X = \m^F\{x,y\}$ is not split as an $\ca^F$-module?
\tqed
\end{question}

This is a particular case of the \textit{realization problem} for $\ca^F$-modules, and is exactly what the $1$-line of the $F$-motivic Adams spectral sequence encodes. 
The following is standard: 

\begin{proposition}\label{prop:einv}
Fix a class $\epsilon\in\Ext_F^{a-1,1,b}$ classifying an extension $0\rightarrow \m^F\{y\}\rightarrow E \rightarrow \m^F\{x\}\rightarrow 0$ of $\ca^F$-modules with $|x| = (0,0)$ and $|y| = (-a,-b)$. Then the following are equivalent:
\begin{enumerate}
\item There is stable $2$-cell complex $C$ with cells in dimensions $(0,0)$ and $(a,b)$ such that $H^{\ast,\ast} C\cong E$.
\item The class $\epsilon$ is a permanent cycle in the $F$-motivic Adams spectral sequence, and thus detects a stable class $\alpha\in\pi_{a-1,b}^F$.
\end{enumerate}
Explicitly, if $\epsilon\in\Ext_F^{a-1,1,b}$ detects $\alpha\in \pi_{a-1,b}^F$, then the cofiber $C(\alpha)$ satisfies $H^{\ast,\ast}C(\alpha)\cong E$; and if $C$ is a stable $2$-cell complex with $H^{\ast,\ast} C = E$, then the fiber of the inclusion $S^{0,0}\rightarrow C$ is a map $\alpha\colon S^{a-1,b}\rightarrow S^{0,0}$ detected by $\epsilon\in \Ext_F^{a-1,1,b}$.
\qed
\end{proposition}

As we will see in \cref{sec:hopf}, the $1$-line of the $F$-motivic Adams spectral sequence is already quite rich, and strongly depends on the base field $F$. Thus in considering the stable Hopf invariant one problem, one may not reduce to the case where $F$ is algebraically closed, unlike in the unstable case.

\subsection{Relation between the unstable and stable motivic Hopf invariant one problems}

We may now relate the unstable and stable questions, \cref{q:hi1} and \cref{q:hi1stable}.

\begin{proposition}\label{prop:hi1ext}
Let $f\colon S^{2a-1,2b}\rightarrow S^{a,b}$ be a map of Hopf invariant one. Then the associated stable class $\alpha\in\pi_{a-1,b}^F$ is detected by a permanent cycle in $\Ext_F^{a-1,1,b}$ which, after base change to the algebraic closure of $F$, is one of the following:
\[
h_0,\quad h_1,\quad \tau h_1,\quad h_2,\quad \tau h_2,\quad \tau^2h_2,\quad h_3,\quad \tau h_3,\quad \tau^2 h_3,\quad\tau^3 h_3,\quad \tau^4 h_3.
\]
In particular, if $\Ext_F^{a-1,1,b}$ does not contain any such permanent cycle, then there is no map $f\colon S^{2a-1,2b}\rightarrow S^{a,b}$ of Hopf invariant one.
\end{proposition}
\begin{proof}
By \cref{lem:hi1basechange}, we may suppose that $F$ itself is algebraically closed. By stabilizing \cref{prop:nohi1}, we find that $\Be(\alpha)$ is detected by $h_1$, $h_2$, or $h_3$ in $\Ext_\cl^1$. Recall from \cref{cor:betti} that Betti realization is modeled on Adams spectral sequences by the map
\[
\Ext_F\rightarrow\Ext_F[\tau^{-1}] \cong\Ext_\cl[\tau^{\pm 1}].
\]
In particular, the structure of $\Ext_F$ (see \cref{prop:extc}) implies that $\alpha$ must be detected by a permanent cycle in $\Ext_F$ of the form $\tau^nh_0$, $\tau^n h_1$, $\tau^n h_2$, or $\tau^n h_3$ for some $n\geq 0$. As $f$ is an unstable map, this class must have nonnegative weight, reducing to the listed classes.
\end{proof}

\begin{remark}\label{rmk:instability}
Our method of relating the unstable motivic Hopf invariant one problem to the stable motivic Hopf invariant one problem, going through the ``Betti realization'' functors of \cref{ssec:betti}, may seem somewhat roundabout. This route was taken for the following reason: if $f\colon S^{2a-1}\rightarrow S^a$ is a map of Hopf invariant $1$, then the fact that $H^\ast(C(f))$ is nonsplit as a $\ca^{\cl}$-module, and thus the associated stable class $\alpha\in\pi_{a-1}^\cl$ is detected in $\Ext_\cl^1$, follows from the instability condition $\Sq^a(x) = x^2$.

Motivically, the analogous instability condition asserts that if $X$ is a motivic space and $x\in H^{2a,a}(X_+)$, then $\Sq^{2a}(x) = x^2$ \cite[Lemma 9.7]{Voe03}. Now suppose that $f\colon S^{2a-1,2b}\rightarrow S^{a,b}$ is an unstable map of Hopf invariant one, and write $H^{\ast,\ast}(C(f)_+) = \m^F[x]/(x^3)$ with $|x|=(-a,-b)$. If $a$ is even and $b\leq a/2$, then one may set $c = a/2-b$ and deduce $\Sq^a(\tau^c x) = \tau^{2c}x^2$, so that $H^{\ast,\ast}(C(f))$ is not split as a $\ca^F$-module. If $a$ is odd, then one may argue by appealing to an integral motivic Hopf invariant and graded commutativity, as in the classical case. Thus it is to rule out the possibility of a map $f\colon S^{2a-1,2b}\rightarrow S^{a,b}$ of Hopf invariant one with $b>a/2$ that we have taken our approach.
\tqed
\end{remark}

Our computations in \cref{sec:hopf} show, for a variety of base fields $F$, when $\Ext_F^1$ contains a permanent cycle whose image over the algebraic closure is one of the classes listed in \cref{prop:hi1ext}, yielding various nonexistence results. To obtain existence results, we must recall how maps of Hopf invariant one arise.

\subsection{Geometric applications}\label{ssec:hopfconstruction}

Adams' resolution of the classical Hopf invariant one problem had geometric consequences; notably, it implied that the only spheres which admit $H$-space structures are the spheres $S^0$, $S^1$, $S^3$, and $S^7$. It makes sense to ask for the motivic analogue of this, i.e.\ to ask which spheres $S^{a,b}$ admit $H$-space structures.

This question is in some sense geometric, but we can also ask for something even more concrete. The spheres $S^{a,b}$ are certain sheaves on the Nisnevich site of smooth $F$-schemes, and so it is reasonable to ask when $S^{a,b}$ is in fact represented by a smooth $F$-scheme. This question was raised and studied by Asok--Doran--Fasel in \cite{ADF17}; in particular, they produce explicit smooth affine schemes representing $S^{a,\ceil{a/2}}$, as well as prove that $S^{a,b}$ is not represented by a smooth scheme for $a>2b$. Motivated by this, we are led to ask the following question.

\begin{question}
For what pairs $(a,b)$ is $S^{a,b}$ a motivic $H$-space? Of these, when is it represented by a smooth $F$-scheme which admits a unital product?
\tqed
\end{question}

Classically, the connection between the $H$-space structures and the Hopf invariant one problem is via the \textit{Hopf construction}. This construction may also be carried out in the motivic category, and has been studied in this context by Dugger--Isaksen \cite{DI13}. We recall the key points.

\begin{definition}\label{Def:HopfConstruction}\cite[Definition C.1]{DI13}
Let $X$, $Y$, and $Z$ be pointed spaces, and let $h: X \times Y \rightarrow Z$ be a pointed map. The \textit{Hopf construction of $h$} is the map $H(h): X \star Y \rightarrow \Sigma Z$ obtained by taking homotopy colimits of the rows of the diagram
\begin{center}
\begin{tikzcd}
X \arrow{d} & X \times Y \arrow{l} \arrow{r} \arrow{d} & Y \arrow{d} \\
\ast & Z \arrow{l} \arrow{r} & \ast
\end{tikzcd}.
\end{center}
\tqed
\end{definition}

Here, $\star$ is the join. Note that $S^{a,b}\star S^{c,d}\simeq S^{a+c+1,b+d}$; thus the Hopf construction may be used to construct maps between motivic spheres. Using the theory of Cayley--Dickson algebras, Dugger--Isaksen \cite[Section 4]{DI13} used this to define \textit{motivic Hopf maps} $\eta\in\pi_{1,1}^F$, $\nu\in\pi_{3,2}^F$, and $\sigma\in\pi_{7,4}^F$. As noted in \cite[Remark 4.14]{DI13}, these motivic Hopf maps have Hopf invariant one. This is a general property of the Hopf construction, which we may summarize in the following.

\begin{lemma}\label{prop:hopfconstruction}
If $\mu : S^{a-1,b} \times S^{a-1,b} \rightarrow S^{a-1,b}$ is an $H$-space product, then its Hopf construction $H(\mu) : S^{2a-1,2b} \rightarrow S^{a,b}$ has Hopf invariant one.
\end{lemma}

\begin{proof}
The proof of the analogous fact for topological spaces \cite[Section I.5]{SE62} extends to motivic spaces. We summarize the key points.

Define the (mod $2$) \textit{degree} of a pointed map $S^{a,b}\rightarrow S^{a,b}$ of motivic spaces to be its induced map in reduced motivic cohomology. A pointed map $f: S^{a-1,b} \times S^{a-1,b} \rightarrow S^{a-1,b}$ of motivic spaces is said to have \textit{degree $(\alpha, \beta)$} if $f|_{S^{a-1,b} \times \{p_2\}}$ has degree $\alpha$ and $f|_{\{p_1\} \times S^{a-1,b}}$ has degree $\beta$. Since $\mu$ is an $H$-space product, its restrictions to $S^{a-1,b} \times \{p_2\}$ and $\{p_1\} \times S^{a-1,b}$ are homotopic to the identity, so $\mu$ has degree $(1,1)$. The lemma follows by showing that, more generally, the Hopf invariant, defined in the evident way, of the Hopf construction of a map of degree $(\alpha, \beta)$ is $\alpha\cdot \beta$.

Steenrod--Epstein's proof of \cite[Lemma 5.3]{SE62} carries over to the motivic setting to complete the proof. The main point is that Steenrod--Epstein work with particular models of the cone, join, homotopy cofiber, and suspension in their proof, but any model would work, as all of their statements only depend on the homotopy types of the relevant spaces and homotopy classes of the relevant maps. More precisely, with notation as in their proof, one may replace $E_1$, $E_2$, $E_+$, and $E_-$ by the cones on $S_1$, $S_2$, $S$, and $S$, respectively, to avoid any potential point-set issues. In particular, one regards $E_1$, $E_2$, $E_+$, and $E_-$ as suspension data in the sense of \cite[Remark 2.9]{DI13} for the various suspensions appearing in the Hopf construction. In this language, the identifications of various pushouts in the proof of \cite[Lemma 5.3]{SE62} are examples of induced orientations \cite[Remark 2.10]{DI13}. The proof carries through unchanged with these new choices of $E_1$, $E_2$, $E_+$, and $E_-$.

To be precise, their proof considers maps $S^{n-1}\times S^{n-1}\rightarrow S^{n-1}$ with $n>1$ even and works integrally. Routine modifications extend this to arbitrary $n\geq 1$ provided one works mod $2$ throughout. Classically, this is the adaption needed to incorporate the degree $2$ map $S^1\rightarrow S^1$, which is the Hopf construction of the standard product on $S^0\cong C_2$. Motivically, this is the adaption needed for our lemma to hold for arbitrary unstable motivic spheres $S^{a-1,b}$, allowing especially for the uniform treatment of $2$ and $\eta$.
\end{proof}

\begin{remark}
We note that under \cref{def:hi1}, the map $h\colon S^{1,1}\rightarrow S^{1,1}$ represented by the squaring map on $\g_m$, sometimes called the ``zeroth Hopf map'' and stably detected by $h_0$, is \textit{not} a map of Hopf invariant one. In the context of \cref{prop:hopfconstruction}, this is justified by the fact that for degree reasons $h$ is not the Hopf construction of an $H$-space structure on any motivic sphere.
\tqed
\end{remark}

We can now summarize what is known in the following.

\begin{thm}\label{thm:hspace}
A motivic sphere is represented by a smooth $F$-scheme admitting a unital product if and only if it is one of the following:
\[
S^{0,0},\qquad S^{1,1},\qquad S^{3,2},\qquad S^{7,4}.
\]
In addition to the motivic spheres listed above, the following motivic spheres admit $H$-space structures:
\[
S^{1,0},\qquad S^{3,0},\qquad S^{7,0}.
\]
The only other motivic spheres that could possibly admit $H$-space structures are the following: 
\[
S^{3,1},\qquad S^{7,3},\qquad S^{7,2},\qquad S^{7,1};
\]
moreover, an $H$-space structure on such a sphere produces a permanent cycle in $\Ext_F$ whose image over the algebraic closure is $\tau h_2$, $\tau h_3$, $\tau^2 h_3$, or $\tau^3 h_3$ respectively.
\end{thm}
\begin{proof}
That the spheres $S^{0,0}$, $S^{1,1}$, $S^{3,2}$, and $S^{7,4}$ are represented by smooth $F$-schemes admitting a unital product is given by work of Dugger--Isaksen \cite{DI13}. The spheres $S^{1,0}$, $S^{3,0}$, and $S^{7,0}$ are the images of $S^1$, $S^3$, and $S^7$, respectively, under the unstable constant functor from spaces to motivic spaces, and so inherit $H$-space structures from their classical structures. That all the spheres listed are the only spheres which may admit $H$-space structures follows from \cref{prop:hopfconstruction} and \cref{prop:hi1ext}, as does the final claim concerning the $F$-motivic Adams spectral sequence. Finally, Asok--Doran--Fasel prove in \cite[Proposition 2.3.1]{ADF17} that if $S^{a-1,b}$ is represented by a smooth $F$-scheme, then necessarily $2b\geq a-1$, and the only possible $H$-spaces satisfying this are $S^{0,0}$, $S^{1,1}$, $S^{3,2}$, and $S^{7,4}$ as listed.
\end{proof}

We note the following special case.

\begin{corollary}\label{cor:realhspace}
Suppose there is an $\r$-motivic map $f\colon S^{2a-1,2b}\rightarrow S^{a,b}$ of Hopf invariant one. Then $(a,b)$ is one of
\[
(1,0),\qquad (2,1),\qquad (4,2),\qquad (8,4),\qquad (2,0),\qquad (4,0),\qquad (8,0).
\]
Moreover, all of these are realized, and in fact
\[
S^{0,0},\qquad S^{1,1},\qquad S^{3,2},\qquad S^{7,4},\qquad S^{1,0},\qquad S^{3,0},\qquad S^{7,0}
\]
are all the $\r$-motivic spheres admitting $H$-space structures.
\end{corollary}
\begin{proof}
This is immediate from \cref{thm:hspace}, either appealing to the fact that $\Ext_\r$ vanishes in the degrees detecting the remaining possibilities, or else noting that the real points of $S^{a,b}$ are $S^{a-b}$, so that if $S^{a,b}$ is an $H$-space then $a-b\in\{0,1,3,7\}$.
\end{proof}

\section{The \texorpdfstring{$1$}{1}-line of the motivic Adams spectral sequence}\label{sec:hopf}

We now analyze the $1$-line of the $F$-motivic Adams spectral sequence. We begin in \cref{ssec:extf} by explaining how to read off the structure of $\Ext_F$ for various fields $F$ from our computation of $\Ext_\r$. After some additional preliminaries in \cref{ssec:hopfexist}, we give a direct motivic analogue of the classical differentials in \cref{ssec:d2ha}, proving $d_2(h_{a+1}) = (h_0+\rho h_1)h_a^2$ for $a \geq 3$ over arbitrary base fields. We then proceed to give more detailed information about the $1$-line for the particular fields $F$ of the form $\r$, $\f_q$ with $q$ an odd prime-power, $\q_p$ with $p$ any prime, and $\q$.

\subsection{Computing \texorpdfstring{$\Ext_F$}{Ext\_F}}\label{ssec:extf}

As a general rule, $\Ext_F$ is largely understood once $\m^F$ and $\Ext_\r$ are both understood. Rather than formulate a precise statement, let us just describe $\Ext_F$ for the various particular fields $F$ we shall encounter, namely those described in \cref{Exm:M2F} as well as $F = \q$.

Recall from \cref{rmk:runiv2} that for any field $F$, we may view $\m^F$ as an $\ca^\r$-module, and there is an isomorphism
\[
\Ext_F\cong \Ext_{\ca^\r}(\m^\r,\m^F).
\]
Thus, the main point is to understand $\m^F$ as a $\ca^\r$-module, and this is in fact determined by $\m^F_0$ as an $\f_2[\rho]$-module. For the examples of interest, we have the following.
Abbreviate
\[
\m = \f_2[\tau,\rho],\qquad \m_{(r)} = \m/(\rho^r).
\]

\begin{lemma}\label{lem:mmodule}
As $\ca^\r$-modules, we have the following.
\begin{enumerate}
\item If $F = \r$, then $\m^F = \m$.
\item If $F = \Fbar$ is algebraically closed, then $\m^F = \m_{(1)}$.
\item If $F = \f_q$ with $q\equiv 1 \pmod{4}$, then $\m^F = \m_{(1)}\{1,u\}$.
\item If $F=\f_q$ with $q\equiv 3 \pmod{4}$, then $\m^F = \m_{(2)}$.
\item If $F=\q_p$ with $p\equiv 1 \pmod{4}$, then $\m^F = \m_{(1)}\{1,\pi,u,\pi u\}$.
\item If $F=\q_p$ with $p\equiv 3\pmod{4}$, then $\m^F = \m_{(2)}\{1,\pi\}$.
\item If $F = \q_2$, then $\m^F = \m_{(3)}\{1\}\oplus\m_{(1)}\{u,\pi\}$.
\item If $F = \q$, then $\m^\q = \m\{1\}\oplus \m_{(1)}\{[2]\}\oplus\m_{(1)}\{[p],a_p:p\equiv 1~(4)\}\oplus \m_{(2)}\{u_p:p\equiv 3~(4)\}$.
\end{enumerate}
\end{lemma}
\begin{proof}
All but the case $F=\q$ may be read off the examples listed in \cref{Exm:M2F}. When $F = \q$, the ring $\m^\q$ is described by Ormsby--{\O}stv{\ae}r in \cite[Propositions 5.3 and 5.4]{OO13}, following work of Milnor \cite{Mil70}. Our description may be read off this upon setting $u_p = [p] + \rho$ for $p\equiv 3 \pmod{4}$.
\end{proof}

For $r\geq 0$, define
\[
\Ext_{(r)} = \Ext_{\ca^\r}(\m,\m_{(r)}) = H_\ast(\Lambda^\r/(\rho^r)).
\]
The $\f_2[\rho]$-module structure of $\Ext_{(r)}$ may be easily computed from $\Ext_\r$ via the long exact sequence associated to the cofiber of $\rho^r$. Even less work is necessary when $\Ext_\r$ has been computed by some method compatible with the $\rho$-Bockstein spectral sequence such as ours; see in particular \cref{rmk:modrho}. Thus \cref{thm:extadd} allows us to read off $\Ext_{(r)}^f$ for $f\leq 2$, as well as the image of $\Ext_\r\rightarrow\Ext_{(r)}^3$. This does not give the entirety of $\Ext_{(r)}^3$; however, we at least know that whatever remains is generated by classes which appear in the $\rho$-Bockstein spectral sequence as $\rho^k \alpha$ with $\alpha\in\Ext_{(1)}^3$ and $k < r$, and this is enough information for our purposes.

\cref{lem:mmodule} describes for various $F$ how $\Ext_F$ may be written as a direct sum of copies of various $\Ext_{(r)}$. For example, $\Ext_{\q_2} = \Ext_{(3)}\{1\}\oplus\Ext_{(1)}\{u,\pi\}$. We may use this to prove a Hasse principle for $\Ext_\q$.

\begin{lemma}\label{lem:hassecoeff}
The map
\[
\m^\q\rightarrow\m^{\q_p}
\]
satisfies
\[
[p]\mapsto \pi,\qquad a_p\mapsto u\pi,\qquad u_p\mapsto \pi+\rho.
\]
Here, the first is relevant for $p = 2$ or $p\equiv 1 \pmod{4}$, the second for $p\equiv 1 \pmod{4}$, and the third for $p\equiv 3\pmod{4}$.
\end{lemma}
\begin{proof}
The behavior of these maps is described by Ormsby--{\O}stv{\ae}r in \cite[Proposition 5.3]{OO13}, following work of Milnor \cite{Mil70}. Our description follows immediately; note we have defined $u_p = [p]+\rho$ for $p\equiv 3\pmod{4}$.
\end{proof}

\begin{proposition}\label{prop:hasse}
The Hasse map
\[
\Ext_\q\rightarrow \Ext_\r\times\prod_p \Ext_{\q_p}
\]
is injective.
\end{proposition}
\begin{proof}
By \cref{lem:mmodule}, we have
\[
\Ext_\q = \Ext_\r \oplus\Ext_{(1)}\{[2]\}\oplus\Ext_{(1)}\{[p],a_p:p\equiv 1 \pmod{4}\}\oplus \Ext_{(2)}\{u_p:p\equiv 3 \pmod{4}\}.
\]
The summand $\Ext_\r$ maps isomorphically to $\Ext_\r$, and the maps $\Ext_{\q}\rightarrow\Ext_{\q_p}$ are determined by \cref{lem:hassecoeff}. In particular, it is easily seen that the maps
\begin{gather*}
\Ext_{(1)}\{[2]\}\rightarrow\Ext_{\q_2},\\
\Ext_{(1)}\{[p],a_p\}\rightarrow\Ext_{\q_p},\\
\Ext_{(2)}\{u_p\}\rightarrow\Ext_{\q_p}
\end{gather*}
are all split injections, and the proposition follows.
\end{proof}

The preceding discussion, together with our computation of $\Ext_\r$, describes what we will need of $\Ext_F$ in low filtrations and arbitrary stem. So that we may rule out various higher differentials in low stems for degree reasons, we record the following.

\begin{lemma}\label{lem:lowstem}
$\Ext_{(1)}$ is given in stems $s\leq 6$ by the following module:
\begin{gather*}
\f_2[\tau]\otimes\left(\f_2\{h_0^n:n\geq 0\}\oplus\f_2\{h_1,h_1^2,h_1^3,h_2,h_0h_1,h_2^2\}\right)\oplus
\f_2[\tau]/(\tau)\{h_1^4,h_1^5,h_1^6\}.
\end{gather*}
\end{lemma}
\begin{proof}
These groups have been computed by Dugger--Isaksen \cite{DI10}. 
\end{proof}

\subsection{Existence of Hopf elements}\label{ssec:hopfexist}

Our computation of the $F$-motivic Adams differentials $d_2(h_{a+1})$ will follow a similar pattern to Wang's computation of the corresponding classical Adams differentials \cite{Wan67} (differentials which were first computed by Adams \cite{Ada60}). This is an inductive argument, beginning with information about the Hopf elements which are known to exist. We record some of this information in this subsection.

Write $\epsilon\in\pi_{0,0}^F$ for the class represented by the twist map $S^{1,1}\otimes S^{1,1}\rightarrow S^{1,1}\otimes S^{1,1}$.

\begin{lemma}\label{lem:gradedcomm}
Fix $\alpha\in\pi_{a,b}^F$ and $\beta\in\pi_{c,d}^F$. Then there is an identity
\[
\alpha\cdot\beta = (-1)^{(a-b)(c-d)}\epsilon^{bd}\cdot\beta\cdot\alpha.
\]
Moreover, $1-\epsilon$ is detected in $\Ext_F$ by $h_0$ and $2$ by $h_0+\rho h_1$.
\end{lemma}
\begin{proof}
The claimed graded commutativity is given in \cite[Corollary 6.1.2]{Mor04}; see also \cite[Section 6.1]{IO19} for a discussion. That $1-\epsilon$ is detected by $h_0$ and $2$ by $h_0+\rho h_1$ is noted in \cite[Remark 6.3]{IO19}. 
\end{proof}

\begin{lemma}\label{lem:lowhopf}
For any field $F$, the class $h_a$ is a permanent cycle for $a\in\{0,1,2,3\}$.
\end{lemma}
\begin{proof}
The class $h_0$ is a permanent cycle by \cref{lem:gradedcomm}. In \cite{DI13}, Dugger--Isaksen construct the motivic Hopf elements $\eta$, $\nu$, and $\sigma$, and in \cite[Remark 4.14]{DI13} they indicate that these are detected by $h_1$, $h_2$, and $h_3$, respectively; see also our discussion in \cref{ssec:hopfconstruction}. Thus these classes must be permanent cycles.
\end{proof}

\subsection{Nonexistence of Hopf elements}\label{ssec:d2ha}

The purpose of this subsection is to prove the following.

\begin{thm}\label{thm:nohopf}
For an arbitrary base field $F$ of characteristic not equal to $2$, there are differentials of the form
\[
d_2(h_{a+1}) = (h_0+\rho h_1)h_a^2
\]
in the $F$-motivic Adams spectral sequence, which are nonzero for $a\geq 3$.
\tqed
\end{thm}
By naturality, it suffices to produce these differentials in the case where $F$ is a prime field, i.e.\ $F = \f_q$ or $F = \q$, and when $F$ is algebraically closed. Moreover, by the Hasse principal given in \cref{prop:hasse}, the case $F = \q$ may be deduced from the cases $F = \q_p$ and $F = \r$ combined. All of these build on the case where $F$ is algebraically closed, which may be treated as follows.

\begin{proposition}\label{prop:algclosedhopf}
If $F = \Fbar$ is algebraically closed, then
\[
d_2(h_{a+1}) = h_0 h_a^2.
\]
This is nonzero for $a\geq 3$.
\end{proposition}
\begin{proof}
The corresponding classical differentials are known due to work of Adams \cite{Ada60}. The proposition could be reduced to this by appealing to \cref{cor:betti}; however, we shall instead proceed as follows.

In \cite[Section 3]{Wan67}, Wang gives another proof of the classical differentials, combining only a minimal amount of homotopical input with a good understanding of $\Ext_{\cl}$. His argument may be applied essentially verbatim to produce the claimed $\Fbar$-motivic differentials. It is this argument that may be adapted to work for other base fields, so to motivate our later computations let us recall this argument in full.

The proof proceeds by induction on $a$, where only the base case requires any homotopical input.

Consider the base case $a=3$. The class $h_3$ is a permanent cycle, detecting the Hopf element $\sigma$; see \cref{lem:lowhopf}. By \cref{lem:gradedcomm}, we find that $2\sigma^2 = 0$. As $2$ is detected by $h_0$ over algebraically closed fields, it follows that $h_0h_3^2$ cannot survive the Adams spectral sequence. The structure of $\Ext_{\Fbar}$ implies that $d_2(h_4) = h_0 h_3^2$ is the only way for $h_0h_3^2$ to die.

Now, suppose we have produced the differential $d_2(h_a) = h_0h_{a-1}^2$ for some $n\geq 4$. The relation $h_{a+1}h_a=0$ together with the Leibniz rule implies
\[
0 = d_2(h_{a+1}h_a) = d_2(h_{a+1})\cdot h_a+h_{a+1}\cdot d_2(h_a).
\]
Applying our inductive hypothesis and the relation $h_{a+1}\cdot h_{a-1}^2 = h_a^3$, this reduces to
\[
(d_2(h_{a+1}) + h_0 h_a^2)\cdot h_a = 0.
\]
The algebraic structure of $\Ext_F^3$ implies that $d_2(h_{a+1})\in\f_2\{h_0h_a^2\}$, so it suffices to verify that $h_0 h_a^3 \neq 0$ for $a\geq 4$. This follows from Wang's computation \cite[Proposition 3.4]{Wan67} by comparison along the map $\Ext_F\rightarrow\Ext_F[\tau^{-1}]\simeq\Ext_\cl[\tau^{\pm 1}]$, and this concludes the proof.
\end{proof}

The base step for the inductive argument given in \cref{prop:algclosedhopf} works for arbitrary base fields, but the inductive step falls apart. This inductive step relies on the algebraic fact that, when working over an algebraically close field, multiplication by $h_a$ is injective on the degree of $d_2(h_{a+1})$ for $a\geq 4$. Over other base fields, this fails for $a = 4$: this degree may contain elements of the form $\omega h_1 h_4^2$ where $\omega\in\Ext_F^{-1,0,-1}$ is a sum of elements such as $\rho$, $\pi$, and $u$, and
\[
\omega h_1 h_4^2\cdot h_4 = \omega h_1 \cdot h_4^3 = \omega h_1 \cdot h_3^2 \cdot h_5 = 0.
\]
Luckily, the inductive step fails only for $a=4$; once we have resolved $d_2(h_5)$, the remaining differentials will follow via the same argument. To resolve this differential, we proceed as follows.

\begin{proposition}\label{lem:d2h5f}
Let $F$ be a field of the form $\f_q$ for $q$ odd, $\q_p$ for any $p$, or $\r$. Then there is a differential
\[
d_2(h_5) = (h_0+\rho h_1)h_4^2
\]
in the $F$-motivic Adams spectral sequence.
\end{proposition}
\begin{proof}
When $F = \r$, we first make the following reduction. Observe that $\Ext_\r$ in the degree of $d_2(h_5)$ is given by $\f_2\{h_0h_4^2,\rho h_1 h_4^2\}$, and that neither of these classes are divisible by $\rho^2$. Thus it is sufficient to verify this differential in the Adams spectral sequence for the cofiber of $\rho^2$. By \cite[Lemma 7.8]{BS20}, this cofiber is a ring spectrum, and so its Adams spectral sequence is multiplicative. Having made this reduction, the remainder of the argument is uniform in the given choices of $F$. For brevity of notation, in the following we shall write $\Ext_F$ for the object so named when $F = \f_q$ or $F = \q_p$, and write the same for $\Ext_{(2)}$ when $F = \r$.

First, observe that as $\tau^4\in\Ext_F^0$, the class $\tau^{16}$ is a square and thus a $d_2$-cycle. As $\tau^{16}$ acts injectively on $\Ext_F^f$ for $f\leq 3$, it suffices to show
\[
d_2(\tau^{16} h_5) = (h_0+\rho h_1)\tau^{16} h_4^2.
\]
 Consider the Hurewicz map
$
c\colon \pi_\ast \rightarrow \pi_{\ast,0}^F.
$
Let $\theta_4\in\pi_{30}S^0$ be the Kervaire class, detected by $h_4^2$ and satisfying $2\theta_4=0$. By \cref{lem:diagonal}, we find that $c(\theta_4)$ is detected by $(\Sq^0)^4(h_0^2) = \tau^{16}h_4^2$. As $2 \cdot c(\theta_4) = 0$, the class $(h_0 + \rho h_1)\tau^{16}h_4^2$ cannot survive. The only possibility is that $d_2(\tau^{16}h_4) = (h_0+\rho h_1)\tau^{16}h_4^2$, yielding the desired differential.
\end{proof}

\begin{remark}
When $F = \r$, the differential $d_2(h_5)$, and in fact all the differentials $d_2(h_{a+1})$, may also be produced as follows. By comparison with $\c$, one finds $d_2(h_5) \in h_0 h_4^2 + \f_2\{\rho h_1h_4^2\}$. Thus it suffices to verify that $d_2(h_5)$ is not $\rho$-torsion. This is a consequence of the fact that the isomorphism $\Ext_\r[\rho^{-1}]\simeq\Ext_{\dcl}[\rho^{\pm 1}]$ \cite[Theorem 4.1]{DI17} commutes with Adams differentials.
\tqed
\end{remark}

We need just one more algebraic fact for the proof of \cref{thm:nohopf}.

\begin{lemma}\label{lem:h1hn3}
Let $\omega\in \Ext_F^0$ be nonzero. Then $\omega h_1h_a^3\neq 0$ for all $a\geq 5$.
\end{lemma}
\begin{proof}
The class $h_0h_{a-1}^3$ is nonzero in $\Ext_{\cl}$ for $a\geq 5$ by \cite[Proposition 3.4]{Wan67}. \cref{prop:retract} gives an injection $\Ext_\dcl\rightarrow\Ext_F$, and this extends by linearity to an injection $\Ext_F^0\otimes_{\f_2}\Ext_\dcl\rightarrow \Ext_F$, as can be seen by using \cref{lem:mmodule} to reduce to the injections $\Ext_{(r)}^0\otimes_{\f_2}\Ext_\dcl\rightarrow \Ext_{(r)}$. The class $\omega h_1 h_a^3$ is the image of $\omega\otimes h_0h_{a-1}^3$ under this map, yielding the claim.
\end{proof}

We may now give the following.

\begin{proof}[Proof of \cref{thm:nohopf}]
As discussed, it suffices to consider only the cases where $F$ is of the form $\f_q$ for some $q$ odd, $\q_q$ for some $q$, or $\r$. So let $F$ be one of these. We now induct on $a$, with base cases $a=3$ and $a=4$.

First, consider the case $a=3$. By \cref{lem:lowhopf}, the class $h_3$ is a permanent cycle detecting the class $\sigma$. By \cref{lem:gradedcomm}, $2\sigma^2 = 0$, and so $(h_0+\rho h_1)h_3^2$ must be the target of a differential. The only possibility is that $d_2(h_4) = (h_0+\rho h_1)h_3^2$.

The case $a=4$ was handled in \cref{lem:d2h5f}.

Now, suppose inductively that we have produced the differential $d_2(h_a) = (h_0+\rho h_1)h_{a-1}^2$ for some $a\geq 5$. Combining the Leibniz rule with the relation $h_{a+1}h_a = 0$, we find
\[
0 = d_2(h_{a+1}h_a) = d_2(h_{a+1})h_a + h_{a+1}d_2(h_a).
\]
Applying our inductive hypothesis and the relation $h_{a+1}h_{a-1}^2 = h_a^2$, we find
\[
(d_2(h_{a+1}) + (h_0+\rho h_1)h_a^2)h_a = 0.
\]
It follows that $d_2(h_{a+1})=(h_0+\rho h_1)h_a^2+x$ where $x$ is some class killed by $h_a$. The only classes in this degree are $h_0 h_a^2$ and those of the form $\omega h_1 h_a^2$ where $\omega\in \Ext_F^0$. By comparison with $\Fbar$, we find that $x$ must be zero or a nonzero class of the form $\omega h_1 h_a^2$ with $\omega\in\Ext_F^{-1,0,-1}$. As $a\geq 5$, \cref{lem:h1hn3} implies that none of the latter are killed by $h_a$. Thus $x = 0$, yielding the desired differential.
\end{proof}

This concludes our uniform analysis of differentials out of $\Ext_F^1$. The rest of this section is dedicated to studying the $1$-line in more detail for particular fields $F$.

\subsection{The real numbers}

We now study the case $F = \r$ in more detail. Recall from \cref{thm:extadd} that
\[
\Ext_\r^1 = \f_2[\rho]\{h_a:a\geq 1\}\oplus\bigoplus_{a\geq 0}\f_2[\rho]/(\rho^{2^a})\{\tau^{\floor{2^{a-1}(4n+1)}}h_a:n\geq 0\}.
\]
Here, recall that $2^{a-1}(4n+1) = 2n$ for $a = 0$. \cref{thm:nohopf} allows one to understand the fate of the classes in the $\rho$-torsion-free summand, so we turn our attention to the $\rho$-torsion subgroup. We shall first pin down which of these $\rho$-torsion classes are permanent cycles, and then by separate methods compute all $d_2$-differentials on these $\rho$-torsion classes. A comparison reveals that there must be numerous higher differentials, but determining these is outside the scope of our computation. The first point of order is the following.

\begin{definition}
For $a\geq 0$, write $a = c + 4 d$ with $0 \leq c \leq 3$, and define $\psi(a) = 2^c + 8 d$ to be the \emph{$a$th Radon--Hurwitz number}.
\tqed
\end{definition}

\begin{proposition}\label{thm:rpc}
The class $\rho^r \tau^{2^{a-1}(4n+1)}h_a$ is a permanent cycle if and only if $r \geq 2^a - \psi(a)$.
\tqed
\end{proposition}

The proof of \cref{thm:rpc} requires some preliminaries. We proceed by comparison with Borel $C_2$-equivariant stable homotopy theory. Let $\Ext_{BC_2}$ denote the $E_2$-page of the Borel $C_2$-equivariant Adams spectral sequence \cite{Gre88}. Explicitly,
\[
\Ext_{BC_2}^{s,f,w} = \Ext_{\ca^{\cl}}^{s-w,f}(\f_2,H^\ast P^\infty_{w});
\]
this is just a combination of the ordinary Adams spectral sequences for the stable cohomotopy groups of infinite stunted projective space. By Lin's positive resolution of the Segal conjecture \cite{Lin80}, this spectral sequence converges to $\pi_{\ast,\ast}^{C_2}$, the homotopy groups of the $2$-completion of the $C_2$-equivariant sphere spectrum.

Betti realization followed by Borel completion yields a functor from the stable $\r$-motivic category to the Borel $C_2$-equivariant stable category $\Fun(BC_2,\Sp)$, and Behrens--Shah \cite[Section 8]{BS20} show that this may be understood as completing at $\rho$ and inverting $\tau$. Applying this to an Adams resolution, we find that
\[
\Ext_{BC_2} = \lim_{n\rightarrow\infty} \Ext_{(2^n)}[\tau^{-2^n}].
\]
The simple form of $\Ext_\r^{\leq 3}$ allows us to immediately read off $\Ext_{BC_2}^{\leq 3}$.

\begin{lemma}\label{lem:extbc2}
$\Ext_{BC_2}^{\leq 3}$ is exactly as $\Ext_\r^{\leq 3}$ is described in \cref{thm:extadd}, only where $n$ is allowed to be negative, and in place of the map $\Ext_\r\rightarrow\Ext_\c$ is a map $\Ext_{BC_2}\rightarrow\Ext_\c[\tau^{-1}]\cong\Ext_{\cl}[\tau^{\pm 1}]$.
\qed
\end{lemma}

In particular,
\[
\Ext_{BC_2}^1 = \f_2[\rho]\{h_a:a\geq 1\}\oplus\bigoplus_{a\geq 0}\f_2[\rho]/(\rho^{2^a})\{\tau^{\floor{2^{a-1}(4n+1)}}h_a:n\in\z\}.
\]
We have introduced $\Ext_{BC_2}$ in order to make the following reduction.

\begin{lemma}\label{lem:pconly}
Write $h\colon \Ext_\r\rightarrow \Ext_{BC_2}$ for the canonical map of spectral sequences. Fix a $\rho$-torsion class $x\in\Ext_\r^1$. Then $x$ is a permanent cycle if and only if $h(x)$ is a permanent cycle.
\end{lemma}
\begin{proof}
Clearly, if $x$ is a permanent cycle, then the same must be true of $h(x)$. Conversely, suppose that $h(x)$ is a nontrivial permanent cycle; we claim that $x$ is a permanent cycle.

Write $\Ext_{C_2}$ for the $E_2$-page of the $C_2$-equivariant Adams spectral sequence \cite[Section 6]{HK01}, converging to the same target as $\Ext_{BC_2}$. This splits additively as $\Ext_{C_2}  = \Ext_\r\oplus \Ext_{NC}$ for a certain summand $\Ext_{NC}$ (see \cite[Section 2]{GHIR19}), and $h$ factors as $h = g\circ f\colon \Ext_\r\rightarrow\Ext_{C_2}\rightarrow\Ext_{BC_2}$, the first map being the obvious inclusion and the second map killing the summand $\Ext_{NC}$.

As $h(x)$ is a nontrivial permanent cycle, it detects a class $\alpha$ in Borel Adams filtration $1$. The class $\alpha$ must then be detected in $\Ext_{C_2}^{\leq 1}$. By \cite{BGI21}, the map $\Ext_{\r}\rightarrow\Ext_{C_2}$ is an isomorphism in the degrees under consideration, so the same must be true for $\Ext_{C_2}\rightarrow\Ext_{BC_2}$. As there is at most one nonzero $\rho$-torsion class in these degrees, the only possibility is that $\alpha$ is detected by $f(x)$ in $\Ext_{C_2}^1$, implying that $f(x)$ is a permanent cycle. As $\Ext_\r\rightarrow\Ext_{C_2}$ is the inclusion of a summand, this implies that $x$ is a permanent cycle, as claimed.
\end{proof}

Thus it suffices to understand permanent cycles in $\Ext_{BC_2}^1$. The main point is the following.

\begin{lemma}\label{lem:split}
There exists a nonzero $\rho$-torsion class $\alpha\in\pi_{s,w}^{C_2}$ detected in Borel Adams filtration $1$ if and only if the inclusion of the bottom cell of $P^{w-1}_{w-s-1}$ is split, where $P^n_k$ is the Thom spectrum of the $k$-fold Whitney sum of the tautological line bundle over the real projective space $\r P^n$. 
\end{lemma}
\begin{proof}
First, suppose given such a map $\alpha$. The structure of $\Ext_{BC_2}^1$ implies that $\alpha$ must have $\rho$-torsion exponent $s+1$, and so there is a lift $\alphabar$ in the diagram
\begin{center}\begin{tikzcd}
\Sigma^{s-w+1}P^{w-1}_{w-s-1}\ar[dr,dashed,"\alphabar"]\\
\Sigma^{s-w}P^\infty_w\ar[u,"\partial"]\ar[r,"\alpha"]&S^0\\
\Sigma^{s-w}P^\infty_{w-s-1}\ar[ur,"\rho^{s+1}\alpha=0"']\ar[u]
\end{tikzcd}.\end{center}
As $\alpha$ and $\partial$ have Adams filtration $1$, necessarily $\alphabar$ has Adams filtration $0$. It follows that precomposing $\alphabar$ with the inclusion of the bottom cell $S^0\rightarrow\Sigma^{s-w+1}P^{w-1}_{w-s-1}$ gives a map $S^0\rightarrow S^0$ which is nonzero in mod $2$ cohomology, and must therefore be an equivalence. In other words, $\alphabar$ splits off the bottom cell of $P^{w-1}_{w-s-1}$.

Conversely, if the inclusion of the bottom cell of $P^{w-1}_{w-s-1}$ is split, then its splitting gives a nonzero map $\alphabar$ as above in Adams filtration $0$. Let $\alpha = \alphabar\circ\partial$; we claim that $\alpha$ is a nonzero class detected in Adams filtration $1$. Indeed, the cofibering $P^{w-1}_{w-s-1}\rightarrow P^\infty_{w-s-1}\rightarrow P^\infty_w$ gives an exact sequence
\begin{center}\begin{tikzcd}[column sep=small]
\Ext^0(\f_2,H^\ast P^\infty_w)\ar[r]&\Ext^0(\f_2,H^\ast P^\infty_{w-s-1})\ar[r]&\Ext^0(\f_2,H^\ast P^{w-1}_{w-s-1})\ar[r,"\partial'"]&\Ext^1(\f_2,H^\ast P^\infty_w)
\end{tikzcd},\end{center}
where $\partial'$ models restriction along $\partial$ in the previous diagram. The first map is exactly
\[
\rho^{s+1}\colon \Ext^{\ast,0,w}_{BC_2}\rightarrow\Ext^{\ast,0,w-s-1}_{BC_2}.
\]
As $\Ext^0_{BC_2} = \f_2[\rho]$, we find that the kernel of $\partial'$ consists of only that class represented by the inclusion $\f_2\rightarrow H^0 P^{w-1}_{w-s-1}$. So $\partial'$ is injective in the relevant degrees, implying that $\alpha$ is nonzero and of Adams filtration $1$ as claimed.
\end{proof}

We may now give the following.

\begin{proof}[Proof of \cref{thm:rpc}]
By \cref{lem:pconly}, it suffices to show that a class $\rho^r \tau^{\floor{2^{a-1}(4n+1)}}h_a\in\Ext_{BC_2}^1$ is a permanent cycle if and only if $r\geq 2^a-\psi(a)$. By sparseness of $\Ext_{BC_2}^1$, the class $\rho^r\tau^{\floor{2^{a-1}(4n+1)}}h_a$ is a permanent cycle if and only if there is some $\rho$-torsion class $\alpha\in\pi_{2^a-r-1,-2^{a+1}n-r}^{C_2}$ detected in Borel Adams filtration $1$. By \cref{lem:split}, this holds if and only if inclusion of the bottom cell of $P^{-2^{a+1}n-r-1}_{-2^{a+1}n-2^a}$ is split. By James periodicity \cite{Jam58, Jam59}, this holds if and only if the inclusion of the bottom cell of $P^{2^N-2^{a+1}n-r-1}_{2^N-2^{a+1}-2^a}$ is split for some sufficiently large $N\gg 0$; that is, we may assume ourselves to be working with suspension spectra of honest real projective spaces. When this happens was resolved by Adams' solution of the vector fields on spheres problem \cite[Theorem 1.2]{Ada62}, yielding the condition claimed.
\end{proof}

\begin{corollary}\label{cor:tauhopfexist}
The classes $\tau^{\floor{2^{a-1}(4n+1)}}h_a$ are permanent cycles for $a\leq 3$.
\qed
\end{corollary}

\cref{cor:tauhopfexist} could also be proved more directly, applying the technique used in the proof of \cref{thm:nohopf} or \cref{prop:tauhopfdiff} below to reduce to the region considered by Belmont--Isaksen.

It is worth summarizing what we have learned from the proof of \cref{thm:rpc} about the stable cohomotopy groups of projective spaces.

\begin{thm}\label{thm:rpcohomotopy}
The subgroup of permanent cycles in $\Ext_{BC_2}^1$ is given by
\[
\f_2[\rho]\{h_1,h_2,h_3,\rho h_4\}\oplus\bigoplus_{a\geq 0}\f_2[\rho]/(\rho^{\psi(a)})\{\rho^{2^a-\psi(a)}\tau^{\floor{2^{a-1}(4n+1)}}h_a:n\in\z\}.
\]
A choice of maps $\Sigma^c P^\infty_w\rightarrow S^0$ detected by these permanent cycles is given by the following.
\begin{enumerate}
\item For all $r\geq 0$, there are maps
\begin{center}
\begin{tikzcd}[row sep=tiny]
P^\infty_{1-r}\ar[r]&P^\infty_1\ar[r,"\eta"]&S^0 \\
\Sigma P^\infty_{2-r}\ar[r]&\Sigma P^\infty_2\ar[r,"\nu"]&S^0 \\
\Sigma^3 P^\infty_{4-r}\ar[r]&\Sigma^3 P^\infty_3\ar[r,"\sigma"]&S^0
\end{tikzcd}.
\end{center}
Here, $\eta$, $\nu$, and $\sigma$ are equivariant refinements of the Hopf maps with the same names. The above composites are detected by by $\rho^r h_1$, $\rho^r h_2$, and $\rho^r h_3$, respectively.
\item For all $r\geq 0$, there is a map
\begin{center}\begin{tikzcd}
\Sigma^7P^\infty_{7-r}\ar[r]&\Sigma^7 P^\infty_7\ar[r,"\Sq(\sigma)"]&S^0
\end{tikzcd},\end{center}
where $\Sq(\sigma)$ is the symmetric square of $\sigma\colon S^7\rightarrow S^0$. The above composite is detected by $\rho^{1+r} h_4$.
\item For all $a\geq 0$, $n\in\z$, and $1\leq r \leq \psi(a)$, there are maps
\begin{center}\begin{tikzcd}
\Sigma^{2^a(2n+1)-1}P^\infty_{-2^a(2n+1)+r}\ar[r,"\partial"]&\Sigma^{2^a(2n+1)}P^{-2^a(2n+1)+r-1}_{-2^a(2n+1)}\ar[r,"s"]&S^0
\end{tikzcd}.\end{center}
Here, $\partial$ is the cofiber of $\Sigma^{2^a(2n+1)-1}P^\infty_{-2^a(2n+1)}\rightarrow\Sigma^{2^a(2n+1)-1}P^\infty_{-2^a(2n+1)+r}$, and $s$ is any map splitting off the bottom cell of $P^{-2^a(2n+1)+r-1}_{-2^a(2n+1)}$. The above composite is detected by $\rho^{2^a-r}\tau^{\floor{2^{a-1}(4n+1)}}h_a$.
\end{enumerate}
\end{thm}
\begin{proof}
Recall that
\[
\Ext_{BC_2}^1 = \f_2[\rho]\{h_a:a\geq 1\}\oplus\bigoplus_{a\geq 0}\f_2[\rho]/(\rho^{2^a})\{\tau^{\floor{2^{a-1}(4n+1)}} h_a:n\in\z\}.
\]
We have just analyzed which classes in the $\rho$-torsion summand are permanent cycles, leading to exactly the claimed $\rho$-torsion permanent cycles with representatives as described in (3). \cref{lem:lowhopf} implies that $h_1$, $h_2$, and $h_3$ are permanent cycles, and these detect the maps described in (1). \cref{thm:nohopf} shows that $\rho^n h_a$ supports a $d_2$-differential for $a\geq 5$ and $n\geq 0$, and that $h_4$ supports a $d_2$-differential but $\rho h_4$ does not. We are left with verifying that $\rho h_4$ is a permanent cycle detecting the map $\Sq(\sigma)$. Indeed, taking geometric fixed points yields an isomorphism $\pi_{\ast,\ast}^{C_2}[\rho^{-1}]\cong \pi_\ast^{\cl}[\rho^{\pm 1}]$ which sends $\Sq(\alpha)$ to $\alpha$ for any $\alpha\in \pi_\ast^{\cl}$. This isomorphism is modeled on Adams spectral sequences by $\Ext_{C_2}[\rho^{-1}] \cong \Ext_\r[\rho^{-1}]\cong \Ext_\cl[\rho^{\pm 1}]$. As $\rho h_4$ is the only class in its degree lifting $h_3\in \Ext_\cl^1$, it must be that $\rho h_4$ detects $\Sq(\sigma)$.
\end{proof}

\cref{thm:rpc} implies that the classes $\tau^{2^{a-1}(4n+1)}h_a$ must support Adams differentials for $a\geq 4$. Although we do not compute all these differentials, we do give the following.

\begin{proposition}\label{prop:tauhopfdiff}
For all $n\geq 0$ and $a\geq 3$, there is a differential
\[
d_2(\tau^{2^a(4n+1)}h_{a+1}) = (h_0+\rho h_1)(\tau^{2^{a-1}(4n+1)}h_a)^2.
\]
\end{proposition}
\begin{proof}
We give separate arguments for the case $a = 3$ and $a > 3$. First consider the case $a = 3$. The class $\tau^{4(4n+1)}h_3$ is a permanent cycle by \cref{cor:tauhopfexist}, detecting a class which we might call $\tau^{4(4n+1)}\sigma$. By \cref{lem:gradedcomm}, $2\cdot (\tau^{4(4n+1)}\sigma)^2 = 0$, and so $(h_0 + \rho h_1)\cdot (\tau^{4(4n+1)}h_3)^2$ must die. This class is not divisible by $\rho$, and the only non-$\rho$-divisible classes that may hit it are $\tau^8 h_4$ and $\tau^8 h_4 + \rho^{16}h_5$. By \cref{thm:nohopf}, if $d_2(\tau^8 h_4 + \rho^{16}h_5) = (h_0 + \rho h_1)\cdot (\tau^{4(4n+1)}h_3)^2$, then $d_2(\tau^8 h_4) = (h_0 + \rho h_1)\cdot (\tau^{4(4n+1)}h_3 + h_4)^2$. This is not possible, as $\tau^8 h_4$ is $\rho$-torsion and this target is not. Thus in fact $d_2(\tau^8 h_4) = (h_0 + \rho h_1)\cdot (\tau^{4(4n+1)}h_3)^2$ as claimed.

Next consider the case $a > 3$. The $\rho$-torsion subgroup of $\Ext_\r$ in the degree of $d_2(\tau^{2^a(4n+1)}h_{a+1})$ is given by $\f_2\{h_0,\rho h_1\}\otimes\f_2\{(\tau^{2^{a-1}(4n+1)}h_a)^2\}$. These classes are not divisible by $\rho^2$, and so it suffices to verify the differential in the Adams spectral sequence for the cofiber of $\rho^2$. By \cite[Lemma 7.8]{BS20}, this cofiber is a ring spectrum, so its Adams spectral sequence is multiplicative. As $\tau^2$ is a cycle, $\tau^4$ is a $d_2$-cycle, so we reduce to showing $d_2(h_{a+1}) = (h_0+\rho h_1) h_a^2$. This was shown in \cref{thm:nohopf}.
\end{proof}

We may summarize what we have learned as follows.

\begin{thm}\label{thm:r}
The nontrivial $d_2$-differentials out of the $1$-line of the $\r$-motivic Adams spectral sequence are exactly those given in the following table.
\begin{longtable}{llll}
\toprule
Source & Target & Constraints\\
\midrule \endhead
\bottomrule \endfoot
$h_4$ & $h_0h_3^2$ \\
$\rho^r h_a$ & $\rho^r(h_0+\rho h_1)h_{a-1}^2$ & $a\geq 5$, $r\geq 0$ \\
$\rho^r \tau^{2^{a-1}(4n+1)}h_a$ & $\rho^r(h_0+\rho h_1) (\tau^{2^{a-2}(4n+1)}h_{a-1})^2$ & $n\geq 0$, $a\geq 4$, $0\leq r \leq 2^{a-1}-1$
\end{longtable}
The $1$-line of the $E_3$-page of the $\r$-motivic Adams spectral sequence has a basis given by the elements in the following table.
\begin{longtable}{lll}
\toprule
$\f_2[\rho]$-module generator & Constraints & $\rho$-torsion exponent\\
\midrule \endhead
\bottomrule \endfoot
$h_a$ & $a\in\{1,2,3\}$ & $\infty$ \\
$\rho h_4$ & & $\infty$ \\
$\tau^{\floor{2^{a-1}(4n+1)}}h_a$ & $n\geq 0$ and $a\in\{0,1,2,3\}$ & $2^a$ \\
$\rho^{2^{a-1}-1}\tau^{2^{a-1}(4n+1)}h_a$ & $n\geq 0$ and $a\geq 4$ & $2^{a-1}+1$
\end{longtable}
Those classes in $\Ext_\r^1$ which are permanent cycles are given in the following table.
\begin{longtable}{llll}
\toprule
$\f_2[\rho]$-module generator & Constraints & $\rho$-torsion exponent & Stem\\
\midrule \endhead
\bottomrule \endfoot
$h_a$ & $a\in\{1,2,3\}$ & $\infty$ & $2^a-1$\\
$\rho h_4$ & & $\infty$ & $14$ \\
$\tau^{\floor{2^{a-1}(4n+1)}}h_a$ & $n\geq 0$ and $a\in\{0,1,2,3\}$ & $2^a$ & $2^a-1$ \\
$\rho^{2^a-\psi(a)}\tau^{2^{a-1}(4n+1)}h_a$ & $n\geq 0$, $a\geq 4$ & $\psi(a)$ & $\psi(a)-1$
\end{longtable}
\end{thm}
\begin{proof}
All of this is immediate from \cref{thm:nohopf}, \cref{thm:rpc}, \cref{prop:tauhopfdiff}, \cref{thm:rpcohomotopy}, and the $\rho$-torsion exponent of generators of $\Ext_\r^3$ given in \cref{thm:extadd}, with the following exception: \cref{prop:tauhopfdiff} produces differentials $d_2(\tau^{8(4n+1)}h_4) = (h_0 + \rho h_1)(\tau^{4(4n+1)}h_3)^2$, and one must use \cref{prop:3linerelations}(4) to check that this target has $\rho$-torsion exponent $7$.
\end{proof}

\subsection{Finite fields}

We now study the case where $F$ is a finite field. For the most part, this case follows by combining \cref{thm:nohopf} with differentials out of $\Ext_F^0$ that may be deduced from work of Kylling \cite{Kyl15}. By naturality, our discussion in this subsection gives information for $F$ an arbitrary field of odd characteristic.

We will need the following definition.

\begin{definition}\label{def:al}
For an integer $q$, let $\nu_2(q)$ denote the $2$-adic valuation of $q$, i.e.\
\[
q = 2^{\nu_2(q)}(2n+1)
\]
for some integer $n$, and let 
\[
\varepsilon(q) = \nu_2(q-1),\qquad\lambda(q) = \nu_2(q^2-1).
\]
\tqed
\end{definition}

We now split into cases based on congruence of the order of the field mod $4$.

\subsubsection{$q\equiv 1 \pmod{4}$}

Fix a prime-power $q$ such that $q\equiv 1 \pmod{4}$. We work over $F = \f_q$. Recall that $\Ext_{\f_q} = \Ext_{(1)}\{1,u\}$. In particular,
\[
\Ext_{\f_q}^1 = \f_2[\tau]\{1,u\}\otimes\f_2\{h_a:a\geq 0\}.
\]
The class $u$ is a permanent cycle for degree reasons, and we have already computed the differential on all the classes $h_a$. However the story does not stop there; instead, we have the following.

\begin{lemma}\label{lem:dtau1}
There are differentials
\[
d_{\varepsilon(q)+s}(\tau^{2^s}) = u  \tau^{2^s-1}h_0^{\varepsilon(q)+s}
\]
for all $s\geq 0$.
\end{lemma}
\begin{proof}
In \cite[Lemma 4.2.1]{Kyl15}, Kylling produces identical differentials in the $\f_q$-motivic Adams spectral sequence for $H\z$. The claimed differentials follow by naturality.
\end{proof}

This may be combined with \cref{thm:nohopf} to easily compute all differentials out of the $1$-line. 

\begin{thm}\label{thm:enf1}
For $q\equiv 1 \pmod{4}$, the $1$-line of the $\f_q$-motivic Adams spectral sequence supports only the nontrivial differentials given in the following table.
\begin{longtable}{llll}
\toprule
Source & $d_r$ & Target & Constraints\\
\midrule \endhead
\bottomrule \endfoot
$\tau^n h_0$ & $d_{\varepsilon(q)+\nu_2(n)}$ & $\tau^{n-1}h_0^{\varepsilon(q)+\nu_2(n)+1}$ & $n\geq 1$ \\
$\tau^{2n+1}h_2$ & $d_2$ & $u \tau^{2n}h_2 h_0^2$ & $n\geq 0$, $\varepsilon(q) = 2$ \\
$\tau^{2n+1}h_3$ & $d_2$ & $u\tau^{2n}h_3h_0^2$ & $n\geq 0$, $\varepsilon(q) = 2$ \\
$\tau^{2n+1}h_3$ & $d_3$ & $u\tau^{4n+1}h_3h_0^3$ & $n\geq 0$, $\varepsilon(q) = 3$
\\
$\tau^{4n+2}h_3$ & $d_3$ & $u\tau^{4n+1}h_3h_0^3$ & $n\geq 0$, $\varepsilon(q) = 2$ \\
$\tau^n h_b$ & $d_2$ & $\tau^nh_0 h_{b-1}^2 + d_2(\tau^n)h_b$ & $n\geq 0$, $b\geq 4$ \\
$u \tau^n h_b$ & $d_2$ & $u \tau^n h_0 h_{b-1}^2$ & $n\geq 0$, $b\geq 4$
\end{longtable}
After these have been run, the $1$-line of the $E_\infty$-page of the $\f_q$-motivic Adams spectral sequence has a basis given by the elements in the following table.
\begin{longtable}{lll}
\toprule
Class & Constraints &\\
\midrule \endhead
\bottomrule \endfoot
$h_0$ \\
$\tau^n h_1$ & $n\geq 0$\\
$\tau^n h_2$ & $n\geq 0$, where if $\varepsilon(q) = 2$ then $n \equiv 0 \pmod{2}$ \\
$\tau^n h_3$ & $n\geq 0$, where if $\varepsilon(q) = 2$ then $n \equiv 0 \pmod{4}$, \\
&\hfill and if $\varepsilon(q)=3$ then $n\equiv 0\pmod{2}$\hphantom{,} \\
$u \tau^n h_b$ & $n\geq 0$, $b\in\{0,1,2,3\}$
\end{longtable}
\end{thm}
\begin{proof}
The first four families of differentials follow immediately from \cref{lem:dtau1} and \cref{lem:lowhopf}, and the remaining two by combining \cref{lem:dtau1} with \cref{thm:nohopf}. Note in particular that $d_2(\tau^n) \equiv 0 \pmod{u}$, and thus $d_2(\tau^n h_b) \neq 0$ for $b\geq 4$. The second table may be easily read off the first, provided we verify that we have not missed any differentials, i.e.\ that the classes listed in the second table are indeed permanent cycles. For degree reasons, the only possible nontrivial differentials on the classes $\tau^n h_b$ with $b\in\{1,2,3\}$ would be of the form
\begin{enumerate}
\item $d_r(\tau^n h_1)\overset{?}{=} \tau^{n-1}h_0^{r-1}$;
\item $d_2(\tau^n h_2) \overset{?}{=} u\tau^{n-1}h_0^2 h_2$; 
\item $d_2(\tau^n h_3) \overset{?}{=} u \tau^{n-1}h_0^2 h_3$; 
\item $d_3(\tau^n h_3) \overset{?}{=} u \tau^{n-1}h_0^3 h_3$;
\end{enumerate}
with $n\geq 1$. The first is impossible for $n=1$ as $h_0$ detects $2$ and thus no power of $h_0$ may be killed, and is impossible for $n\geq 2$ as the class $\tau^{n-1}h_0^{r-1}$ must support the differential given the first row of the first table.  The remaining three differentials may occur, and when they occur is accounted for in the given tables.
\end{proof}

\subsubsection{$q\equiv 3\pmod{4}$}

Now fix a prime-power $q$ such that $q\equiv 3 \pmod{4}$. We work over $F = \f_q$. Recall that $\Ext_{\f_q} = \Ext_{(2)}$.

\begin{lemma}\label{Lem:ExtFq01}
We may identify
\[
\Ext_{\f_q}^0 = \f_2[\tau^2,\rho,\tau\rho]/(\rho^2 = \rho\cdot (\tau\rho) = (\tau \rho)^2=0),
\]
and $\Ext_{\f_q}^1$ is the tensor product of $\f_2[\tau^2]$ with the following:
\[
\f_2\{h_0, \rho \tau \cdot h_0\}\oplus\f_2\{h_1,\rho\cdot h_1, \rho\tau\cdot h_1,\tau h_1\}\oplus\f_2\{h_b,\rho \cdot h_b, \rho\tau\cdot h_b : b\geq 2\}.
\]
\end{lemma}
\begin{proof}
This follows quickly from our computation of $\Ext_\r$, following the recipe of \cref{rmk:modrho}. Alternately, one may compute the $\rho$-Bockstein spectral sequence
\[
\Ext_{(1)}[\rho]/(\rho^2)\Rightarrow\Ext_{(2)}
\]
directly (cf.\ \cite{WO17}); the only relevant differential is $d_1(\tau) = \rho h_0$. 
\end{proof}

As in the previous case, powers of $\tau$ support arbitrarily long differentials.

\begin{lemma}\label{lem:dtau3}
There are differentials
\[
d_{\lambda(q)+s}(\tau^{2^{s+1}}) = \rho\tau^{2^{s+1}-1}h_0^{\lambda(q)+s}
\]
for all $s\geq 0$. On the other hand, $\rho\tau$ is a permanent cycle.
\end{lemma}
\begin{proof}
The class $\rho\tau$ is a permanent cycle for degree reasons. In \cite[Lemma 4.2.2]{Kyl15}, Kylling produces identical differentials in the $\f_q$-motivic Adams spectral sequence for $\f_q$-motivic $H\z$. The claimed differentials follow by naturality.
\end{proof}

\begin{thm}\label{thm:enf3}
For $q\equiv 3 \pmod{4}$, the $1$-line of the $\f_q$-motivic Adams spectral sequence supports the differentials given in the following table.
\begin{longtable}{llll}
\toprule
Source & $d_r$ & Target & Constraints\\
\midrule \endhead
\bottomrule \endfoot
$\tau^{2n}h_0$ & $d_{\lambda(q)+\nu_2(n)}$ & $\rho\tau^{2n-1}h_0^{\lambda(q)+\nu_2(n)+1}$ & $n\geq 1$ \\
$\tau^{4n+2}h_3$ & $d_3$ & $\rho\tau^{4n+1}h_0^3 h_3$ & $n\geq 0$, $\lambda(q) = 3$ \\
$\tau^{2n}h_b$ & $d_2$ & $\tau^{2n}(h_0+\rho h_1)h_{b-1}^2$ & $n\geq 1$, $b \geq 4$ \\
$\rho\tau^{2n+1}h_b$ & $d_2$ & $\rho\tau^{2n+1}h_0 h_{b-1}^2$ & $n\geq 0$, $b \geq 4$
\end{longtable}
After the $d_2$-differentials have been run, the $1$-line of the $E_3$-page of the $\f_q$-motivic Adams spectral sequence has a basis given by the classes in the following table.
\begin{longtable}{ll}
\toprule
Class & Constraints \\
\midrule \endhead
\bottomrule \endfoot
$h_0$ \\
$\rho^\epsilon\cdot\tau^{2n}h_b$ & $n\geq 0$, $\epsilon\in\{0,1\}$, $b\in\{1,2,3\}$\\
$\rho\tau^{2n+1}h_b$ & $n\geq 0$, $b\in\{0,1,2,3\}$ \\
$\rho^\epsilon\tau^{4n+1}h_1$ & $n\geq 0$, $\epsilon\in\{0,1\}$ \\
\hline
$\tau^{2n}h_0$ & $n \geq 1$ \\
\hline
$\rho\tau^{2n} h_b$ & $n\geq 0$, $b \geq 4$
\end{longtable}
Of these, all the classes in the first region are permanent cycles, with the exception that $\tau^{4n+2}h_3$ supports a $d_3$-differential if $\lambda(q) = 3$. The classes $\tau^{2n}h_0$ for $n \geq 1$ are not permanent cycles, and we leave open the fate of the classes $\rho\tau^{2n}h_b$ for $n\geq 1$ and $b\geq 4$.
\end{thm}
\begin{proof}
The given differentials follow quickly by combining \cref{thm:nohopf} with \cref{lem:dtau3}, and this accounts for all $d_2$-differentials. Note in particular that $\tau^2$ is a $d_2$-cycle as $\lambda(q)\geq 3$ whenever $q\equiv 3 \pmod{4}$. Thus the given $E_3$-page may be produced by linearly propagating the differentials of \cref{thm:nohopf}. Note also that $d_2(\rho\tau^{2n}h_b) =  \rho\tau^{2n}(h_0+\rho h_1)h_{b-1}^2 = 0$ for all $n\geq 0$ and $b\geq 4$, yielding the clases in final row of the second table.

It remains only to verify that the permanent cycles provided are indeed permanent cycles. As $\rho$ and $\rho\tau$ are permanent cycles for degree reasons, we may reduce to considering only the classes $\tau^{2n}h_b$, $\rho\tau^{2n+1}h_0$, and $\tau^{4n+1}h_1$ for $b\in\{1,2,3\}$ and $n\geq 0$. For degree reasons, the only possible nontrivial differentials supported by these classes would be of the form
\begin{enumerate}
\item $d_2(\tau^{2n}h_b) \overset{?}{=} \rho\tau^{2n-1}h_0^2 h_b$ for $b\in\{2,3\}$;
\item $d_3(\tau^{2n}h_3) \overset{?}{=} \rho\tau^{2n-1}h_0^3 h_3$;
\end{enumerate}
with $n\geq 1$. The first does not hold, as $\tau^2$ and $h_b$ are $d_2$-cycles. The second holds only when $\lambda(q) = 3$, and this is accounted for in the theorem statement.
\end{proof}

\subsection{The \texorpdfstring{$p$}{p}-adic rationals}

We now work over $F = \q_p$, the $p$-adic rationals. This is very similar to the case where $F = \f_q$, only where the additional input necessary to understand differentials out of $\Ext_{\q_p}^0$ comes from work of Ormsby \cite{Orm11} for $p$ odd and Ormsby--{\O}stv{\ae}r \cite{OO13} for $p=2$. The case where $p$ is odd turns out to entirely reduce to what we have already done.

\begin{lemma}\label{lem:qpdiff}
There are the following differentials in the $\q_p$-motivic Adams spectral sequence.
\begin{enumerate}
\item If $p\equiv 1 \pmod{4}$, then $d_{a(q)+s}(\tau^{2^s}) = u \tau^{2^s-1}h_0^{a(q)+s}$;
\item If $p\equiv 3 \pmod{4}$, then $d_{\lambda(q)+s}(\tau^{2^{s+1}}) = \rho\tau^{2^{s+1}-1}h_0^{\lambda(q)+s}.$
\end{enumerate}
\end{lemma}
\begin{proof}
In \cite[Theorem 5.2]{Orm11}, Ormsby produces identical differentials in the $\q_p$-motivic Adams spectral sequence for $BP\langle 0 \rangle$. The claimed differentials follow by naturality.
\end{proof}

We may summarize the situation as follows.

\begin{thm}\label{thm:qpf}
Fix an odd prime $p$, and consider the facts outlined about the $\f_p$-motivic Adams spectral sequence in \cref{thm:enf1} and \cref{thm:enf3}. The same facts hold for the $\q_p$-motivic Adams spectral sequence upon tensoring with $\f_p\{1,\pi\}$.
\end{thm}
\begin{proof}
The class $\pi$ is a permanent cycle for degree reasons, and the differentials given in \cref{lem:qpdiff} agree with those given in \cref{lem:dtau1} and \cref{lem:dtau3}. All of the work carried out over $\f_p$ then goes through verbatim, only where everything in sight has a twin copy indexed by $\pi$.
\end{proof}

\begin{remark}
The somewhat awkward phrasing of \cref{thm:qpf} is necessary as we did not wish to repeat two verbatim copies of both \cref{thm:enf1} and \cref{thm:enf3}, but we have not shown that the $1$-line of the $\q_p$-motivic Adams spectral sequence is a direct sum of two copies of the $1$-line of the $\f_p$-motivic Adams spectral sequence. The possible failure of this arises from the fact that when $p \equiv 3 \pmod{4}$, the classes $\rho\tau^{2n}h_b$ for $b\geq 4$ could support different higher differentials over $\f_p$ and $\q_p$. 
\tqed
\end{remark}

The case where $p = 2$ requires a separate analysis. Recall that
\[
\Ext_{\q_2} = \Ext_{(3)}\{1\}\oplus\Ext_{(1)}\{u,\pi\}.
\]

\begin{lemma}
We may identify
\[
\Ext_{(3)}^0 = \f_2(\tau^4,\rho\tau^2,\rho^2\tau,\rho^2\tau^3,\rho)\subset\f_2[\tau,\rho]/(\rho^3),
\]
and $\Ext_{(3)}^1$ is the tensor product of $\f_2[\tau^4]$ with the direct sum of the following modules:
\begin{gather*}
\f_2\{h_0,\tau^2 h_0,\rho^2\tau h_0,\rho^2\tau^3 h_0\}\\
\f_2\{1,\rho\}\otimes\f_2\{\tau h_1\}\oplus\f_2\{\rho\tau^3 h_1\}\oplus\f_2\{1,\rho,\rho^2,\rho\tau^2,\rho^2\tau^2,\rho^2\tau^3\}\otimes\f_2\{h_1\}  \\
\f_2\{1,\rho,\rho^2,\rho^2\tau,\rho^2\tau^3,\rho\tau^2,\rho^2\tau^2\}\otimes\f_2\{h_b:b\geq 2\}.
\end{gather*}
\end{lemma}
\begin{proof}
As with \cref{Lem:ExtFq01}, this follows from our computation of $\Ext_\r$ via the recipe in \cref{rmk:modrho}, or via the $\rho$-Bockstein spectral sequence; here, the relevant $\rho$-Bockstein differentials are $d_1(\tau) = \rho h_0$ and $d_2(\tau^2) = \rho^2 \tau h_1$.
\end{proof}

\begin{lemma}\label{lem:dq2}
The classes
\[
\tau^{4n+1}\rho^2,\qquad\tau^{2n}\rho,\qquad\tau^{4n+3}\rho^2,\qquad \pi\tau^n,\qquad u,\qquad u\tau^{2n+1}
\]
are permanent cycles. There are differentials
\[
d_{4+r}(\tau^{2^{r+2}}) = \pi\tau^{2^{r+2}-1}h_0^{4+r},\quad d_{3+r}(u\tau^{2^{r+1}}) = \rho^2\tau^{2^{r+1}-1}h_0^{3+r},\quad d_{3+r}(\tau^{2^{r+1}}h_0) = \pi\tau^{2^{r+1}-1}h_0^{4+r}
\]
for all $r\geq 0$.
\end{lemma}
\begin{proof}
In \cite[Lemma 5.7]{OO13}, Ormsby--{\O}stv{\ae}r compute differentials in the $\q_2$-motivic Adams spectral sequence for $BP\langle 0\rangle$.The claimed facts follow by comparison.
\end{proof}

\begin{thm}\label{thm:q2}
The $1$-line of the $\q_2$-motivic Adams spectral sequence supports the following nontrivial differentials.
\begin{longtable}{llll}
\toprule
Source & $d_r$ & Target & Constraints\\
\midrule \endhead
\bottomrule \endfoot
$\tau^{2n}h_0$ & $d_{3+\nu_2(n)}$ & $\pi\tau^{2n-1}h_0^{4+\nu_2(n)}$ & $n\geq 1$\\
$\tau^{4n}h_b$ & $d_2$ & $\tau^{4n}(h_0+\rho h_1)h_{b-1}^2$ & $n\geq 0$, $b\geq 4$ \\
$\rho \tau^{2n}h_b$ & $d_2$ & $\rho^2\tau^{2n}h_1h_{b-1}^2$ & $n\geq 0$, $b\geq 5$ \\
$u\tau^n h_b$ & $d_2$ & $u\tau^n h_0 h_{b-1}^2$ & $n\geq 0$, $b\geq 4$ \\
$\pi \tau^n h_b$ & $d_2$ & $\pi \tau^n h_0 h_{b-1}^2$ & $n\geq 0$, $b\geq 4$ \\
$u \tau^{4n+2}h_3$ & $d_3$ & $\rho^2\tau^{4n+1}h_0^3 h_3$ & $n\geq 0$
\end{longtable}
After all the $d_2$-differentials have been run, the $1$-line of the $E_3$-page of the $\q_2$-motivic Adams spectral sequence has a basis given by the classes in the following table.
\begin{longtable}{lll}
\toprule
Class & Constraints &\\
\midrule \endhead
\bottomrule \endfoot
$h_0$ \\
$\rho^\delta\tau^{4n}h_b$ & $n\geq 0$, $\delta\in\{0,1,2\}$, $b\in\{1,2,3\}$\\
$\rho^2\tau^{2n+1} h_0$ & $n \geq 0$ \\
$\rho^\epsilon \tau^{4n+1}h_1$ & $n\geq 0$, $\epsilon\in\{0,1\}$ \\
$\rho^{1+\epsilon}\tau^{4n+3}h_1$ & $n\geq 0$, $\epsilon\in\{0,1\}$ \\
$\rho^{1+\epsilon}\tau^{4n+2}h_1$ & $n\geq 0$, $\epsilon\in\{0,1\}$ \\
$u h_0$ \\
$u\tau^{2n+1}h_0$ & $n\geq 0$\\
$u\tau^n h_b$ & $n\geq 0$, $b\in\{1,2\}$ \\
$u\tau^{2n+1}h_3$ & $n\geq 0$ \\
$u\tau^{4n}h_3$ & $n\geq 0$ \\
$\pi \tau^n h_b$ & $n\geq 0$, $b\in\{0,1,2,3\}$ \\
\hline
$u^\epsilon\tau^{2n}h_0$ & $n\geq 1$, $\epsilon\in\{0,1\}$ \\
$u\tau^{4n+2}h_3$ & $n\geq 0$ \\
\hline
$\rho^{1+\epsilon}\tau^{4n}h_4$ & $n\geq 0$, $\epsilon\in\{0,1\}$ \\
$\rho^2\tau^{4n}h_b$ & $n\geq 0$, $b\geq 5$
\end{longtable}
Of these, the classes in the first region are permanent cycles, the classes $u^\epsilon\tau^{2n}h_0$ with $n\geq 1$ and $\epsilon\in\{0,1\}$, as well as $u\tau^{4n+2}h_3$ with $n\geq 0$, support higher differentials, and we leave open the fate of the classes $\rho^{1+\epsilon}\tau^{4n}h_4$ and $\rho^2\tau^{4n}h_b$ for $n\geq 0$, $\epsilon\in\{0,1\}$, and $b\geq 5$.
\end{thm}
\begin{proof}
The given differentials follow by combining \cref{thm:nohopf} with \cref{lem:dq2}. For example, 
\[
d_2(\rho\tau^{2n}h_b) = \rho\tau^{2n}\cdot d_2(h_b) = \rho\tau^{2n}\cdot (h_0+\rho h_1)h_{b-1}^2 = \rho^2 \tau^{2n}h_1 h_{b-1}^2
\]
for $b\geq 4$, which is nonzero precisely when $b\geq 5$; as another example,
\[
d_3(u\tau^{4n+2}h_3) = d_3(u\tau^2)\cdot \tau^{4n}h_3 = \rho^2\tau\cdot\tau^{4n}h_3 = \rho^2\tau^{4n+1}h_3.
\]
We must verify that all $d_2$-differentials are accounted for in this table; the claimed description of the $E_3$-page follows quickly. We must also verify that the classes we give as permanent cycles are indeed permanent cycles. It suffices to verify the latter.

We may cut down the number of classes to consider by taking into account the classes which are products of the permanent cycles given in \cref{lem:dq2} with some other class. After this reduction, degree considerations rule out all differentials except for possibly
\begin{enumerate}
\item $d_r(\tau^{4n+1}h_1) \overset{?}{=}\tau^{4n}h_0^{r+1}$;
\item $d_r(\rho\tau^{4n+3}h_1) \in\f_2\{u,\pi\}\otimes\f_2\{\tau^{4n+2}h_0^{r+1}\}$;
\item $d_r(\rho\tau^{4n+2}h_1) \overset{?}{=} \rho^2\tau^{4n+1}h_0^{r+1}$;
\item $d_r(u \tau^{2n} h_1) \overset{?}{=} \rho^2 \tau^{2n-1}h_0^{r+1}$;
\end{enumerate}
with $n\geq 0$, and in the fourth case $n\geq 1$. In all cases, the possible nonzero targets are present and not boundaries in Ormsby--{\O}stv{\ae}r's computation of the Adams spectral sequence for $\q_2$-motivic $BP\langle 0 \rangle$ \cite{OO13}, so by naturality they cannot be boundaries in the Adams spectral sequence for the sphere. Thus these possible nonzero differentials are in fact not possible, yielding the theorem.
\end{proof}

\subsection{The rational numbers}

We end by considering the case $F = \q$. By naturality, this gives information over arbitrary fields of characteristic zero. Recall the functions $\varepsilon$ and $\lambda$ defined in \cref{def:al}.

\begin{thm}\label{thm:q}
The $1$-line of the $E_3$-page of the $\q$-motivic Adams spectral sequence is given by a direct sum of that for the $\r$-motivic Adams spectral sequence with the classes in the following table, where $p$ ranges through all primes.
\begin{longtable}{ll}
\toprule
Class & Constraints \\
\midrule \endhead
\bottomrule \endfoot
$\tau^n h_b[2]$ & $n\geq 0$, $b\in\{0,1,2,3\}$ \\
$h_0[p]$ & $p\equiv 1\,(4)$ \\
$\tau^n h_1[p]$ & $p\equiv 1\,(4)$, $n\geq 0$ \\
$\tau^{2n}h_2[p]$ & $p\equiv 1\,(4)$, $n\geq 0$ \\
$\tau^{2n+1}h_2[p]$ & $p\equiv 1\,(4)$, $n\geq 0$, $\varepsilon(p)\geq 3$ \\
$\tau^{4n}h_3[p]$ & $p\equiv 1\,(4)$, $n\geq 0$ \\
$\tau^{4n+2}h_3[p]$ & $p\equiv 1\,(4)$, $n\geq 0$, $\varepsilon(p)\geq 3$ \\
$\tau^{2n+1}h_3[p]$ & $p\equiv 1\,(4)$, $n\geq 0$, $\varepsilon(p)\geq 4$ \\
$\tau^n h_b a_p$ & $p\equiv 1\,(4)$, $n\geq 0$, $b\in\{0,1,2,3\}$ \\

$h_0 u_p$ & $p\equiv 3\,(4)$ \\
$\tau^{2n}h_bu_p$ & $p\equiv 3\,(4)$, $n\geq 0$, $b\in\{1,2\}$ \\
$\tau^{4n}h_3u_p$ & $p\equiv 3\,(4)$, $n\geq 0$ \\
$\tau^{4n+2}h_3u_p$ & $p\equiv 3\,(4)$, $n\geq 0$, $\lambda(p)\geq 4$ \\
$\rho\tau^{2n}h_bu_p$ & $p\equiv 3\,(4)$, $n\geq 0$, $b\in\{1,2,3\}$ \\
$\rho\tau^{2n+1}h_b u_p $ & $p\equiv 3\,(4)$, $n\geq 0$, $b\in\{1,2,3,4\}$ \\
$\rho^\epsilon \tau^{4n+1}h_1 u_p$ & $p\equiv 3\,(4)$, $n\geq 0$, $\epsilon\in\{0,1\}$ \\

\hline
$\tau^{2n}h_0[p]$ & $p\equiv 1\,(4)$, $n\geq 1$ \\
$\tau^{2n+1}h_0[p]$ & $p\equiv 1\,(4)$, $n\geq 1$, $\varepsilon(p)\geq 3$ \\

$\tau^{4n+2}h_3[p]$ & $p\equiv 1\,(4)$, $n\geq 0$, $\varepsilon(p) = 2$ \\
$\tau^{2n+1}h_3[p]$ & $p\equiv 1\,(4)$, $n\geq 0$, $\varepsilon(p)=3$ \\

$\tau^{4n+2}h_3 u_p$ & $p\equiv 3\,(4)$, $n\geq 0$, $\lambda(p) = 3$ \\
$\tau^{2n}h_0 u_p$ & $p\equiv 3\,(4)$, $n\geq 1$ \\
\hline
$\rho\tau^{2n}h_b u_p$ & $p\equiv 3\,(4)$, $n\geq 0$, $b\geq 4$ 
\end{longtable}
Moreover, we have the following information about higher differentials. The classes in the first region of this table are permanent cycles, as are the classes $h_a$ and $\tau^{\floor{2^{a-1}(4n+1)}}h_a$ for $a\leq 3$. The classes in the second region of this table support higher differentials, as do the classes in $\Ext_\r^1$ which must support higher differentials by \cref{thm:r}. We leave open the fate of the classes in the third region of this table, as well as the possibility of exotic higher differentials on the classes $\rho h_4$ and $\rho^{2^a-\psi(a)}\tau^{2^{a-1}(4n+1)}h_a$ for $a\geq 4$.
\end{thm}
\begin{proof}
Recall the splitting
\[
\Ext_\q = \Ext_\r \oplus\Ext_{(1)}\{[2]\}\oplus\Ext_{(1)}\{[p],a_p:p\equiv 1 \pmod{4}\}\oplus \Ext_{(2)}\{u_p:p\equiv 3 \pmod{4}\}
\]
implied by \cref{lem:mmodule}. As in the proof of \cref{prop:hasse}, each of these summands is itself either $\Ext_\r$ or an identifiable summand of some corresponding $\Ext_{\q_p}$; for $p$ odd, this summand looks like $\Ext_{\f_p}$. We may thus read the given table off the information given in \cref{thm:r}, \cref{thm:qpf} (with \cref{thm:enf1} and \cref{thm:enf3}), and \cref{thm:q2}, provided we verify the following claim: if $\alpha[p]\in\Ext_\q^1$ is a class in stem $s\leq 6$, then $\alpha[p]$ or $\alpha u_p$ is a $d_r$-cycle if and only if it projects to a $d_r$-cycle in the $\q_p$-motivic Adams spectral sequence; and likewise if $\alpha\in\Ext_\r^1$ is a class in stem $s\leq 7$, then $\alpha$ is a $d_r$-cycle in the $\q$-motivic Adams spectral sequence if and only if it projects to a $d_r$-cycle in the $\r$-motivic Adams spectral sequence. 

As in the proofs of \cref{thm:enf1}, \cref{thm:enf3}, and \cref{thm:q2}, differentials on the classes $\alpha[p]$ and $\alpha u_p$ in stems $s\leq 6$ are completely determined by the structure of differentials on the classes $[p]\tau^{2^i}$ and $u_p\tau^{2^i}$ in the $\q$-motivic Adams spectral sequence for $BP\langle 0 \rangle$, together with the fact that $h_0$, $h_1$, $h_2$, and $h_3$ are permanent cycles. The $\q$-motivic Adams spectral sequence for $BP\langle 0 \rangle$ was computed by Ormsby--{\O}stv{\ae}r in \cite[Theorem 5.8]{OO13}.  We find that differentials on the classes $[p]\tau^{2^i}$ and $u_p\tau^{2^i}$ in the $\q$-motivic Adams spectral sequence for $BP\langle 0 \rangle$ are entirely detected over $\q_p$, and our first claim follows. That the classes $h_a\in\Ext_\r^1$ for $a\leq 3$ are permanent cycles was seen in \cref{lem:lowhopf}, and the classes $\tau^{\floor{2^{a-1}(4n+1)}}h_a\in\Ext_\r^1$ must be permanent cycles for $a\leq 3$ as there is no room for exotic higher differentials.
\end{proof}


\newcommand{\etalchar}[1]{$^{#1}$}

\end{document}